\documentclass[11pt]{article}
\usepackage{amsmath}
\usepackage{amsthm}
\usepackage{amssymb}
\usepackage[authoryear]{natbib}

\makeatletter
\theoremstyle{plain}
\newtheorem{thm}{\protect\theoremname}
\theoremstyle{plain}
\newtheorem{lem}[thm]{\protect\lemmaname}
\theoremstyle{plain}
\newtheorem{assumption}[thm]{\protect\assumptionname}
\newtheorem{proposition}{Proposition}
\newtheorem{corollary}{Corollary}

\@ifundefined{date}{}{\date{}}
\usepackage{amsfonts}\usepackage{amsthm}\usepackage{graphicx}\usepackage{bm}
\usepackage{subfigure}
\usepackage{soul}
\usepackage[font=small,labelfont=bf]{caption}
\usepackage{algpseudocode}
\usepackage{algorithm}

\usepackage{verbatim}
\usepackage{rotating}
\usepackage{caption}
\usepackage{paralist}
\usepackage{booktabs}
\usepackage{hyperref}

\usepackage{epsfig}

\usepackage{bbm}
\usepackage{indentfirst}
\usepackage{float}
\usepackage{authblk}
\theoremstyle{plain}

\providecommand{\keywords}[1]
{
  \small	
  \textbf{\textit{Keywords---}} #1
}

\makeatother

\providecommand{\assumptionname}{Assumption}
\providecommand{\lemmaname}{Lemma}
\providecommand{\theoremname}{Theorem}

\title{Higher-Order Expansion and Bartlett Correctability of Distributionally Robust Optimization}

\author{Shengyi He$^{1}$, Henry Lam$^{1}$  \\
        \small $^{1}$ Department of Industrial Engineering and Operations Research, Columbia University \\
}

\begin{document}
\maketitle
\begin{abstract}{Distributionally robust optimization (DRO) is a worst-case framework for stochastic optimization under uncertainty that has drawn fast-growing studies in recent years. When the underlying probability distribution is unknown and observed from data, DRO suggests to compute the worst-case distribution within a so-called uncertainty set that captures the involved statistical uncertainty. In particular, DRO with uncertainty set constructed as a statistical divergence neighborhood ball has been shown to provide a tool for constructing valid confidence intervals for nonparametric functionals, and bears a duality with the empirical likelihood (EL). In this paper, we show how adjusting the ball size of such type of DRO can reduce higher-order coverage errors similar to the Bartlett correction. Our correction, which applies to general von Mises differentiable functionals, is more general than the existing EL literature that only focuses on smooth function models or $M$-estimation. Moreover, we demonstrate a higher-order ``self-normalizing'' property of DRO regardless of the choice of divergence. Our approach builds on the development of a higher-order expansion of DRO, which is obtained through an asymptotic analysis on a fixed point equation arising from the Karush-Kuhn-Tucker conditions.}
\end{abstract}


\section{Introduction}

We consider the basic problem of confidence interval (CI) estimation for a general statistical functional $\psi(P_0)$, where $P_0$ is an unknown distribution and $\psi$ maps from a probability distribution to a real value. In particular, we are interested in constructing CIs with coverage errors of smaller orders than nominally applying the central limit theorem or the delta method. This classical problem has been addressed via multiple approaches, such as iterated resampling when applying the bootstrap (e.g., \cite{davison1997bootstrap} \S2.4; \cite{efron1982jackknife} \S10.4; \cite{Hall1992} \S3.11), and Bartlett correction when applying the empirical likelihood (EL) (\citealt{owen2001empirical} \S13; \citealt{diciccio1991}). 


In this paper, we study an alternative approach called \emph{distributionally robust optimization (DRO)} to construct high-accuracy CIs. DRO is a framework to handle stochastic optimization when the true underlying probability distribution is unknown or uncertain. It originates from operations research \citep{scarf1957min} and stochastic control \citep{petersen2000minimax}, then gains substantial applications across disciplines ranging from economics \citep{hansen2008robustness} and finance \citep{ghaoui2003worst,glasserman2014robust} to queueing \citep{jain2010optimality} and revenue management \citep{lim2006model}, and recently exhibits a surge of interest in machine learning due to its risk-averse interpretation \citep{ruszczynski2006optimization,rockafellar2007coherent} and connections to regularization \citep{kuhn2019wasserstein,rahimian2019distributionally}. DRO advocates a worst-case approach to handle optimization under uncertainty, where the uncertain parameter is postulated to lie in a feasible region known as the uncertainty set or ambiguity set that, intuitively speaking, contains the ground truth with high likelihood. It thus often leads to a minimax formulation where the inner optimization computes the worst-case scenario on the parameter value in the uncertainty set. This framework can be viewed as an extension of the classical paradigm of robust optimization \citep{bertsimas2011theory,ben2009robust,xu2009robustness}, where the uncertain parameter in DRO is the underlying probability distribution in a stochastic problem. Throughout this paper, we use the term DRO broadly to refer to a worst-case optimization over any unknown distribution in a feasible region.


We consider especially DRO whose uncertainty set is a neighborhood ball surrounding some baseline distribution, where the ball size is measured by a statistical $\phi$-divergence. This formulation, known as divergence-based DRO \citep{BenTal2013}, can be written as
\begin{equation}
\begin{array}{ll}
\max/\min_{P} & \psi(P)\\
\text{subject to} & D_{\phi}(P||\hat{P})\leq\eta
\end{array}\label{eq: DRO_formulation}
\end{equation}
where $D_{\phi}(P||Q)=\int\phi\left(\frac{dP}{dQ}\right)dQ$ for a suitable function $\phi:\mathbb R_+\to\mathbb R$, and $\hat{P}$ is the empirical distribution given by $\hat{P}:=\frac{1}{n}\sum_i\delta_{X_i}$ with $\delta_\cdot$ denoting the delta measure on $\cdot$. Here, $\max/\min$ means we consider a pair of worst-case optimization problems, a max and a min, to obtain an upper and lower bound for $\psi(P)$. These bounds can be readily seen to cover $\psi(\hat P)$ by the definition of the optimization and, at least intuitively, they get increasingly closer to $\psi(\hat P)$ as the ball size $\eta$ shrinks to 0, thus hinting a resemblance with the normality CI when $\eta$ is suitably chosen. Indeed, the DRO \eqref{eq: DRO_formulation} has been recently shown to generate asymptotically exact CIs for $\psi(P_0)$ when $\eta$ is well-chosen in terms of the sample size $n$, the $\phi$-function, and the quantile of a $\chi^2$-distribution \citep{BenTal2013,LAM2017301,duchi2016statistics}. Compared to other established statistical approaches to construct CIs, DRO is motivated from its advantageous preservation of convexity. Namely, when $\psi(\cdot)$ depends convexly on a decision variable to be optimized (in addition to the input distribution $P$) and serves as a constraint or objective in a bigger optimization problem, the resulting minimax-type formulation (with outer optimization on the decision variable and inner optimization on $P$) still remains convex and hence computationally tractable \citep{duchi2019variance,lam2019recovering}.

Our main goal in this paper is to study the use of DRO \eqref{eq: DRO_formulation} to obtain CIs with higher-order coverage improvements. Let $\psi_{\min}(n,\eta)$ and $\psi_{\max}(n,\eta)$ denote the optimal values of \eqref{eq: DRO_formulation}. We aim to find a choice of $\eta$ such that 
\begin{equation}\label{eq: small_error}
    P(\psi(P_0)\in[\psi_{\min}(n,\eta),\psi_{\max}(n,\eta)]) = 1-\alpha + O(n^{-3/2})
    \end{equation}
where the error term $O(n^{-3/2})$ is of higher order than $O(n^{-1})$ that is elicited by the standard (nonparametric) delta method. Compared to iterated bootstraps, this approach does not require a multiplicative amount of Monte Carlo runs that can be computationally demanding. It is closely related to the Bartlett correction in EL, the latter based on a simple adjustment on the involved statistic to make its asymptotic distribution closer to $\chi^2$. In \eqref{eq: small_error}, we analytically adjust the ball size in a similar spirit as the Bartlett correction, but our approach is more general in that it applies to general von Mises differentiable functionals, beyond the smooth function model $\psi(P)=f(E_P[Z])$ for some function $f$ and variable $Z$, or $M$-estimation where $\psi(P)$ is defined through an estimating equation $E_P[m(Z,\psi)]=0$, which are used in all the EL literature to our best knowledge. 

To this end, we briefly explain our main mathematical machinery and how this offers a different, more general, analysis route to obtain correction than EL. The ball size adjustment to achieve \eqref{eq: small_error} is obtained via a high-order asymptotic expansion on the optimal values of DRO, $\psi_{\min}(n,\eta)$ and $\psi_{\max}(n,\eta)$, in terms of $\eta$, against which an Edgeworth expansion is compared to calibrate $\eta$. Through the Karush-Kuhn-Tucker (KKT) conditions, the decision variable (i.e., $P$) and the Lagrange multipliers in the constrained problems \eqref{eq: DRO_formulation} can be expressed in terms of a fixed point equation of a suitable iteration scheme whose asymptotic can be analyzed. In contrast, existing Bartlett correction derivations for EL are all based on mean estimation, for which the optimal distribution (or probability weights) in attaining the profile likelihood can be written in closed form as a function of the Lagrange multipliers. For instance, in the smooth function model, a transformation argument is used to translate the correction formula for the mean to a function of mean (e.g., in the seminal work \cite{diciccio1988}, Section 3, steps 1-3 consider EL for the mean and, based on this, steps 4-10 consider Bartlett correction for the smooth function model). Likewise, in $M$-estimation, since the estimating equation is linear in $P$, the optimal solution can also be solved in closed form like in mean estimation \citep{zhang1996accuracy}. However, when $\psi(P)$ is a general von Mises functional, such a route is no longer clear. In this regard, the DRO analysis appears much more direct than a potential analysis based on EL, since in this route the constraint in \eqref{eq: DRO_formulation} ensures that each feasible solution must be close to $\hat{P}$, which allows one to approximate $\psi(P)$ in terms of influence functions for any feasible $P$. 

In addition to developing a new analysis route to obain higher-order coverage improvements, we also present several implications from our DRO investigation. First is what we call a \emph{higher-order self-normalizing} property of DRO: To reduce the coverage error of a CI to $O(n^{-3/2})$ using DRO \eqref{eq: DRO_formulation}, we only need to estimate the influence function of $\psi$ up to the second order. On the other hand, if we attempt to correct for the delta method directly, then the same order of reduction would require the estimation of the third-order influence function. This can be viewed as a higher-order counterpart of the self-normalizing behavior of EL \citep{diciccio1991} or DRO \citep{duchi2016statistics} where no (estimated) information of $\psi$ is needed to generate asymptotically exact CIs, while the delta method requires estimating the first-order influence function. Our second implication is the criterion for being Bartlett correctable. Based on our derived formulas, we generalize the observations made in \cite{CORCORAN}, which applies only to the smooth function model,
that a necessary and sufficient condition for being Bartlett correctable (in the sense of \citealt{diciccio1991}) for general functionals is
\begin{equation}\label{eq: Bartlett_correctability}
\phi^{\prime\prime\prime}(1)=-2\phi^{\prime\prime}(1),\phi^{(4)}(1)=-3\phi^{\prime\prime\prime}(1).
\end{equation} 
From \cite{diciccio1991}, we see that when a model is Bartlett correctable, there exists a simple correction to the EL statistic that makes its asymptotic distribution closer to $\chi^2$. However, for constructing CIs, even if a model is not Bartlett correctable, we can still achieve \eqref{eq: small_error} by allowing the correction mechanism to depend on the nominal confidence level. As our third implication, we derive our general coverage improvement scheme in \eqref{eq: small_error} taking this latter viewpoint, noting that the scheme can be simplified when \eqref{eq: Bartlett_correctability} holds. As a comparison, if one takes the established view in EL and insists on making an EL statistic to be closer to $\chi^2$ when \eqref{eq: Bartlett_correctability} is violated, the needed correction seems to be substantially more intricate. 

Finally, we stipulate that the presented technical advantages of DRO over EL are foreseen to continue to hold for more general CI construction problems, such as in the presence of auxiliary conditions like $m(P)=m_0$ for some functional $m$ and value $m_0$. Analyzing the higher-order asymptotics of such a DRO could be handled by ``projecting'' the expansion coefficients on the auxiliary constraints, while EL remains challenging to proceed. On the other hand, a current limitation of the DRO formulation \eqref{eq: DRO_formulation} is that it is single-dimensional, whereas EL can generate multi-dimensional confidence regions. We leave the investigation of auxiliary information incorporation and multi-dimensional generalization to the future, but provide some hints to the latter extension in Section \ref{sec:multi}.

The rest of this paper is organized as follows. Section \ref{sec: literature} reviews related literature. Section \ref{sec: DRO expansion} presents the expansions for the optimal values of DRO. Section \ref{sec: Bartlett} presents our correction formula for coverage error reduction. Section \ref{sec: self normalizing} discusses the higher-order self-normalizing property of DRO. Section \ref{sec: comparison with EL} compares DRO with EL. Section \ref{sec: numerical} applies the derived formula to numerical examples. Section \ref{sec:multi} discusses a potential way to generalize our results to multidimensional settings. Sections \ref{sec: DRO_technical} and \ref{subsec: Bartlett_technical} provide the technical developments for the derivations of DRO expansion and Bartlett correction formula respectively. Technical proofs, some lengthy computation steps and additional numerical results are provided in the Appendix.



\section{Literature Review}\label{sec: literature}

Our work is related to the literatures on DRO, EL and Bartlett correction. In the following, we discuss our connection to these literatures.

\subsection{Distributionally Robust Optimization} 
As discussed in the introduction, DRO originates as a methodology to handle optimization under uncertainty, and can be viewed as a special class of robust optimization (RO) \citep{bertsimas2011theory,ben2009robust} where the unknown or uncertain parameter is the underlying probability distribution. This uncertainty can come from an intrinsic lack of information about the parameter, or due to the statistical noise of data, the latter of which is known as data-driven DRO \citep{delage2010}. The main idea of (D)RO is to use a so-called uncertainty set or ambiguity set to capture information of these parameters, imposed as a constraint in a worst-case optimization. As a basic statistical guarantee of data-driven DRO, it can be shown that, when the uncertainty set is a confidence region on the true probability distribution, the optimal value of the DRO provides bounds on the target objective, with a confidence at least at the same level as the uncertainty set in covering the true distribution \citep{delage2010,bertsimas2018robust,Esfahani2018Data}. 

Most uncertainty sets in DRO belong to two major classes. One is a neighborhood ball surrounding a baseline distribution, where the ball size is measured by statistical distance. Commonly used distances include the $\phi$-divergence considered in this paper \citep{Glasserman2014,gupta2019near,bayraksan2015data,Iyengar2005,hu2013kullback,blanchet2020extreme,atar2015robust} and the Wasserstein distance \citep{Esfahani2018Data,blanchet2021sample,gao2016distributionally,xie2019tractable}. This class of uncertainty sets has the advantage of being statistically ``consistent'', in the sense that when data size increases, the constructed set, when suitably calibrated to bear a confidence guarantee, shrinks to a singleton of the true distribution. DROs with these uncertainty sets are closely related to regularized empirical optimization which, in the context of statistical learning, helps reduce overfitting and improve out-of-sample performances. In particular, Wasserstein DRO recovers a range of regularization formulations in regression and classification \citep{shafieezadeh2019regularization,Blanchet2019,gao2020wasserstein,chen2018robust}. Divergence DRO provides variance regularization effects on empirical optimization \citep{duchi2019variance}, and also improves a bias-variance trade-off on the attained cost when the variance of future cost is considered \citep{gotoh2018robust}. The second major class of uncertainty sets uses moment and support information  \citep{Goh2010,delage2010,agrawal2012price,Wiesemann2014}, marginal constraints \citep{doan2015robustness,dhara2021worst}, or distributional shape such as convexity and unimodality \citep{popescu2005semidefinite,van2016generalized,chen2021discrete}. This class of uncertainty sets may not be consistent, as they ``summarize'' information in low dimensions. They are motivated from the provision of robust bounds in situations where data are intrinsically lacking, e.g., extremal estimation \citep{lam2017tail}, or in non-stationary environments with large intrinsic uncertainties, e.g., in power system optimization \citep{zhang2016distributionally,zhao2017distributionally}.


In this paper, we consider DRO based on $\phi$-divergence, which belongs to the first class above. We are particularly interested in using it as a machinery to construct high-accuracy CIs. In this regard, \cite{BenTal2013} considers the tractabiliy of this DRO and the calibration of the ball size so that the uncertainty set covers the true distribution with a prescribed confidence, but the resulting bound generated from such calibrated DRO could be loose. \cite{lam2019recovering} develops asymptotically exact confidence bounds via DRO for expected values parametrized by decision variables for use in constrained optimization. \cite{LamZhou2015,LAM2017301} use DRO to construct CIs for the optimal values of stochastic optimization problems. \cite{duchi2016statistics} further develops asymptotically exact CIs for general smooth functionals via DRO. Except \cite{BenTal2013}, all the above works connect with the classical EL theory which we will describe in the next subsection. Compared to all the above works, our study focuses on higher-order corrections to the CIs, and in generality that is beyond the EL literature.

Our main mathematical developments are based on asymptotic expansions of the optimal values of DRO. Such type of expansions has been studied mostly in the context of sensitivity analysis \citep{dupuis2015pathspace,Lam2016robust}.
In particular, 
\cite{gotoh2018robust,gotoh2021calibration} and \cite{Lam2016robust,Lam2018sensitivity} study higher-order expansions of DRO. \cite{gotoh2018robust,gotoh2021calibration} approximate the mean and variance of an underlying reward by deriving higher-order expansions. Because of their focus, their expansions apply only to expectations, i.e., linear functions in $P$. Moreover, they consider a penalty formulation, instead of a nominally defined DRO, in which the constraints do not need to be explicitly handled in the asymptotic. As such their analysis is considerably simpler than our setting. 
\cite{Lam2016robust,Lam2018sensitivity}, like our work, derive higher-order expansions for DRO when the ball size shrinks to 0, but our expansions are different and more general in the sense that our baseline is an empirical distribution, which is random instead of a fixed baseline, and that we allow $\phi$ to be a general convex function instead of Kullback-Leibler (KL) or $\chi^2$-divergences considered in \cite{Lam2016robust,Lam2018sensitivity}.

We close this subsection by noting that DRO relates intimately with risk-averseness. It is widely known that any coherent risk measure, such as the conditional value-at-risk, can be expressed as a DRO \citep{shapiro2014lectures}. This view has been utilized recently in machine learning to improve model training on minority subpopulations and to incite fairness under a lack of data diversity \citep{duchi2020learning,duchi2020distributionally,pmlr-v80-hashimoto18a}. Similar ideas are used in deep learning, in particular adversarial training, in which models are trained using adversarially perturbed input data \citep{GoodfellowSS14,madry2018towards,sinha2018certifiable}. Lastly, DRO also helps improve model performances under distributional shifts in transfer learning \citep{Sagawa2020Distributionally,liu2021stable}. 

\subsection{Empirical Likelihood and Bartlett Correction}
EL, first proposed by \cite{owen1988empirical}, can be viewed as a nonparametric analog to the maximum likelihood theory. In particular, the empirical distribution plays the role of a nonparametric maximum likelihood estimate, and inference is conducted based on the induced likelihood ratio and an analog of Wilks' theorem. EL has since then been vastly studied across statistical problems which include, to name a few, regression \citep{owen1991empirical}, kernel density estimation \citep{hall1993empirical,chen2000empirical}, biased or censored data \citep{LI199595,murphy1997semiparametric} and dependent data \citep{mykland1995dual} (for more applications, see the monograph \cite{owen2001empirical}). More recently, \cite{hjort2009} considers extensions of EL to allow for nuisance parameters with slower than $\sqrt{n}$-rate of convergence, and the number of estimating equations to grow with the sample size $n$. Regarding the relation between DRO and EL, in addition to the works described in the last subsection,
a notable development is brought by \cite{blanchet2021sample,Blanchet2019,blanchet2019confidence} who propose the Wasserstein profile function, which is a Wasserstein analog to the profile likelihood in classical EL that uses the KL divergence, and develop its asymptotic theory and regularization implications.

\sloppy The notion of Bartlett correction first appeared in the parametric likelihood ratio test \citep{Bartlett1937properties,lawley1956general}. 
In the nonparametric setting, \cite{diciccio1991} first showed that EL with reverse KL divergence is Bartlett correctable, while the bootstrap is not. Bartlett correctability for EL using other choices of divergence is studied by several papers, including \citet{Baggerly1998}
and \citet{CAMPONOVO201438} who consider the Cressie--Read
power divergence family, and \citet{CORCORAN} who considers general $\phi$-divergences. These papers all consider the smooth function model. In addition, \cite{zhang1996accuracy} extends the Bartlett correctability of EL to $M$-estimation using techniques similar to the smooth function model, and \cite{chen1993smoothed} studies the Bartlett correctability of quantile estimation, with a smoothed EL that is different from the standard EL approach considered in this paper. 

The development of Bartlett correction employs Edgeworth expansion, which we discuss briefly. \citet{bhattacharya1978} shows the validity of Edgeworth expansions for smooth function models. \cite{Callaert1980edgeworth} studies Edgeworth expansions for $U$-statistics with residual $o(n^{-1})$. \cite{withers1983expansions} provides a procedure to compute formal Edgeworth expansions for von Mises functionals, but does not show their validity. \cite{Bhattacharya1983} shows a potential approach to obtain valid Edgeworth expansions for the characteristic function, which could be used to derive expansion for the distribution. However, as they explicitly mentioned, their imposed condition is hard to verify except for the smooth function model.  \citet{TAKAHASHI198856} gives a low-order
Edgeworth expansion with residual $o(n^{-1/2})$ for von Mises functionals. 

\subsection{Other Higher-Order Coverage Correction Methods}
Finally, we review other works on higher-order coverage correction, mostly based on the bootstrap. \cite{davison1997bootstrap} \S2.4 presents
iterated bootstrap studentization. \cite{hall1988bootstrap} and \cite{Hall1992} \S3.11 study iterated bootstrap for statistical problems formulated as the root of a population equation. \cite{hall1986bootstrap} shows the connection between iterated bootstrap and Edgeworth expansion. \cite{beran1987prepivoting} studies prepivoting that can be iterated to obtain higher-order corrections. \cite{efron1982jackknife} \S10.7,  \cite{efron1987better} and \cite{diciccio1996bootstrap} use the bootstrap to automate tricks based on transformations and asymptotic arguments, to obtain bias-corrected and accelerated CIs.



\section{Asymptotic Expansion for DRO}\label{sec: DRO expansion}
Suppose we have observed $X_1,X_2,\dots,X_n$ that are i.i.d. generated from $P_0$. Our goal is to use the maximum and minimum values obtained from \eqref{eq: DRO_formulation} to construct a CI for $\psi(P_0)$ that has a small coverage error, by suitably choosing $\eta$. To do so, we first need to develop a higher-order stochastic expansion for these optimal values. 

We first rewrite \eqref{eq: DRO_formulation} by converting the decision variable $P$ to likelihood ratio $L:=\frac{dP}{d\hat{P}}$, which is well-defined for $P$ that is absolutely continuous with respect to $\hat P$, i.e., with the same support as the data. We also let $\eta$ shrink at the speed of $n^{-1}$, the reason of which will be apparent in a moment. More specifically, we rewrite \eqref{eq: DRO_formulation} as
\begin{equation}\label{eq: DRO_L_formlation}
\begin{array}{ll}
\max/\min_{L} & \psi(\hat{P}L)\\
\text{subject to} & \hat{E}[\phi(L)]\leq\frac{q}{2n}\\
& \hat{E}L=1
\end{array}
\end{equation}
for some $q>0$, and $\hat E$ denotes the expectation with respect to the empirical distribution $\hat P$.

Recall that $\hat{P}=\frac{1}{n}\sum_{i=1}^n\delta_{X_i}$. Note that since $\hat{P}$ is supported on $\{X_1,\dots,X_n\}$, both $L$ and $\hat{P}L$ can be encoded as $n$-dimensional vectors. With this observation, we will sometimes abuse notation (with clear notice) to regard $L = \left(\frac{dP}{d\hat{P}}(X_1),\frac{dP}{d\hat{P}}(X_2),\dots,\frac{dP}{d\hat{P}}(X_n)\right)$ and $\hat{P}L=L/n$. Thus, \eqref{eq: DRO_L_formlation} can be seen as a pair of optimization problems with objective $\psi(\hat{P}L)$ when $L$ varies in a small region around the vector $\mathbf{1}_n := (1,1,\dots,1)^T$. 

We assume that $\psi$ admits an expansion in terms of influence functions of high enough orders (Chapter 6 of \citet{serfling1980}), uniformly over the domain of $L$. More precisely, we make the following assumption. 
\begin{assumption}
There exist functions $IF_j,j=1,2,3$ and $R_n = O_p (n^{-2})$ such that:
\begin{enumerate}
\item
\[
\left|\psi(\hat{P}L)  -\psi(\hat{P})-\sum_{j=1}^3\frac{1}{j!} E_{\Delta}IF_{j}(K_{1},K_{2},\dots,K_{j};\hat{P})\right|\leq R_n
\]
holds for all $L$ such that
$\hat{E}[\phi(L)]\leq\frac{q}{2n}$. Here, $E_{\Delta}$ denotes the expectation under which $K_1,K_2,\dots\stackrel{\text{i.i.d.}}{\sim} \hat{P}L-\hat{P}$ (i.e., a product signed measure). \\
\item For any $1\leq m\leq5$, $\hat{E}IF_{1}^{m}(K_1;\hat{P})=O_{p}(1)$,
$\hat{E}IF_{2}^{m}(K_{1},K_{2};\hat{P})=O_{p}(1),\hat{E}IF_{3}^{m}(K_{1},K_{2},K_{3};\hat{P})=O_{p}(1)$. Moreover, $\frac{1}{\widehat{Var}IF_{1}(K_1;\hat{P})}=O_{p}(1)$. Here, under $\hat{E}$ and $\widehat{Var}$ we have $K_1,K_2,\dots\stackrel{\text{i.i.d.}}{\sim} \hat{P}$. 
\item For $i=1,2,3$, $IF_{i}(K_{1},\dots,K_{i};\hat{P})$ is permutation invariant in $K_{1},\dots,K_{i}$, and has zero marginal expectations under $\hat{P}$:
\[
E_{K_{j}\sim \hat{P}}IF_{i}(K_{1},\dots,K_{i};\hat{P})=0,\ j=1,2,\dots,i.
\]
Here, $E_{K_j\sim\hat{P}}$ stands for the expectation under $K_j\sim\hat{P}$ with $K_s,s\neq j$ fixed.
\end{enumerate}
\label{assu: psi_expansion}
\end{assumption}
In Assumption \ref{assu: psi_expansion}, part 1 specifies the form of expansion for $\psi(\hat{P}L)$ and the requirement on the residual. In part 2, the first three conditions guarantee that the influence functions are stochastically bounded, and the condition $\frac{1}{\widehat{Var}IF_{1}}=O_{p}(1)$ ensures that the limiting distribution of $\sqrt{n}(\psi(\hat{P}L)-\psi(\hat{P}))$ is not degenerated. Part 3 requires that the influence functions are written in a canonical form, since there could exist many versions of influence functions that satisfy the first two parts (e.g., changing the influence functions by a constant will not affect parts 1-2). Note that Part 3 is not restrictive since for any influence functions that satisfy the first two parts, we may: 1) replace them with the average of permutations to make them permutation invariant and then 2) replace $IF_{i}(X_1,\dots,X_i;\hat{P}),i=1,2,3$ with (similar to Lemma A of \cite{serfling1980} \S 6.3.2)
\[
\left[\prod_{j=1}^{i}\left(1-E_{K_{j}\sim \hat{P}}\right)\right]IF_{i}(K_{1},\dots,K_{i};\hat{P}).
\]
After this operation, the condition in part 3 is satisfied, but $E_{\Delta} IF_{i}(X_{1},X_{2},\dots,X_{i};\hat{P})$
will not change so that parts 1-2 still hold.
We will provide some examples where Assumption \ref{assu: psi_expansion} holds in Section \ref{subsec: Verification-of-Assumption}. 





Our next assumption is on the smoothness and convexity
of the divergence function $\phi$:
\begin{assumption}\label{assu: phi_smoothness}
$\phi$ is convex, $\phi(1)=\phi^{\prime}(1)=0,\phi^{\prime\prime}(1)>0$
and has continuous fifth-order derivative on $(0,\infty)$. Moreover, there exists $\delta>0$ such that $\phi^{\prime\prime}(x)>\delta$ for $x\in [c_{0},C_{0}]$ where $[c_{0},C_{0}]$ is the range where $\phi(x)\leq q/2$.
\end{assumption}

Assumption \ref{assu: phi_smoothness} is satisfied for most common choices of $\phi$ \citep{pardo2018statistical}. Examples include $\chi^2$ divergence where $\phi(x) = (x-1)^2$ and KL divergence where $\phi(x) = x\log x - x + 1$.




With the above assumptions, we have the following expansion on the DRO \eqref{eq: DRO_L_formlation}:
\begin{thm}
\label{thm: DRO_expansion}Under Assumptions \ref{assu: psi_expansion} and \ref{assu: phi_smoothness}, the maximum value of \eqref{eq: DRO_L_formlation} admits the following stochastic expansion:
\begin{align*}
 & \psi(\hat{P})+n^{-1/2}\sqrt{\frac{q\hat{\kappa}_{2}}{\phi^{\prime\prime}(1)}}+
 \frac{n^{-1}q}{\phi^{\prime\prime}(1)\hat{\kappa}_2}\left(-\frac{1}{6}\frac{\phi^{\prime\prime\prime}(1)\hat{\gamma}}{\phi^{\prime\prime}(1)}+\frac{1}{2}\hat{\mu}_{2,c}\right)+\frac{n^{-3/2}q^{3/2}}{\phi^{\prime\prime}(1)^{3/2}\hat{\kappa}_{2}^{3/2}}\left[\frac{1}{6}\hat{\mu}_{3,c}-\frac{1}{2}\frac{\phi^{\prime\prime\prime}(1)}{\phi^{\prime\prime}(1)}\hat{\mu}_{2,b}\right.\\
 & \left.+\frac{\phi^{\prime\prime\prime}(1)}{3\hat{\kappa}_{2}\phi^{\prime\prime}(1)}\hat{\gamma}\hat{\mu}_{2,c}+\left(\frac{\phi^{\prime\prime\prime}(1)^{2}}{8\phi^{\prime\prime}(1)^{2}}-\frac{1}{24}\frac{\phi^{(4)}(1)}{\phi^{\prime\prime}(1)}\right)\hat{\mu}_{4}+\frac{1}{2}\hat{\mu}_{2,a}-\frac{\phi^{\prime\prime\prime}(1)^{2}}{18\phi^{\prime\prime}(1)^{2}\hat{\kappa}_{2}}\hat{\gamma}^{2}-\frac{\phi^{\prime\prime\prime}(1)^{2}}{8\phi^{\prime\prime}(1)^{2}}\hat{\kappa}_{2}^{2}-\frac{1}{2\hat{\kappa}_{2}}\hat{\mu}_{2,c}^{2}\right]+O_{p}(n^{-2})
\end{align*}

Here, 
\begin{gather*}
    \begin{array}{c}
         \hat{\kappa}_{2}:=\hat{E}IF_{1}^{2}(X;\hat{P}),\\
         \hat{\gamma}:=\hat{E}IF_{1}^{3}(X;\hat{P}),\\
         \hat{\mu}_{2,a}:=\hat{E}IF_{1}(X;\hat{P})IF_{1}(Z;\hat{P})IF_{2}(X,Y;\hat{P})IF_{2}(Y,Z;\hat{P}),\\
         \hat{\mu}_{2,b}:=\hat{E}IF_{1}(X;\hat{P})IF_{1}(Y;\hat{P})^{2}IF_{2}(X,Y;\hat{P}),\\
         \hat{\mu}_{2,c}:=\hat{E}IF_{1}(X;\hat{P})IF_{1}(Y;\hat{P})IF_{2}(X,Y,\hat{P}),\\
         \hat{\mu}_{3,c}:=\hat{E}IF_{1}(X;\hat{P})IF_{1}(Y;\hat{P})IF_{1}(Z;\hat{P})IF_{3}(X,Y,Z;\hat{P}),\\
         \hat{\mu}_{4}:=\hat{E}IF_{1}^{4}(X;\hat{P})
    \end{array}
\end{gather*}
In the above definitions, $X,Y,Z\stackrel{\text{i.i.d.}}{\sim}\hat{P}$ and $\hat E$ denotes the corresponding expectation under this product measure. The minimum value of \eqref{eq: DRO_L_formlation} has the same expansion, but with a negation on the coefficients for $n^{-1/2}$ and $n^{-3/2}$.
\end{thm}

Theorem \ref{thm: DRO_expansion} is obtained by using the KKT conditions to express the decision variable $L$ and Lagrange multipliers in terms of the fixed point equation of a suitable iteration scheme. From this, we analyze the asymptotic of the fixed point. The detailed proof is given in Section \ref{sec: DRO_technical}. 

\subsection{Examples of \texorpdfstring{$\psi(P)$}{psiP}}\label{subsec: Verification-of-Assumption}
We provide several examples where an explicit expansion of $\psi(P)$ is available and Assumption \ref{assu: psi_expansion} could be readily verified. 

\paragraph*{Example 1: Smooth function models}

We consider the model in \citet{diciccio1991} where $\psi(P)=f(E_{P}Z)$ and $Z$ is a random vector written as $Z=(Z_{1},Z_{2},...,Z_{q})$. For this model, we can verify Assumption \ref{assu: psi_expansion} by Taylor expanding $\psi$. We have the following:

\begin{proposition}\label{prop: function_of_mean}
Suppose that 
$f$ has continuous derivatives up to the fourth order
in a neighborhood of $E_{0}Z$ and $Z$ has finite moments up to the
$15$-th order, and $\text{Var}_0 (\nabla f(E_0Z)^{\top}\cdot Z)>0$. Then Assumption \ref{assu: psi_expansion} holds with influence functions given by 
\begin{equation}\label{eq: IF_function_mean}
IF_{k}(Z^{(1)},...,Z^{(k)};\hat{P})=\sum_{1\leq i_{1}...i_{k}\leq q}f_{i_{1}i_{2}..i_{k}}(\hat{E}Z)(Z_{i_{1}}^{(1)}-\hat{E}Z_{i_{1}})...(Z_{i_{k}}^{(k)}-\hat{E}Z_{i_{k}}).
\end{equation}
Here $f_{i_{1}i_{2}..i_{k}}=\frac{\partial^k f}{\partial z_1\dots\partial z_q} $, $E_0$ and $\text{Var}_0$ denote the expectation and variance under $P_0$.
\end{proposition}

\paragraph*{Example 2: $V$-statistics}
Consider a $V$-statistic $\psi(P)=E_{X_i\stackrel{\text{i.i.d.}}{\sim}P}h(X_1,X_2,\dots,X_T)$. Without loss of generality (WLOG), suppose that the kernel function $h$ is invariant under permutation (otherwise we can consider averaging the permutation sum). To verify Assumption \ref{assu: psi_expansion}, we can use the separability assumption in \citet{Lam2016robust}. Denote $X_{1:T}:=(X_1,\dots,X_T)$. We have the following:
\begin{proposition}\label{prop: V_stat}
Assume that $|h(X_{1:T})|\leq\Lambda_{1}(X_{1})+\Lambda_{2}(X_{2})+\cdots+\Lambda_{T}(X_{T})$
and $E_{0}e^{\theta\Lambda_{t}(X)}<\infty$ for each $t=1,2,\cdots,T$
and some $\theta\neq 0$. Moreover, assume that $\text{Var}_0 g(X)>0$ where $g(x)=E_0 [h(X_{1:T})|X_1=x]$. Then parts 1-2 of Assumption \ref{assu: psi_expansion} hold with influence functions given by 
\[
IF_{k}(x_{1},x_{2},\dots,x_{k};\hat{P})=\sum_{i_{1},i_{2},\dots,i_{k}\text{mutually different}}E_{\hat{P}}\left[h(X_{1},\dots,X_{T})|X_{i_{1}}=x_{1},\dots,X_{i_{k}}=x_{k}\right].
\]
\end{proposition}
As explained after Assumption \ref{assu: psi_expansion}, the influence functions given in Proposition \ref{prop: V_stat} could be easily converted to satisfy part 3 of Assumption \ref{assu: psi_expansion}.

\paragraph*{Example 3: Optimization}
We consider $\psi(P)=\inf_{x}E_{P}[\ell(x;\xi)]$ as in \citet{duchi2016statistics}. Under smoothness assumptions on $\ell$, by a uniform convergence and Taylor expansion argument, we can show that
\begin{proposition}\label{prop: optimization}
Suppose the function class $\{\ell(x,\cdot):x\in\mathcal{X}\} $ is Glivenko-Cantelli with respect to $P_0$. Also suppose that under $P_0$, $E_{P_0}\ell(x,\xi) $ has a well seperated unique minimizer $x_{P_0}^* $ in the sense that for any $\delta_1>0$, there exists $\epsilon>0$ such that $\{x: E_{P_0}\ell(x,\xi)\leq E_{P_0}\ell(x_{P_0}^*,\xi) +\epsilon \}\subset \{x: \left\Vert x-x_{P_0}^*\right\Vert\leq \delta_1\} $. Suppose that $x_{P_0}^*$ is in the interior of $\mathcal{X}$, $\ell(x,\xi)$ is fourth-order continuously differentiable w.r.t. $x$ and there exists measurable functions $h(\xi)$
and $H(\xi)$ and $\delta>0$ such that
\[
h(\xi)\leq\inf_{\left\Vert x-x_{P_0}^{*}\right\Vert\leq\delta}\ell_{xx}(x,\xi)
\]

\[
H(\xi) \geq \sup_{\left\Vert x-x_{P_0}^{*}\right\Vert\leq\delta} \ell^2(x,\xi)
\]
\[
H(\xi)\geq\sup_{\left\Vert x-x_{P_0}^{*}\right\Vert\leq\delta,j=1,2,3,4}\left(\frac{\partial^j}{\partial x^j}\ell(x,\xi)\right)^2
\]
and 
\[
Eh(\xi)>0,EH(\xi)<\infty,\text{Var}_0 \ell(x^*_{P_0},\xi)>0.
\]
Then Assumption \ref{assu: psi_expansion} holds with the first two influence functions given by 
\[
IF_{1}(\xi;\hat{P})=\ell(x_{\hat{P}}^{*},\xi)-E_{\hat{P}}\ell(x_{\hat{P}}^{*},\xi)
\]
and 
\[
IF_{2}(\xi_{1},\xi_{2};\hat{P})=-\frac{2\ell_{x}(x_{\hat{P}}^{*},\xi_{1})\ell_{x}(x_{\hat{P}}^{*},\xi_{2})}{E_{\hat{P}}\ell_{xx}(x_{\hat{P}}^{*},\xi)}.
\]
\end{proposition}

The proofs of Propositions \ref{prop: function_of_mean}, \ref{prop: V_stat} and \ref{prop: optimization} are presented in Appendix \ref{sec:verification proofs}.



%

\section{Coverage Error Reduction}\label{sec: Bartlett}
We present our main results on coverage error reduction and provide a roadmap of our argument, leaving the technical details to Section \ref{subsec: Bartlett_technical}. For simplicity, we write the expansion for the maximum value in Theorem \ref{thm: DRO_expansion} as
\[
\psi_{max}=\psi(\hat{P})+\sum_{k=1}^{3}n^{-k/2}\left(\frac{q}{\phi^{\prime\prime}(1)}\right)^{k/2}\hat{C}_{k}+O_{p}(n^{-2}).
\]
Here, the definition of $\hat{C}_k$ can be seen by comparing the above with Theorem \ref{thm: DRO_expansion}. Similarly, we have an expansion for the minimum value:
\[
\psi_{min}=\psi(\hat{P})+\sum_{k=1}^{3}n^{-k/2}\left(\frac{-q}{\phi^{\prime\prime}(1)}\right)^{k/2}\hat{C}_{k}+O_{p}(n^{-2}).
\]

Hence letting 
\begin{equation}\label{eq: g_expression}
g(x) =  \sum_{k=1}^{3} \hat{C}_k x^k
\end{equation}
we have that $\psi_{\max} =\psi(\hat{P}) + g \left(\sqrt{\frac{q}{n\phi^{\prime\prime}(1)}} \right)+O_p(n^{-2})$ and $\psi_{\min}=\psi(\hat{P}) +g \left(-\sqrt{\frac{q}{n\phi^{\prime\prime}(1)}} \right)+O_p(n^{-2})$.

Consequently, the coverage probability of the CI constructed as $(\psi_{min},\psi_{max})$ can be approximated by
\[
P\left(\psi(P_0)\in \left[\psi(\hat{P}) + g \left(-\sqrt{\frac{q}{n\phi^{\prime\prime}(1)}} \right), \psi(\hat{P}) +g \left(\sqrt{\frac{q}{n\phi^{\prime\prime}(1)}} \right)\right]\right)+O(n^{-3/2}).
\]
By inverting the $g$ function (it can be argued that $g$ is monotone in a small neighborhood of 0 with high probability), the above can in turn be approximated by
\[
P\left(\sqrt{n}g^{-1}\left(\psi(P_0)-\psi(\hat{P})\right)\in \left[-\sqrt{\frac{q}{\phi^{\prime\prime}(1)}}, \sqrt{\frac{q}{\phi^{\prime\prime}(1)}}\right]\right)+O(n^{-3/2}).
\]

Based on the expansion of $g$ in \eqref{eq: g_expression}, we can get an expansion for $g^{-1}$ around 0. In addition, we have an expansion for $\psi(P_0)-\psi(\hat{P})$ based on the influence functions of $\psi$. Putting these together, we obtain an expansion for  $\sqrt{n}g^{-1}\left(\psi(P_0)-\psi(\hat{P})\right)$, which enables us to study its distribution via Edgeworth expansion. As a result, we can show the following theorem:

\begin{thm}\label{thm: Bartlett}
Under some regularity conditions regarding the validity of Edgeworth expansion, the tail probability in the expansion in Assumption \ref{assu: psi_expansion}, and the expansion of $\psi(\hat{P})$ around $P_0$ (all these are explicitly given in Theorem \ref{thm: Bartlett_full}), we have the following approximation to the coverage probability of the CI $[\psi_{min},\psi_{max}]$:
\begin{equation}\label{eq: coverage_expansion}
P(\psi_{\min}\leq\psi(P_0)\leq\psi_{\max}) =  P\left(\chi^{2}_1\leq\frac{q}{\phi^{\prime\prime}(1)}\right)+n^{-1}A\left(\sqrt{\frac{q}{\phi^{\prime\prime}(1)}}\right)\phi\left(\sqrt{\frac{q}{\phi^{\prime\prime}(1)}}\right)+O(n^{-3/2})
\end{equation}
where 
\begin{align*}
A(x)= & -\frac{(2\phi^{\prime\prime}+\phi^{\prime\prime\prime})^{2}x^{5}\gamma^{2}}{36\phi^{\prime\prime2}\kappa_{2}^{3}}\\
 & -\frac{x^{3}}{36\phi^{\prime\prime2}\kappa_{2}^{3}}\times\left(\begin{array}{c}
4(\phi^{\prime\prime}+\phi^{\prime\prime\prime})(2\phi^{\prime\prime}+\phi^{\prime\prime\prime})\gamma^{2}+3(-2\phi^{\prime\prime2}-4\phi^{\prime\prime}\phi^{\prime\prime\prime}-3\phi^{\prime\prime\prime2}+\phi^{\prime\prime}\phi^{(4)})\kappa_{2}\mu_{4}\\
+9(2\phi^{\prime\prime}+\phi^{\prime\prime\prime})^{2}\kappa_{2}^{3}+6\phi^{\prime\prime}(2\phi^{\prime\prime}+\phi^{\prime\prime\prime})\gamma\mu_{2,c}-6\phi^{\prime\prime}(2\phi^{\prime\prime}+\phi^{\prime\prime\prime})\gamma\kappa_{2}\mu_{2,d}
\end{array}\right)\\
 & -\frac{x}{36\kappa_{2}^{3}}\times\left(\begin{array}{c}
-12\gamma^{2}+18\kappa_{2}\mu_{4}+36\kappa_{2}\left(\mu_{2,a}+2\mu_{2,b}\right)-36\gamma\mu_{2,c}\\
-9\mu_{2,c}^{2}-18\kappa_{2}\mu_{2,c}\mu_{2,d}+9\kappa_{2}^{2}\left(-2\mu_{2,2}+\mu_{2,d}^{2}-4\mu_{1,2,d}\right)
\end{array}\right).
\end{align*}
Here,
\[
\mu_{1,2,d}=EIF_{1}(X;P_{0})IF_{2}(X,X;P_{0}),
\]
\[
\mu_{2,d}=EIF_2(X,X;P_0),
\]
\[
\mu_{2,2}=E[IF_{2}(X,Y;P_{0})]^{2}
\]
and the definitions of $\kappa_2,\gamma,\mu_{2,a},\mu_{2,b},\mu_{2,c},\mu_4$ can be found in Theorem \ref{thm: DRO_expansion},
with the hat notation removed and $\hat{P}$ replaced with $P_{0}$. $E$ means the expectation for $X,Y\stackrel{i.i.d.}{\sim}P_0$. In the expression of $A(x)$, all of $\phi^{\prime\prime}, \phi^{\prime\prime\prime}$, and $\phi^{(4)} $ are evaluated at 1. 
\end{thm}

Based on the leading term on the RHS of \eqref{eq: coverage_expansion}, we can derive the choice of $q$ which gives an asymptotically exact CI.

\begin{corollary}\label{cor: exact_CI}
Under the same condition as in Theorem \ref{thm: Bartlett}, if $q$ is chosen as $q_0:=\phi^{\prime\prime}(1)\chi_{1;1-\alpha}^{2}$, then 
\[
P(\psi_{\min}\leq\psi(P_0)\leq\psi_{\max}) = \alpha + O(n^{-1}).
\]
Here, $\psi_{\min}$ and $\psi_{\max}$ are the minimum and maximum values of \eqref{eq: DRO_L_formlation} respectively, and $\chi^2_{1;1-\alpha}$ is the upper $(1-\alpha)$-quantile of $\chi^2_1$, i.e., $P(\chi^2_1\leq \chi^2_{1;1-\alpha})=\alpha$..
\end{corollary}

Corollary \ref{cor: exact_CI} is consistent with the result in \cite{duchi2016statistics}. Note that this choice of $q$ does not depend on the underlying distribution, which reveals a pivotal property of DRO.

Now we consider a higher-order correction based on the $n^{-1}$ order term in \eqref{eq: coverage_expansion}, which gives the desired Bartlett correction formula. 
\begin{corollary}\label{cor: Bartlett_formula}
Under the same condition as in Theorem \ref{thm: Bartlett}, if $q$ is chosen as
\begin{equation}\label{eq: Bartlett_formula}
q = \phi^{\prime\prime}(1){\chi_{1;1-\alpha}^{2}}\left(1-\frac{1}{n}\frac{A(\sqrt{\chi_{1;1-\alpha}^{2}})}{\sqrt{\chi_{1;1-\alpha}^{2}}}\right),
\end{equation} 
then
\[
P(\psi_{\min}\leq\psi(P_0)\leq\psi_{\max}) = \alpha + O(n^{-3/2}).
\]
Here, $\psi_{\min}$ and $\psi_{\max}$ are the minimum and maximum values of \eqref{eq: DRO_L_formlation} respectively. 
\end{corollary}

\begin{proof}
With a bit of calculation, it is not hard to show that for $q_0$ given in Corollary \ref{cor: exact_CI} and any $c$,
\[
P\left(\chi^{2}_1\leq q_0\left(1+\frac{c}{n}\right)\right)
=P(\chi^{2}_1\leq q_0)+\sqrt{q_0}\frac{c}{n}\nu(\sqrt{q_0})+O(n^{-2})
\]
Here $\nu(\cdot)$ is the standard normal density. Thus in \eqref{eq: coverage_expansion}, if we set $q=q_0(1+c/n)$, we would have the coverage probability given by 
\[
P\left(\chi^{2}\leq\frac{q_0}{\phi^{\prime\prime}(1)}\right)+\sqrt{\frac{q_0}{\phi^{\prime\prime}(1)}}\frac{c}{n}\nu\left(\sqrt{\frac{q_0}{\phi^{\prime\prime}(1)}}\right)+n^{-1}A(\sqrt{\frac{q_0}{\phi^{\prime\prime}(1)}})\nu(\sqrt{\frac{q_0}{\phi^{\prime\prime}(1)}})+O(n^{-3/2}).
\]
To make the above expression equal to $\alpha+O(n^{-3/2}) $, we can let $c=-\frac{A(\sqrt{\chi_{1;1-\alpha}^{2}})}{\sqrt{\chi_{1;1-\alpha}^{2}}}$. This gives \eqref{eq: Bartlett_formula}. 
\end{proof}

We notice that $A$ is a polynomial that depends on the expectations of (products of) influence functions. We can estimate $A$ using its sampled approximation, as this would only introduce an error of order $O(n^{-3/2})$ since $A$ is multiplied by $n^{-1}$ in our expression.


In the following two sections, we present further implications of our results.


\section{Higher-Order Self-Normalizing Property}\label{sec: self normalizing}
One notable property of the corrected choice of $q$ given by \eqref{eq: Bartlett_formula} is that it does not depend on the third-order influence function. To see why this is nontrivial, we compare with the higher-order delta method described as follows.

Since $A_n:=\frac{\sqrt{n}\left(\psi(\hat{P})-\psi(P_{0})\right)}{\left(\widehat{Var}[IF(X;\hat{P})]\right)^{1/2}}$ is asymptotically normal, we may study its distribution via Edgeworth expansion and then try to find $U_{1}$ and $U_{2}$ such that (WLOG, suppose that $\phi^{\prime\prime}(1)=1$ so that $q_0=\chi^2_{1;1-\alpha}$)
\[
P\left(-\sqrt{q_0}-U_{1}\leq A_n\leq\sqrt{q_0}+U_{2}\right)=\alpha+O(n^{-3/2}).
\]
Note that once we have the above relation, we can get the following CI for $\psi(P_0)$ with a higher-order coverage error than applying the standard delta method:
\[
\left[\psi(\hat{P})-n^{-1/2}\left(\widehat{Var}[IF(X;\hat{P})]\right)^{1/2}(\sqrt{q_0}+U_2),\psi(\hat{P})+n^{-1/2}\left(\widehat{Var}[IF(X;\hat{P})]\right)^{1/2}(\sqrt{q_0}+U_1)\right].
\]




To proceed, one natural way is to set $U_{1}=U_{2}=D/n$, and then find $D$ such that
\begin{equation}\label{eq: selfn_D}
P\left(-\sqrt{q_0}-D/n\leq A_n \leq\sqrt{q_0}+D/n\right)=\alpha+O(n^{-3/2}).
\end{equation}
Letting $\Phi(\cdot)$ and $\nu(\cdot)$ be the cumulative distribution function and density of standard normal respectively, we write the Edgeworth expansion of $A_n$ as
\begin{equation}
F_{A_{n}}(x)=\Phi(x)+p_{1}(x)\nu(x)n^{-1/2}+p_{2}(x)\nu(x)n^{-1}+O(n^{-3/2}).\label{eq: edgeworth_A_n}
\end{equation}
where $p_{1}(x)$ is an odd polynomial of $x$ and $p_{2}(x)$ is an even polynomial of $x$. By computation, we can show that the polynomial $p_{2}(x)$ is dependent on the third-order influence function (roughly speaking, this is because the third-order influence function is necessary to approximate $A_n$ up to a residual of $O_p(n^{-3/2})$). Based on the Edgeworth expansion for $A_n$, the LHS of \eqref{eq: selfn_D} can be expanded as
\begin{align*}
 & P(-\sqrt{q_0}-D/n\leq A_n\leq\sqrt{q_0}+D/n)\\
= & \Phi(\sqrt{q_0}+D/n)-\Phi(-\sqrt{q_0}-D/n)+n^{-1}(p_{2}(\sqrt{q_0})-p_{2}(-\sqrt{q_0}))\nu(\sqrt{q_0})+O(n^{-3/2})\\
= & P(\chi^{2}\leq q_0)+2n^{-1}D\nu(\sqrt{q_0})+2n^{-1}p_{2}(\sqrt{q_0})\nu(\sqrt{q_0})+O(n^{-3/2}).
\end{align*}
Therefore, to achieve \eqref{eq: selfn_D}, the choice of $D$ is given by $D=-p_{2}(q_0)$, which depends on the third-order influence function. 

The third-order influence function is required even if we choose $U_1$ and $U_2$ in more complicated ways. For example, suppose that we set $U_1=D_1n^{-1/2}+D_2n^{-1}$ and $U_2=D_3n^{-1/2}+D_4n^{-1}$. Then  we want
\[
P(-\sqrt{q_0}+D_{1}n^{-1/2}+D_{2}n^{-1}\leq A_n\leq\sqrt{q_0}+D_{3}n^{-1/2}+D_{4}n^{-1})=\alpha + O(n^{-3/2})
\]
From the Edgeworth expansion for $A_{n}$ given in \eqref{eq: edgeworth_A_n}, the LHS above can be expanded as
\begin{align*}
 & \Phi(\sqrt{q_0}+D_{3}n^{-1/2}+D_{4}n^{-1})-\Phi(-\sqrt{q_0}+D_{1}n^{-1/2}+D_{2}n^{-1})\\
+ & n^{-1/2}\left(p_{1}(\sqrt{q_0}+D_{3}n^{-1/2})\nu(\sqrt{q_0}+D_{3}n^{-1/2})-p_{1}(-\sqrt{q_0}+D_{1}n^{-1/2})\nu(-\sqrt{q_0}+D_{1}n^{-1/2})\right)\\
+ & 2n^{-1}p_{2}(\sqrt{q_0})\nu(\sqrt{q_0})+O(n^{-3/2}).
\end{align*}
We need to choose $D_{1},D_{2},D_{3},D_{4}$ so that the above
is equal to $\alpha+O(n^{-3/2})$. From our computation
we can show that $p_{1}$ does not depend on third-order influence
functions. More precisely, the moments of $A_{n}$ are computed in Appendix
\ref{sec:Computation_edgeworth}. From there, we can see that the
first three cumulants of $A_{n}$ do not depend on the
third-order influence function of $\psi$. From formula (2.24) of
\citet{Hall1992} we know that $p_{1}$ only depends on cumulants up
to the third order, and hence we can conclude that $p_{1}$ is independent
of the third-order influence function. However, $p_2$ depends on the third-order influence function. Therefore, to cancel
the third-order influence function in $p_{2}$ in the above display,
some of $D_{1},D_{2},D_{3},D_{4}$ must be dependent on the third-order influence function.

The above analysis shows that the third-order influence function is necessary if we want to construct a CI with a coverage error of order $O(n^{-3/2})$ using the delta method. Therefore, the independence of the third-order influence function in the correction of $q$ can be seen as an advantage for DRO. We call this a higher-order self-normalizing property. ``Self-normalizing'' means that, in order to get an asymptotically exact CI from EL (or DRO whose connection will be explained in detail momentarily), it suffices to set $q$ as the quantile of $\chi^2$, which does not depend on any (estimated) information for $\psi$. However, if we construct the CI using the delta method, we need to estimate the variance of the first-order influence function. The independence of the third-order influence function in the correction of $q$ in DRO, compared to the need of estimating the third-order influence function in the natural correction of the delta method, in order to drive down coverage error, therefore forms a higher-order analog of the self-normalizing behavior.

\section{Comparisons with Empirical Likelihood}\label{sec: comparison with EL}

\subsection{Duality between DRO and EL}\label{sec:duality}

DRO and EL are connected via a duality relation in the following sense.
First, the so-called profile likelihood, denoted $W_n(\psi_1)$, in the EL method is defined as
\begin{equation}\label{eq: EL_formulation}
\begin{array}{c} 
\min_P D_{\phi}(P||\hat{P})\\
s.t.\ \psi(P)=\psi_1
\end{array}\end{equation}
where $D_{\phi}$ is set as the reverse KL divergence, or the EL, in the original setup in \cite{owen1988empirical}. The resulting CI can be written as $\{\psi: W_n(\psi_1)\leq \eta\} $, where $\eta$ is precisely $q_0/2n$, and $q_0$ is the choice of $q$ that gives asymptotic exact CI using DRO \eqref{eq: DRO_L_formlation} as defined in Corollary \ref{cor: exact_CI}. This is seen by establishing a nonparametric analog of Wilks' Theorem that shows $W_n(\psi(P_0))$ converges to $\chi^2_1$. Compared to DRO, EL places the divergence in the objective instead of constraint, and the target functional in the constraint instead of the objective.


To argue more concretely the dual relationship between DRO \eqref{eq: DRO_formulation} and EL \eqref{eq: EL_formulation}, suppose for simplicity there exists a minimizer for both optimization problems. Notice that, if $W_n (\phi_1)\leq \eta$, then  there exists $P_1$ such that $\psi (P_1) = \psi_1$ and $D_\phi (P||\hat{P})\leq \eta$. This means this $P_1$ is feasible for \eqref{eq: DRO_formulation} with objective value $\psi_1$. In other words, $\psi_1\in [\psi_{\min}(n,\eta),\psi_{\max}(n,\eta)]$. On the other hand, suppose that $\psi_1\in [\psi_{\min}(n,\eta),\psi_{\max}(n,\eta)]$. Then notice that the feasible region of DRO \eqref{eq: DRO_formulation} is convex (and thus connected), and we have that $\psi_1$ could be achieved by some $P_1$ feasible to \eqref{eq: DRO_formulation}. Therefore, there exists $P_1$ such that $D_\phi (P_1||\hat{P})\leq \eta$ with $\psi(P_1)\leq \eta$. This implies $W_n (\psi_1)\leq \eta$. Combing both directions, we have that $[\psi_{\min}(n,\eta),\psi_{\max}(n,\eta)]=\{\psi: W_n(\psi_1)\leq \eta\} $ so DRO and EL constructs the same CI.

\subsection{Difference in Analysis Route between DRO and EL}


From the equivalence of EL and DRO, in principle we can use either one for developing higher-order corrections. However, in the EL formulation, there is no explicit constraint that tells $P$ should be close to $\hat{P}$ (or $L$ should be close to 1). This imposes difficulty to rigorously expand $\psi(P)$ around $P_0$ uniformly over $L$. In comparison, the following lemma holds for DRO:
\begin{lem}
\label{lem: L_bound}(Lemma 13 of \citet{duchi2016statistics}) Suppose
that $\phi$ is convex and $\phi(1)=\phi^{\prime}(1)=0,\phi^{\prime\prime}(1)>0$
and has continuous second-order derivative on $(0,\infty)$. Then
there exists constants $c_{\phi}$ and $C_{\phi}$ such that $\left\Vert L-1\right\Vert_{2}\in[c_{\phi},C_{\phi}]$ for all $L$ such that
$\hat{E}[\phi(L)]\leq\frac{q}{2n}$. Here $\left\Vert\cdot\right\Vert_2$ is the Euclidean norm and $L$ is regarded as a vector when computing the norm.
\end{lem}

Lemma \ref{lem: L_bound} implies that (regarding $\hat{P}L$ and $L$ as vectors) $||\hat{P}L-\hat{P}||_{2}\in[n^{-1/2}c_{\phi},n^{-1/2}C_{\phi}]$, and thus $\hat{P}L-\hat{P}$ is of order $n^{-1/2}$. Hence, in the expansion in Assumption \ref{assu: psi_expansion} part 1, we expect that $E_{\Delta}IF_j(K_1,K_2,\dots,K_j;\hat{P})$ is of order $ n^{-j/2}$. In other words, the DRO formulation explicitly requires that $L$ belongs to a bounded ball, which makes it easier to obtain uniform expansions.  



Nonetheless, for some special models, it is more convenient
to study using EL. For example,
if $\psi(P)$ is implicitly given as the root of $E_{P}m(X,\psi)=0$,
then we could just write the constraint $\psi\left(P\right)=\psi\left(P_{0}\right)$
as $E_{P}m(X,\theta)=0$ which is linear
in $P$. For general statistical functionals, however, DRO is easier to work with as we have argued.

\subsection{Generalizing Criterion for Being Bartlett Correctable}

Let $W_{n0}:=2n\phi^{\prime\prime}(1)W_n(\psi_0)$ denote the EL statistic, whose asymptotic distribution is $\chi^2_1$. The Bartlett correctability in the EL literature refers to the existence of constant $a$ such that $P(W_{n0}(1-a/n) \leq t) = P(\chi^2_1\leq t) + O(n^{-3/2})$ or $P(W_{n0} \leq t(1+a/n)) = P(\chi^2_1\leq t) + O(n^{-3/2})$. From the duality between DRO and EL, we have that  $P(W_{n0}\leq t(1+a/n))=P(\psi(P_0)\in (\psi_{\min}(n,\eta),\psi_{\max}(n,\eta)))$ where $\eta = \frac{t(1+a/n)}{2n\phi^{\prime\prime}(1)}$. Let $t=\chi^2_{1;1-\alpha}$. Then comparing with \eqref{eq: Bartlett_formula}, we have that $a=-\frac{A(\sqrt{\chi^2_{1;1-\alpha}})}{\sqrt{\chi^2_{1;1-\alpha}}}$. Hence $a$ being a constant is equivalent to that $\frac{A(\sqrt{\chi^2_{1;1-\alpha}})}{\sqrt{\chi^2_{1;1-\alpha}}}$ is a constant independent of $\alpha$. This is in turn equivalent to saying that the coefficients for $x^3$ and $x^5$ in the expression for $A(x)$ are zero. By checking the expression of $A(x)$, we can see that this is equivalent to 
\[
\phi^{\prime\prime\prime}(1)=-2\phi^{\prime\prime}(1),\phi^{(4)}(1)=-3\phi^{\prime\prime\prime}(1).
\]
In other words, Bartlett correctability is equivalent to the above relation, for general statistical functional estimation. This generalizes the observations in Section 3 of \cite{CORCORAN}, which only works for the smooth function model $\psi(P)=f(E_P Z)$.

\subsection{Comparison with Bartlett Corrections for Smooth Function Models}\label{sec:comparison smooth}
For the special case that $\psi(P)=\theta(E_{P}{X})$ and $\phi$ is set as the reverse KL-divergence (i.e., $\phi(x)=-\log x+ x - 1$), an explicit
formula for Bartlett correction is given in \citet{diciccio1991} (also given in Section 13.2 of \citet{owen2001empirical}). To compare with that result, we first reduce ours to this special case using the notations in \citet{diciccio1991}. We have the following:
\begin{corollary} \label{cor: function_of_mean}
Suppose that the conditions in Theorem \ref{thm: Bartlett} holds. Suppose further that $\psi(P)=\theta(E_{P}X)$ and $\phi(x)=-\log x+ x - 1$. Define ${\mu_0} = E_{P_0}X$ and $\alpha^{j_1\dots j_k} = E(X^{j_1}-\mu^{j_1}_0)\dots(X^{j_k}-\mu^{j_k}_0)$ where $X^{j}$ and $\mu^j$ denote the $j$-th element of ${X}$ and ${\mu}$, respectively. Suppose that $\alpha_{ij}=\delta_{ij}:=1\{i=j\}$. Let $\theta_{j_1\dots j_k}=\frac{\partial^k\theta(\mu)}{\partial\mu^{j_1}\dots\partial\mu^{j_k}}|_{\mu=\mu_0}$. Define $Q=(\sum_i \theta_i^2)^{-1},M^{ij}=Q\theta_i\theta_j,N^{i}=Q\theta_i$, 
\begin{gather*}
    \begin{array}{l}
         t_1 = \sum_{jklmno}\alpha^{jkl}\alpha^{mno}M^{jm}M^{kn}M^{lo},\  t_2=\sum_{jklmno}\alpha^{jkl}\alpha^{mno}M^{jk}M^{lm}M^{no},  \\
         t_3 = \sum_{jklm}\alpha^{jklm}M^{jk}M^{lm},\  t_4 = \sum_{jklmn}\alpha^{jkl}N^{j}\theta_{mn}(\delta^{mk}-M^{mk})(\delta^{nl}-M^{nl}),\\
         t_5 = \sum_{jklm}Q\theta_{jk}\theta_{lm}\left\{(\delta^{jk}-M^{jk})(\delta^{lm}-M^{lm})-2(\delta^{jl}-M^{jl})(\delta^{km}-M^{km})\right\}
    \end{array}
\end{gather*}
Then $-A(x)/x=\frac{5}{3}t_1-2t_2+\frac{1}{2}t_3-t_4+\frac{1}{4}t_5$.
\end{corollary}
  Compared to the result in \cite{diciccio1991}, the only difference is that we have a ``-2'' instead of ``+2'' in the formula for $t_5$. 
We politely claim that this is due to some error in the computation of \citet{diciccio1991}, and detail in what follows. If $X$ is one dimensional, such a difference would not affect the Bartlett correction formula. But if $X$ is multidimensional, then in general when $\sum_{jklm}Q\theta_{jk}\theta_{lm}(\delta^{jl}-M^{jl})(\delta^{km}-M^{km})\neq 0$, this would result in a different formula.

We will continue to use the notations and conditions as in Corollary \ref{cor: function_of_mean}. In addition, the definitions of $\bar{W}(\mu),B,U$ are given in the computation of step 4 of \cite{diciccio1988}, which is the technical report for the computation procedure for the result in \cite{diciccio1991}. $A$ is defined as $A=\frac{1}{n}\sum_{i=1}^nX_i$ (column vector) and it is assumed that $\mu_0=EX_i=0$ and $\Sigma_0 = EX_iX_i^T=I$, as in \cite{diciccio1988}.

In step 4 on page 9 of \citet{diciccio1988},
based on the asymptotic expansion for $\bar{W}(\mu)$, they want to
find $\mu$ to minimize $\bar{W}(\mu)$ subject to $\theta(\mu)=\theta_{0}$. This idea appears correct, but there is some problem with their computation method. More precisely, they suppose that the optimal $\mu$ (denoted by $\tilde{\mu}$)
has an expansion $\tilde{\mu}=A-(\mu_{1}+\mu_{2}+\mu_{3})$ where
$\mu_{i}$ is of order $O_{p}(n^{-i/2})$. Then,
they try to solve $\mu_{1}$, then $\mu_{2}$ and then $\mu_{3}$.
The problem appears when they try to solve $\mu_{2}$. After solving
$\mu_{1}$, they write that $\mu_{2}$ should be chosen to minimize the $\Omega(n^{-2})$ part of $n^{-1}\bar{W}(\mu)$, which leads to the claim that $\mu_2$ is the optimal solution to
\begin{equation}\label{eq: dic_obj} 
\sum_j\left[\mu_{2}^{j}\mu_{2}^{j}+2\mu_{2}^{j}(\mu_{1}^{j}+B^{j})\right]\ \ \text{such that}\ \ \Theta\mu_{2}=U
\end{equation}
However, $\sum_j\left[\mu_{2}^{j}\mu_{2}^{j}+2\mu_{2}^{j}(\mu_{1}^{j}+B^{j})\right]$
does not include all of the $O_{p}(n^{-2})$ terms of $n^{-1}\bar{W}(\mu)$. Indeed,
2$\sum_j\mu_{1}^{j}\mu_{3}^{j}$ is another $O_{p}(n^{-2})$ term. Although $\sum_j\mu_{1}^{j}\mu_{3}^{j}$ does not depend on $\mu_2$ explicitly, different choices of $\mu_2$ will give different restrictions to the choice of $\mu_3$. Hence, it is necessary to take into account $\sum_j\mu_{1}^{j}\mu_{3}^{j}$ in the objective. 



\subsection{A Corrected Version of \texorpdfstring{\cite{diciccio1988}}{TEXT}}\label{sec:corrected}

The corrected version goes as follows. We solve $\mu_{1}$ as in \cite{diciccio1988}. Before solving $\mu_{2}$, we investigate the choice of $\mu_{3}$
when $\mu_{2}$ is given. As derived in step 4 of \cite{diciccio1988}, the constraint for $\mu_{3}$ is given by
\[
\Theta\mu_{3}=-\mu_{2}^{\top}\nabla^{2}(A-\mu_{1})+\frac{1}{6}\sum_{jkl}\theta_{jkl}(A-\mu_{1})^{j}(A-\mu_{1})^{k}(A-\mu_{1})^{l}
\]
where $\Theta$
is the gradient of $\theta(\mu)$ at $0$ and is a row vector, and $\nabla^2$ is the Hessian of $\theta(\mu)$ at 0. 
The computation in \cite{diciccio1988} gives that $\mu_{1}=MA=\Theta^{\top}(\Theta\Theta^{\top})^{-1}\Theta A$, which is parallel to $\Theta^{\top}$.
Hence once $\mu_1$ and $\mu_2$ are given, $\mu_1^{\top}\mu_3$ can be expressed as a multiple of $\Theta\mu_{3}$, which is determined by the above displayed equation. Adding $2\mu_1^{\top}\mu_3$ to the objective in \eqref{eq: dic_obj}, we find that we should actually optimize
\[
\mu_{2}^{\top}\mu_{2}+2\mu_{2}^{\top}(\mu_{1}+B+R) \text{\ \ such that\ \ } \Theta\mu_{2}=U
\]
where $R:=-(A^{\top}N)\nabla^{2}(I-M)A$. 
The optimal solution should be
\[
\mu_{2}^{\prime}:=NU-(I-M)(B+R) =\mu_2 - (I-M)R
\]
where $\mu_2$ on the RHS is the optimal solution for $\mu_2$ given in \cite{diciccio1988}. Then, following the computation for $\mu_3$ as in \cite{diciccio1988}, we get that the correct expression for $\mu_3$ should be 
\[
\mu_3^{\prime} := \mu_3 + NA^{\top}(I-M)^{\top}\nabla^2(I-M)R
\]
As a result, $n^{-1}W_{0}$ computed in (3.8) of \cite{diciccio1988} should be deducted by (note that $\mu_1^{\top}(\mu_2^{\prime}-\mu_2)=0$ so $\mu_1^{\top}\mu_2$ in (3.8) is not changed)
\begin{align*}
& 2\mu_1^{\top}(\mu_3^{\prime}-\mu_3) - (\mu_2-\mu_2^{\prime})^{\top} (\mu_2-\mu_2^{\prime})+ 2\mu_2^{\top}(\mu_2 - \mu_2^{\prime}) \\
= & (A^{\top}N)^2 A^{\top}(I-M)\nabla^{2}(I-M)\nabla^{2}(I-M)A + 2\mu_2^{\top}(\mu_2 - \mu_2^{\prime}),
\end{align*}
It can be checked that the second part of the RHS above only contributes $O(n^{-3})$ to the expectation, i.e., $E\mu_2^{\top}(\mu_2-{\mu_2}^{\prime})=O(n^{-3})$. Thus, the expectation of the above is 
\[
\frac{1}{n^{2}}\sum_{j,k,l,m}\left(Q\theta_{jk}\theta_{lm}(I-M)^{kl}(I-M)^{mj}\right)+O(n^{-3}).
\]
Following (3.11) and (3.12) of \cite{diciccio1988}, we can see that the Bartlett adjustment term (which corresponds to the ``$\frac{5}{3}t_1-2t_2+\frac{1}{2}t_3-t_4+\frac{1}{4}t_5$'' in the statement of Corollary \ref{cor: function_of_mean} and also (3.12) of \cite{diciccio1988}; note that there is a typo in the latter: the coefficient of $t_4$ should be $-1$) is given by $n^2E[n^{-1}(W_0-1)]$. So we should deduct $n^2$ times of the above from the Bartlett correction formula (2.5) of \cite{diciccio1991}. This gives the result as in Corollary \ref{cor: function_of_mean}. 

\section{Numerical Experiments}\label{sec: numerical}
We conduct some numerical experiments to demonstrate the coverage enhancement of our DRO approach. The setup and choice of parameters are similar to \cite{diciccio1991}, but we consider general functionals, namely $V$-statistics and a stochastic optimization problem. In addition, in Appendix \ref{sec:numerics smooth} we present an example on the smooth function model where we numerically compare with the formula in  \cite{diciccio1991}, in support of our correctness claims in Sections \ref{sec:comparison smooth} and \ref{sec:corrected}.

\subsection{$V$-Statistics}
Consider $\psi (P) = E_{X_1,X_2\stackrel{\text{i.i.d.}}{\sim}P}h(X_1,X_2)$. We construct a CI for $\psi (P)$ using $n$ samples drawn independently from $P$. We compare among: 1) Standard EL: This is also equivalent to DRO via the duality discussed in Section \ref{sec:duality}; 2) Estimated Bartlett (EB): This refers to DRO with the corrected ball size $\eta$ in Corollary \ref{cor: Bartlett_formula} where the adjustment term is estimated from the sample (as discussed at the end of Section \ref{sec: Bartlett}); and 3) Theoretical Bartlett (TB): This refers to DRO with Bartlett correction using the theoretical adjustment term. In our numerical experiments, this is done by estimating the adjustment term using 5,000 samples. This sample size is large compared to $n$, which is set to be $\leq 50$ in all our experiments. The true value $\psi(P)$ is estimated by generating $1,000,000$ independent pairs  $(X_{1i},X_{2i})$ and taking the average of $h(X_{1i},X_{2i})$. For each method and each tested value of $n$, we replicate the experiment $N=100,000$ times to estimate the coverage accuracy.  Table \ref{tab: numerical_VstatGamma} describes the result when $P=\text{Gamma}(2,1)$, $h(X,Y)=\min\{12,(X-Y)^2+X+Y\}$ and the reverse KL-divergence with $\phi(x)=-\log x+ x - 1$ is used.

\begin{table}[H]
\centering
\caption{Coverage probabilities (with 95\% confidence interval) for $V$-statistics with $h(X,Y)=\min\{12,(X-Y)^2+X+Y\}$}
\label{tab: numerical_VstatGamma}
\small
\begin{tabular}{c|l l l}
\toprule

 Nominal Level       & 80\%   & 90\%      & 95\% \\
 \midrule
 EL,$n=15$     &$76.17\%\pm 0.26\%$ &$85.70\%\pm 0.22\%$ & $89.82\%\pm 0.19\%$\\
EB, $n=15$   &$78.23\%\pm 0.26\%$  & $86.59\%\pm 0.21\%$  & $90.64\%\pm 0.18\%$ \\
TB, $n=15$   &$77.87\%\pm 0.26\%$  &$86.76\%\pm 0.21\%$  & $90.75\%\pm 0.18\%$   \\
EL, $n=30$     &$78.59\%\pm 0.25\%$  &$88.79\%\pm 0.20\%$  & $93.97\%\pm 0.15\%$  \\
EB, $n=30$     & $79.65\%\pm 0.25\%$ &$89.62\%\pm 0.19\%$ & $94.60\%\pm 0.14\%$  \\
TB, $n=30$ &$79.47\%\pm 0.25\%$ &$89.39\%\pm 0.19\%$ & $94.47\%\pm 0.14\%$\\

\bottomrule
\end{tabular}
\end{table}

As we can see from Table \ref{tab: numerical_VstatGamma}, both TB and EB have improvements over EL, and the performances of TB and EB appear close. For example, when $n=15$ and the nominal level is 90\%, the empirical coverage probability of EL is 85.70\%, while EB and TB improve the empirical coverage probability to 86.59\% and 86.76\% respectively. We also observe that the improvement is more significant when $n$ is smaller. For example, when the nominal level is 80\%, the improvement given by EB compared to EL is about 2\% when $n=15$,  but the improvement is only about 1\% when $n=30$. 


%


Consider another example where $h(X_1,X_2) = \sin (X_1^2+X_2)$ and $P$ is the $t$-distribution with 3 degrees of freedom. We use the $\chi^2$-divergence where $\phi (x) = (x-1)^2.$ Other setups are the same as the previous experiment. The result is shown in Table \ref{tab: numerical_VstatSin}. 
\begin{table}[H]
\centering
\caption{Coverage probabilities (with 95\% confidence interval)  for $V$-statistics with $h(X,Y) = \sin (X^2+Y)$}
\label{tab: numerical_VstatSin}
\small
\begin{tabular}{c|l l l}
\toprule

 Nominal Level       & 80\%   & 90\%      & 95\% \\
 \midrule
 EL,$n=10$     &$75.88\%\pm 0.27\%$ &$86.49\%\pm 0.21\%$ &$92.04\%\pm 0.17\%$ \\
EB, $n=10$   &$78.66\%\pm 0.25\%$  & $88.82\%\pm 0.20\%$  &$94.06\%\pm 0.15\%$ \\
TB, $n=10$   & $80.72\%\pm 0.24\%$ &$90.66\%\pm 0.18\%$  &  $95.46\%\pm 0.13\%$  \\
 EL,$n=15$     &$76.85\%\pm 0.26\% $ &$87.36\%\pm 0.21\%$ & $92.84\%\pm 0.16\%$\\
EB, $n=15$   & $79.43\%\pm 0.25\%$ & $89.84\%\pm 0.19\%$  &$94.74\%\pm 0.14\%$ \\
TB, $n=15$   &$80.12\%\pm 0.25\%$  &$90.40\%\pm 0.18\%$  & $95.20\%\pm 0.13\%$    \\
EL, $n=30$     & $78.15\%\pm 0.26\%$ &$88.38\%\pm 0.20\%$  & $93.67\%\pm 0.15\%$   \\
EB, $n=30$     &$79.79\%\pm 0.25\%$  &$89.87\%\pm 0.19\%$ &$94.94\%\pm 0.14\%$    \\
TB, $n=30$ &$79.93\%\pm 0.25\%$ &$89.98\%\pm 0.19\%$ &$94.95\%\pm 0.14\%$ \\

\bottomrule
\end{tabular}
\end{table}
From Table \ref{tab: numerical_VstatSin}, we see a more significant improvement given by our correction. For example, when the nominal level is 80\% and the sample size $n$ is only 10, the  empirical coverage probability of EB is 78.66\%. On the other hand, for EL, even when $n=30$, the empirical coverage probability is only 78.15\%. When $n=30$, the performances of EB and TB appear close as in the previous experiment. However, when $n$ is small ($n=10$), we see that TB has significantly smaller coverage errors compared to EB. For example, when the nominal level is 80\% and $n=10$, TB exhibits an empirical coverage probability of 80.72\%, which is significantly better than EB whose empirical coverage probability is 78.66\%. We conjecture the reason is that with only 10 observations, it is difficult to estimate the adjustment term accurately. 

\subsection{Optimization}
Consider estimating the optimal value of a stochastic optimization problem, $\psi (P) = \inf_x E_P \ell (x;\xi)$ where $\xi = (Y,Z)$. Suppose that under $P$, $Z\sim \chi^2_2$ and $Y=Z+\epsilon$ where $\epsilon\sim N(0,1)$ and $Z$ and $\epsilon$ are independent. The loss function is chosen as the squared loss $\ell (x,(Y,Z)) = (Y-xZ)^2$. We use the reverse KL-divergence as in the first experiment. By construction, we know that $\psi(P) = 1$ which is obtained at the solution $x=1$, so we do not need to estimate $\psi(P)$ by simulation. Other configurations are the same as the previous experiments. The result is shown in Table \ref{tab: numerical_Opt}.

\begin{table}[H]
\centering
\caption{Coverage probabilities (with 95\% confidence interval) for the optimization model}
\label{tab: numerical_Opt}
\small
\begin{tabular}{c|l l l}
\toprule

 Nominal Level       & 80\%   & 90\%      & 95\% \\
 \midrule
 EL, $n=30$     &$73.49\%\pm 0.27\%$  &$83.85\%\pm0.23\%$ & $89.52\%\pm 0.19\%$  \\
EB, $n=30$     & $75.42\%\pm 0.27\%$ &$85.36\%\pm 0.22\%$ & $90.47\%\pm 0.18\%$  \\
TB, $n=30$ &$78.71\%\pm 0.25\%$ &$87.19\%\pm 0.21\%$ & $92.18\%\pm 0.16\%$\\
 EL,$n=50$     &$75.99\%\pm 0.26\%$ &$86.26\%\pm 0.21\%$ &$91.99\%\pm 0.17\%$ \\
EB, $n=50$   &$77.49\%\pm 0.26\% $ &$87.33\%\pm 0.21\%$  &$92.78\%\pm 0.16\%$ \\
TB, $n=50$   &$78.55\%\pm 0.25\%$  &$88.48\%\pm 0.20\%$  &$93.53\%\pm 0.15\%$    \\

\bottomrule
\end{tabular}
\end{table}

As in the preceding two experiments, our correction reduces the coverage error significantly. However, this example is harder in the sense that to get a coverage probability close to the nominal level, the required sample size is larger than the $V$-statistic case (so we set $n=30$ and $n=50$ for this example). Similar to the second experiment, the difference between TB and EB is significant when $n$ is smaller ($n=30$). For example, when $n=30$ and the nominal level is 80\%, the estimated coverage probabilities for EB and TB are 75.42\% and 78.71\%, respectively, which differ by more than 3\%. 

In all of the three experiments, the empirical coverage probability of EL is below the nominal level. Our correction enlarges the CI and consequently improves the coverage. 

\section{Generalization to Multidimensional Settings}\label{sec:multi}

We discuss a potential way to generalize our results to the multidimensional case, leaving the full investigation to future work. The goal is to construct a confidence region for a multidimensional von Mises functional. In parallel to the one-dimensional case, we want a confidence region with shape $A_{\eta}:=\{\boldsymbol{\psi}(P):D_{\phi}(P||\hat{P})\leq\eta\}$ where $\boldsymbol{\psi}(P)\in\mathbb{R}^d$.
Moreover, we want to achieve a higher-order correction:
\[
P(\boldsymbol{\psi}(P_{0})\in A_{\eta})=1-\alpha+O(n^{-3/2}).
\]
Notice that for any $t\in \mathbb{R}^{d}$, $t\cdot A_{\eta}:=\{t\cdot\boldsymbol{\psi}_1:\boldsymbol{\psi}_1\in A_{\eta}\}$ is always a
connected interval, which can be seen by reducing to the one-dimensional case with objective
$t\cdot\boldsymbol{\psi}(P)$. Hence, the above display is equivalent to
\[
P\left(t\cdot\left(\boldsymbol{\psi}(P_{0})-\boldsymbol{\psi}(\hat{P})\right)\in t\cdot\left(A_{\eta}-\boldsymbol{\psi}(\hat{P})\right),\forall t\in \mathbb{R}^{d}\right)=1-\alpha+O(n^{-3/2}).
\]

In Theorem \ref{thm: DRO_expansion}, by replacing $\psi$ with $t\cdot\boldsymbol{\psi}$,
we can get an approximation to the shape of the confidence region:
\[
t\cdot A_{\eta}\approx[S_{\min}(t,\eta),S_{\max}(t,\eta)].
\]

Based on this approximated confidence region, it suffices to find a choice of $\eta$ such that
\[
P\left(S_{\min}(t,\eta)\leq t\cdot\left(\boldsymbol{\psi}(P_{0})-\boldsymbol{\psi}(\hat{P})\right)\leq S_{\max}(t,\eta),\forall t\in \mathbb{R}^{d}\right)=1-\alpha+O(n^{-3/2}).
\]
From the formula in Theorem \ref{thm: DRO_expansion} we can see that both $S_{\min}$ and $S_{\max}$ are polynomials of $t$ where the coefficients are determined by the moments of the influence functions. As in Section \ref{sec: Bartlett}, we introduce a function $g_t$ such that $S_{\max}(t,\eta)=g_{t}(\sqrt{\eta}),S_{\min}(t)=g_{t}(-\sqrt{\text{\ensuremath{\eta}}})$. Then
\begin{align*}
 & P(S_{\min}(t,\eta)\leq t\cdot\left(\boldsymbol{\psi}(P_{0})-\boldsymbol{\psi}(\hat{P})\right)\leq S_{\max}(t,\eta),\forall t\in \mathbb{R}^{d})\\
= & P\left(-\sqrt{\eta}\leq g_{t}^{-1}\left(t\cdot\left(\boldsymbol{\psi}(P_{0})-\boldsymbol{\psi}(\hat{P})\right)\right)\leq\sqrt{\eta},\forall t\in \mathbb{R}^{d}\right)\\
= & P\left(\max_{t\in \mathbb{R}^{d}}\left(g_{t}^{-1}\left(t\cdot\left(\boldsymbol{\psi}(P_{0})-\boldsymbol{\psi}(\hat{P})\right)\right)\right)^{2}\leq\eta\right)\\
= & P\left(\max_{t\in \mathbb{R}^{d},\left\Vert t\right\Vert _{2}=1}\left(g_{t}^{-1}\left(t\cdot\left(\boldsymbol{\psi}(P_{0})-\boldsymbol{\psi}(\hat{P})\right)\right)\right)^{2}\leq\eta\right).
\end{align*}
The last line follows from the observation that $g_{at}=ag(t)$ for any $a\in \mathbb{R}$ (since by scaling the objective function, it is readily seen that $S_{\max}(at,\eta)=aS_{\max}(t,\eta)$).  We can solve a constrained optimization problem to approximate 
$\max_{t\in \mathbb{R}^{d},\left\Vert t\right\Vert _{2}=1}\left(g_{t}^{-1}\left(t\cdot\left(\boldsymbol{\psi}(P_{0})-\boldsymbol{\psi}(\hat{P})\right)\right)\right)^{2}$
 and then use Edgeworth expansion to give an expansion to the
preceding probability. Then we can find $\eta$ such that the probability
equals $1-\alpha+O(n^{-3/2})$.

\section{Technical Developments for DRO Expansion}\label{sec: DRO_technical}

\subsection{The Roadmap}
Under Assumption \ref{assu: psi_expansion}, to get an approximation for the optimal value with residual $O_p(n^{-2})$, we can replace the objective function in \eqref{eq: DRO_L_formlation} with its approximation and solve the following problem
\begin{equation} \label{eq: hatP_L_leading}
\begin{array}{ll}
\max_{L} & \psi(\hat{P})+\sum_{j=1}^3\frac{1}{j!} E_{\Delta}IF_{j}(K_{1},K_{2},\dots,K_{j};\hat{P})\\
\text{subject to} & \hat{E}[\phi(L)]\leq\frac{q}{2n}\\
& \hat{E}L=1
\end{array}
\end{equation}
Throughout this section, we denote the objective of \eqref{eq: hatP_L_leading} as 
\begin{equation}\label{eq: psi_truncation}
{\psi}(\hat{P}L) :=  \psi(\hat{P})+\sum_{j=1}^3\frac{1}{j!} E_{\Delta}IF_{j}(K_{1},K_{2},\dots,K_{j};\hat{P}).
\end{equation}
Our starting point is the KKT condition. The following lemma allows us to use the KKT condition
as a necessary condition. In what follows, when we write $L_i$ or take the derivative w.r.t. $L_i$, $L$ is regarded as a vector (so $L_i=L(X_i)$ is the likelihood ratio at $X_i$). Note that the expression \eqref{eq: psi_truncation} is a polynomial of the vector $L$.
\begin{lem}
\label{lem: KKT}Let Assumptions \ref{assu: psi_expansion}, \ref{assu: phi_smoothness}
hold. Then with probability $1-o(1)$, we have that there exists $\tilde{\alpha}$,
$\beta$ and an optimal solution $L$ which satisfy the following conditions
\begin{gather}
\begin{array}{c}
\tilde{\alpha}n\frac{\partial}{\partial L_{i}}\psi(\hat{P}L)-\phi^{\prime}(L_{i})-\tilde{\alpha}\beta=0,i=1,2,\dots,n\\
\tilde{\alpha}>0\\
\hat{E}[\phi(L)]-\frac{q}{2n}=0\\
\hat{E}L-1=0
\end{array}\label{eq: KKT_alpha-1}
\end{gather}
\end{lem}

The KKT condition \eqref{eq: KKT_alpha-1} can be seen as a fixed point equation for the likelihood ratio $L$ and the Lagrange multipliers $\tilde{\alpha},\beta$. We can derive an expansion for $L,\tilde{\alpha},\beta$ based on \eqref{eq: KKT_alpha-1}. We first expand $L$ in terms of $\tilde{\alpha},\beta$ based on the first equation of \eqref{eq: KKT_alpha-1}. To do that, we rewrite the first equation of \eqref{eq: KKT_alpha-1} as 
\begin{equation}\label{eq: fixed_point}
L_i = 1+ h(\tilde{\alpha}\frac{\partial}{\partial L_i}n\psi(\hat{P}L)-\tilde{\alpha}\beta)
\end{equation}
where $h = (\phi^{\prime})^{-1}-1$. Since $\phi^{\prime}(1)=0$, we have $h(0)=0$.   By arguing that $\tilde{\alpha} = O_p (n^{-1/2})$ and $\beta = O_p(1)$, we can show that $L\rightarrow 1$. With the knowledge that $L\rightarrow 1$, the RHS of \eqref{eq: fixed_point} could be further expanded, which will give us a more accurate approximation to $L$. Iterating this procedure, we can get an expansion of $L$ in terms of $\tilde{\alpha}$ and $\beta$ up to the desired order. Once we have the expansion of $L$ in terms of $\tilde{\alpha}$ and $\beta$, we can plug it back to the third and fourth equation of \eqref{eq: KKT_alpha-1} and solve for an expansion of $\tilde{\alpha}$ and $\beta$ using the same strategy.  



\subsection{Technical Lemmas}\label{subsec: technical lemmas}

We will investigate an expansion for $L$ based on the system of equations \eqref{eq: KKT_alpha-1}. To make the computation rigorous, we need to properly define the residuals
of the expansion. In what follows, we use $\bar{O}_{p}^{(1/m)}(a_n)$
to denote (measurable) functions $\Lambda(\cdot)$ such that $\hat{E}\Lambda(X)^{m}=O_{p}(a_n)$.
An immediate observation is that
\begin{lem}\label{lem: residual_product} 
if $\Lambda_{1}=\bar{O}_{p}^{(1/m_{1})}(1)$
and $\Lambda_{2}=\bar{O}_{p}^{(1/m_{2})}(1)$, then $\Lambda_{1}\Lambda_{2}=\bar{O}_{p}^{(1/m)}(1)$
where $1/m=1/m_{1}+1/m_{2}$.
\end{lem} 
\begin{proof}
$\hat{E}\left(\Lambda_{1}\Lambda_{2}\right)^{m}=\hat{E}\left(\left(\Lambda_{1}^{m_{1}}\right)^{\frac{m}{m_{1}}}\left(\Lambda_{2}^{m_{2}}\right)^{\frac{m}{m_{2}}}\right)\leq\hat{E}\Lambda_{1}^{m_{1}}\hat{E}\Lambda_{2}^{m_{2}}=O_{p}(1)$
where the inequality follows from H\"older's inequality. 
\end{proof} 
By our Assumption \ref{assu: psi_expansion} part 2 on the influence functions, we can prove that
\begin{lem}
\label{lem: integrate_moment}Under Assumption \ref{assu: psi_expansion}, for
any function $\Lambda=\bar{O}_{p}^{(4/5)}(1)$ and $G(x)$ defined as $G(X):=\hat{E}_{Y}IF_{2}(X,Y;\hat{P})\Lambda(Y)$,
we have that $G=\bar{O}_{p}^{(1/5)}(1)$.
\end{lem}

\begin{lem}
\label{lem: Integrate_moment_2}Under Assumption \ref{assu: psi_expansion},
for any function $\Lambda_{i}(x)$ ($i=1,2$) such that $\Lambda_{i}(X)=O_{p}^{(4/5)}(1)$
and $G(x)$ defined as $G(X):=\hat{E}_{Y,Z}IF_{3}(X,Y,Z)\Lambda_{1}(Y)\Lambda_{2}(Z)$,
we have that $G=\bar{O}_{p}^{(1/5)}(1)$.
\end{lem}

The next lemma estimates the orders of $\tilde{\alpha}$ and $\beta$
as $n\rightarrow\infty$.
\begin{lem}
\label{lem: order_alpha_beta}Under Assumptions \ref{assu: psi_expansion} and \ref{assu: phi_smoothness},
$\tilde{\alpha}=O_{p}(n^{-1/2})$, $\tilde{\alpha}^{-1}=O_{p}(n^{1/2})$,
$\beta=O_{p}(1)$.
\end{lem}

The proofs of Lemmas \ref{lem: KKT}, \ref{lem: integrate_moment}, \ref{lem: Integrate_moment_2}, and \ref{lem: order_alpha_beta} are presented in Appendix \ref{sec: proof_technical}.

\subsection{Proof of Theorem \ref{thm: DRO_expansion}} \label{subsec: DRO_expansion_computation}

\paragraph*{Notations for this proof} To facilitate computation, we introduce some notations. For any vector $(K_{i})_{i=1,2,\dots,n}$, we denote \underline{$K_i$} as a random variable whose value is $K_i$ when $X=X_i$. In particular, this means $\underline{L_i}=L$ under $\hat{P}$.  
Corresponding to the notation introduced in the previous section , $\underline{\Lambda_i} = \bar{O}_{p}^{(1/m)}(n^{-3/2})$ means $\hat{E}\Lambda^m = \frac{1}{n}\sum_{i=1}^{n}\Lambda^{m}_i=O_{p}(n^{-3/2})$.
For the influence functions, we introduce the notation: $IF_{i}:=IF_{1}(X_{i};\hat{P})$,
$IF_{ij}:=IF_{2}(X_{i},X_{j};\hat{P})$, $IF_{ijk}:=IF_{3}(X_{i},X_{j},X_{k};\hat{P})$.
We sum over repeated subscripts (similar to Einstein's convention for summation), e.g., $IF_{ij}(L_{j}-1)=\sum_{j=1}^{n}IF_{ij}(L_{j}-1)$. However, $i$ is a special index: we will not sum over $i$ unless we explicitly write $\sum_i$, even if $i$ appears twice. In addition, the definitions of $\hat{\kappa}_2, \hat{\gamma}, \hat{\mu}_{2,a},\hat{\mu}_{2,b},\hat{\mu}_{2,c},\hat{\mu}_{3,c},\hat{\mu}_{4}$ are given as in the statement of this theorem. 

Our starting point is the necessary condition \eqref{eq: KKT_alpha-1}
for the optimal $L$, which holds with probability $1-o(1)$. The
computation is divided into three parts.

\subsubsection{Expansion of \texorpdfstring{$L$}{TEXT} in terms of \texorpdfstring{$\tilde{\alpha}$}{TEXT} and \texorpdfstring{$\beta$}{TEXT}}

From \eqref{eq: psi_truncation} we know that $\psi(\hat{P}L)$ can be written as a polynomial of vector $L$:
\[
\psi(\hat{P}L)= \psi(\hat{P}) + \frac{1}{n}\sum_i IF_i(L_i-1) + \frac{1}{2}\sum_{i,j}IF_{ij}(L_i-1)(L_j-1)+\frac{1}{6}\sum_{i,j,k}IF_{ijk}(L_i-1)(L_j-1)(L_k-1) 
\]
We let
\begin{equation}
D_{Li}:=\tilde{\alpha}\frac{\partial}{\partial L_{i}}n\psi(\hat{P}L)-\tilde{\alpha}\beta=\tilde{\alpha}\left(IF_{i}+\frac{1}{n}IF_{ij}(L_{j}-1)+\frac{1}{2n^{2}}IF_{ijk}(L_{j}-1)(L_{k}-1)\right)-\tilde{\alpha}\beta.\label{eq: DL_expression}
\end{equation}

By Taylor expansion of \eqref{eq: fixed_point} we know that
\begin{equation}
L_i=1+h^{\prime}(0)D_{Li}+\frac{1}{2}h^{\prime\prime}(0)D_{Li}^{2}+\frac{1}{6}h^{\prime\prime\prime}(0)D_{Li}^{3}+\frac{1}{24}h^{(4)}(\xi_{L})D_{Li}^{4}\label{eq: L_expand_iter}
\end{equation}
Here $\xi_L$ lies in the line segment between 0 and $\tilde{\alpha}\frac{\partial}{\partial L_i}n\psi(\hat{P}L)-\tilde{\alpha}\beta$. \eqref{eq: L_expand_iter} is the fixed point equation for $L$ that we will iterate on. Note that, since $L_i\in [c_0,C_0]$, by Assumption \ref{assu: phi_smoothness} we know that $h^{(4)}(\xi_{L})$ is bounded.

By Assumption \ref{assu: psi_expansion} part 2 and Lemmas \ref{lem: L_bound},
\ref{lem: integrate_moment}, and \ref{lem: Integrate_moment_2},
we have that
\begin{equation}
\begin{array}{c}
\underline{IF_{i}}=\bar{O}_{p}^{(1/5)}(1)\\
\underline{\frac{1}{n}IF_{ij}(L_{j}-1)}=n^{-1/2}\bar{O}_{p}^{(1/5)}(1)\\
\underline{\frac{1}{n^{2}}IF_{ijk}(L_{j}-1)(L_{k}-1)}=n^{-1}\bar{O}_{p}^{(1/5)}(1)
\end{array}\label{eq: O_bar_property}
\end{equation} Hence from Lemma \ref{lem: order_alpha_beta} and \eqref{eq: DL_expression}, we have that
\begin{equation}\label{eq: D_L_1st} 
\underline{D_{Li}}=\tilde{\alpha}\underline{(IF_{i}-\beta)}+\bar{O}_{p}^{(1/5)}(n^{-1})
\end{equation}
(which also implies $\underline{D_{Li}}=\bar{O}_{p}^{(1/5)}(n^{-1/2})$). Plugging this in \eqref{eq: L_expand_iter}, we get
\[
\underline{L_{i}}=1+h^{\prime}(0)\tilde{\alpha}\underline{(IF_{i}-\beta)}+n^{-1}\bar{O}_{p}^{(4/5)}(1).
\]
Then plugging in the above to the expression for $D_{L}$ in \eqref{eq: DL_expression},
and using Lemmas \ref{lem: integrate_moment} and \ref{lem: Integrate_moment_2} to bound the residual (notice that $\bar{O}_{p}^{(4/5)}(1)$ satisfies the
condition for Lemmas \ref{lem: integrate_moment} and \ref{lem: Integrate_moment_2}), we have that
\[
\underline{D_{Li}}=\tilde{\alpha}\underline{(IF_{i}-\beta)}+\frac{1}{n}\tilde{\alpha}h^{\prime}(0)\underline{IF_{ij}\tilde{\alpha}(IF_{j}-\beta)}+\bar{O}_{p}^{(1/5)}(n^{-3/2})
\]
Then by plugging this into \eqref{eq: L_expand_iter}, we have that
\[
\underline{L_{i}}=1+h^{\prime}(0)\tilde{\alpha}\underline{(IF_{i}-\beta)}+\frac{1}{n}h^{\prime}(0)\tilde{\alpha}^2\underline{IF_{ij}h^{\prime}(0)(IF_{j}-\beta)}+\frac{1}{2}h^{\prime\prime}(0)\left(\tilde{\alpha}\underline{(IF_{i}-\beta)}\right)^{2}+n^{-3/2}\bar{O}_{p}^{(4/5)}(1)
\]
Then, again by plugging the above into the expression for $D_{L}$
in \eqref{eq: DL_expression}, we have that 
\begin{align*}
\underline{D_{Li}}= & \tilde{\alpha}\underline{(IF_{i}-\beta)}+\frac{1}{n}\tilde{\alpha}^2h^{\prime}(0)\underline{IF_{ij}(IF_{j}-\beta)}+\tilde{\alpha}^3\frac{1}{n^{2}}h^{\prime}(0)^2\underline{IF_{ij}IF_{jk}(IF_{k}-\beta)}\\
 & +\frac{1}{n}\tilde{\alpha}^3\frac{1}{2}h^{\prime\prime}(0)\underline{IF_{ij}\left(IF_{j}-\beta\right)^{2}}+\frac{1}{2}\tilde{\alpha}^3h^{\prime}(0)^2\underline{IF_{ijk}(IF_{j}-\beta)(IF_{k}-\beta)}+n^{-2}\bar{O}_{p}^{(1/5)}(1)
\end{align*}
Then, once again by plugging the above equation to the expression for $L$
in \eqref{eq: L_expand_iter}, we have that
\begin{align}
\underline{L_{i}}= & 1+h^{\prime}(0)\tilde{\alpha}\underline{(IF_{i}-\beta)}+\frac{1}{n}h^{\prime}(0)\tilde{\alpha}^2h^{\prime}(0)\underline{IF_{ij}(IF_{j}-\beta)}+\frac{1}{2}h^{\prime\prime}(0)\left(\tilde{\alpha}\underline{(IF_{i}-\beta)}\right)^{2}\nonumber \\
 & +h^{\prime}(0)^3\tilde{\alpha}^3\frac{1}{n^{2}}\underline{IF_{ij}IF_{jk}(IF_{k}-\beta)}\nonumber \\
 & +h^{\prime}(0)\frac{1}{n}\tilde{\alpha}^3\frac{1}{2}h^{\prime\prime}(0)\underline{IF_{ij}\left(IF_{j}-\beta\right)^{2}}\nonumber \\
 & +\frac{1}{2}\tilde{\alpha}^3h^{\prime}(0)^3\underline{IF_{ijk}(IF_{j}-\beta)(IF_{k}-\beta)}\nonumber \\
 & +h^{\prime\prime}(0)\tilde{\alpha}^3\frac{1}{n}h^{\prime}(0)\underline{IF_{ij}(IF_{i}-\beta)(IF_{j}-\beta)}\nonumber \\
 & +\frac{1}{6}h^{\prime\prime\prime}(0)\tilde{\alpha}^3\left(\underline{IF_{i}-\beta}\right)^{3}+n^{-2}\bar{O}_{p}^{(4/5)}(1)\nonumber \\
=: & 1+\underline{L_{1i}}+\underline{L_{2i}}+\underline{L_{3i}}+n^{-2}\bar{O}_{p}^{(4/5)}(1).\label{eq: L_expand_alpha_beta}
\end{align}
Here, in the last step, we group the terms according to the degree of $\tilde{\alpha}$, i.e., $\underline{L_{1i}}=h^{\prime}(0)\tilde{\alpha}\underline{(IF_{i}-\beta)},\underline{L_{2i}}=\frac{1}{n}h^{\prime}(0)\tilde{\alpha}^2h^{\prime}(0)\underline{IF_{ij}(IF_{j}-\beta)}+\frac{1}{2}h^{\prime\prime}(0)\left(\tilde{\alpha}\underline{(IF_{i}-\beta)}\right)^{2}$
and $L_{3}$ stands for the rest of the terms (except the residual). As we mentioned earlier, we only sum over $j$ in expression $\frac{1}{n}h^{\prime}(0)\tilde{\alpha}^2h^{\prime}(0)\underline{IF_{ij}(IF_{j}-\beta)}+\frac{1}{2}h^{\prime\prime}(0)\left(\tilde{\alpha}\underline{(IF_{i}-\beta)}\right)^{2}$. From Assumption \ref{assu: psi_expansion} part 2 and Lemmas \ref{lem: integrate_moment}, \ref{lem: Integrate_moment_2},
we have that $\underline{L_{ki}}=n^{-k/2}\bar{O}_{p}^{(k/5)}(1)$. With a similar
procedure, the following lower-order expansions are also valid: $\underline{L_{i}}=1+\underline{L_{1i}}+\underline{L_{2i}}+n^{-3/2}\bar{O}_{p}^{(3/5)}(1)$, $\underline{L_{i}}=1+\underline{L_{1i}}+n^{-1}\bar{O}_{p}^{(2/5)}(1)$, and $\underline{L_{i}}=1+n^{-1/2}\bar{O}_{p}^{(1/5)}(1)$.

\subsubsection{Expansion of \texorpdfstring{$\beta$}{TEXT} in terms of \texorpdfstring{$\tilde{\alpha}$}{TEXT}}

Denote $L_k = \underline{L_{ki}},k=1,2,3$. From the expansion for $L$ given in \eqref{eq: L_expand_alpha_beta},
by taking expectation, we have that 
\[
\hat{E}L=1+\hat{E}L_{1}+\hat{E}L_{2}+\hat{E}L_{3}+O_{p}(n^{-2}).
\]
From the constraint, we have $\hat{E}L=1$, so we have that 
\[
\hat{E}L_{1}+\hat{E}L_{2}+\hat{E}L_{3}=O_{p}(n^{-2})
\]

By plugging in $L_{1}$ and a little algebra, from the above equation
we can conclude that (here we used $\frac{1}{\tilde{\alpha}}=O_{p}(n^{1/2})$
and multiplied $\tilde{\alpha}^{-1}$ on both sides)
\begin{equation}
\beta=\hat{E}IF_{1}(X,\hat{P})+\frac{1}{\tilde{\alpha}h^{\prime}(0)}\hat{E}L_{2}+\frac{1}{\tilde{\alpha}h^{\prime}(0)}\hat{E}L_{3}+O_{p}(n^{-3/2})\label{eq: beta_iter}
\end{equation}
The above equation would be the equation that we iterate on. Since $L_{k}=n^{-k/2}\bar{O}_{p}^{(k/5)}(1)$, we
have that $\hat{E}L_{k}=O_{p}(n^{-k/2})$ and hence from \eqref{eq: beta_iter}
we have that
\begin{equation}
\beta=\hat{E}IF_{1}+O_{p}(n^{-1/2})=O_{p}(n^{-1/2}).\label{eq: beta_approximation_1}
\end{equation}
Here the second equality holds since the marginal expectations of the influence functions are 0.

Next, we keep two terms in \eqref{eq: beta_iter} and write
\[
\beta=\hat{E}IF_{1}+\frac{1}{\tilde{\alpha}h^{\prime}(0)}\hat{E}L_{2}+O_{p}(n^{-1}).
\]
 Note that $\hat{E}L_2 = \frac{1}{n^2}\sum_{i,j}\tilde{\alpha}^2h^{\prime}(0)^2 IF_{ij}(IF_j-\beta)+ \frac{1}{2}h^{\prime\prime}(0)\tilde{\alpha}^2\hat{E}IF_{1}^{2}= \frac{1}{2}h^{\prime\prime}(0)\tilde{\alpha}^2\hat{E}IF_{1}^{2}$ where we used $\sum_i IF_{ij}=0$ in the last equality. Plugging this into the preceding equality and using the equation $\hat{E}IF_1(X,\hat{P})=0$, we get
\[
\beta=\frac{1}{2}\frac{h^{\prime\prime}(0)}{h^{\prime}(0)}\tilde{\alpha}\hat{\kappa}_2+O_{p}(n^{-1})
\]
Plugging this (and the expression for $L_k$) into the RHS of \eqref{eq: beta_iter}, we get that
\begin{align*}
\beta= & \frac{1}{2}\frac{h^{\prime\prime}(0)}{h^{\prime}(0)}\tilde{\alpha}\hat{\kappa}_2+h^{\prime\prime}(0)\tilde{\alpha}^{2}\hat{\mu}_{2,c}+\frac{1}{6}h^{\prime\prime\prime}(0)\tilde{\alpha}^{2}\hat{\gamma}+O_{p}(n^{-3/2}).
\end{align*}
Here, the definition of $\hat{\mu}_{2,c}$ can be found in the statement of this theorem. It can also be expressed as $\hat{\mu}_{2,c} = \frac{1}{n^2}\sum_{i,j}IF_{ij}IF_iIF_j$. The preceding equation is the desired expansion of $\beta$ in terms of $\tilde{\alpha}$.
Plugging this into the expression for $L$ in \eqref{eq: L_expand_alpha_beta},
we get an approximation for $L$ given by 
\begin{align}
L = \underline{L_{i}}= & 1+h^{\prime}(0)\tilde{\alpha}\underline{IF_{i}}+\frac{1}{n}\tilde{\alpha}^2h^{\prime}(0)^2IF_{ij}(IF_{j})+\frac{1}{2}h^{\prime\prime}(0)\left(\tilde{\alpha}\underline{IF_{i}}\right)^{2}\nonumber \\
 & -\tilde{\alpha}^2\frac{1}{2}{h^{\prime\prime}(0)}\hat{\kappa}_2-h^{\prime}(0)\tilde{\alpha}h^{\prime\prime}(0)\tilde{\alpha}^{2}\hat{\mu}_{2,c}\nonumber \\
 & -\tilde{\alpha}\frac{1}{6}h^{\prime\prime\prime}(0)\tilde{\alpha}^{2}\hat{\gamma} -h^{\prime\prime}(0)\tilde{\alpha}^{2}\underline{IF_{i}}\frac{1}{2}\frac{h^{\prime\prime}(0)}{h^{\prime}(0)}\tilde{\alpha}\hat{\kappa}_2 +h^{\prime}(0)^{3}\tilde{\alpha}^{3}\frac{1}{n^{2}}\underline{IF_{ij}IF_{jk}IF_{k}}\nonumber \\
 & +h^{\prime}(0)\frac{1}{n}\tilde{\alpha}^{3}\frac{1}{2}h^{\prime\prime}(0)\underline{IF_{ij}\left(IF_{j}\right)^{2}} +h^{\prime}(0)^{3}\frac{1}{6}\tilde{\alpha}^{3}\underline{IF_{ijk}IF_{j}IF_{k}}\nonumber \\
 & +h^{\prime\prime}(0)h^{\prime}(0)\tilde{\alpha}^{3}\frac{1}{n}\underline{IF_{ij}IF_{i}IF_{j}} +\frac{1}{6}h^{\prime\prime\prime}(0)\left(\tilde{\alpha}(\underline{IF_{i}})\right)^{3}+n^{-2}\bar{O}_{p}^{(4/5)}(1)\label{eq: L_expand_wrt_alpha}\\
=: & 1+L_{1}^{\prime}+L_{2}^{\prime}+L_{3}^{\prime}+n^{-2}\bar{O}_{p}^{(4/5)}(1).\nonumber 
\end{align}
Note that in the above expression, we only sum over $j$ when we write $IF_{ij}IF_iIF_j$. Here in the last step, we group the terms according to the degree of $\tilde{\alpha}$:
\[
L_{1i}^{\prime}=h^{\prime}(0)\tilde{\alpha}IF_{i}
\]
\[
L_{2i}^{\prime}=\frac{1}{n}h^{\prime}(0)\tilde{\alpha}IF_{ij}h^{\prime}(0)\tilde{\alpha}(IF_{j})+\frac{1}{2}h^{\prime\prime}(0)\left(\tilde{\alpha}IF_{i}\right)^{2}-\frac{1}{2}h^{\prime\prime}(0)\tilde{\alpha}^{2}\hat{E}IF_{1}^{2}
\]
and $L_{3}^{\prime}$ is the rest of terms in \eqref{eq: L_expand_wrt_alpha},
except for the residual. Similar to $L_{ki}$, we have that $\underline{L_{ki}^{\prime}}=n^{-k/2}\bar{O}_{p}^{(k/5)}(1)$.

\subsubsection{Asymptotic expansion of \texorpdfstring{$\tilde{\alpha}$}{TEXT}}\label{subsec: alpha_expansion}

From the boundedness of $\phi^{(5)}(\xi)$ when $\xi$ is bounded, and
that $L$ has expansion $1+n^{-1/2}\bar{O}_{p}^{(1/5)}(1)$, by Taylor
expansion we can show that

\begin{align*}
\hat{E}\phi(L)= & \frac{1}{2}\phi^{\prime\prime}(1)\hat{E}(L-1)^{2}+\frac{1}{6}\phi^{\prime\prime\prime}(1)\hat{E}(L-1)^{3}+\frac{1}{24}\phi^{(4)}(1)\hat{E}(L-1)^{4}+\hat{E}\phi^{(5)}(\xi)(L-1)^{5}\\
= & \frac{1}{2}\phi^{\prime\prime}(1)\hat{E}(L-1)^{2}+\frac{1}{6}\phi^{\prime\prime\prime}(1)\hat{E}(L-1)^{3}+\frac{1}{24}\phi^{(4)}(1)\hat{E}(L-1)^{4}+O_{p}(n^{-5/2})
\end{align*}
In the second inequality, we used that 
\[
\left|\hat{E}\phi^{(5)}(\xi)(L-1)^{5}\right|\leq\left(\sup_{x:\phi(x)\leq\frac{q}{2}}\left|\phi^{(5)}(x)\right|\right)\hat{E}|L-1|^{5}=\left(\sup_{x:\phi(x)\leq\frac{q}{2}}\left|\phi^{(5)}(x)\right|\right)\hat{E}|n^{-1/2}\bar{O}_{p}^{(1/5)}(1)|^{5}=O_{p}(n^{-5/2}).
\]
Plugging in the expansion for $L$ in \eqref{eq: L_expand_wrt_alpha},
we obtain that (proof given in Appendix \ref{sec: proof_technical})

\begin{equation}\label{eq: phi_L_expansion} 
\hat{E}\phi(L)=\frac{1}{2}\phi^{\prime\prime}(1)\hat{E}(L_{1}^{\prime2}+L_{2}^{\prime2}+2L_{1}^{\prime}L_{2}^{\prime}+2L_{1}^{\prime}L_{3}^{\prime})+\frac{1}{6}\phi^{\prime\prime\prime}(1)\hat{E}\left(L_{1}^{\prime3}+3L_{1}^{\prime2}L_{2}^{\prime}\right)+\frac{1}{24}\phi^{(4)}(1)\hat{E}L_{1}^{\prime4}+O_{p}(n^{-5/2})
\end{equation}

From the constraint $\hat{E}\phi(L)=\frac{q}{2n}$, we have
\begin{equation}
\frac{q}{2n}=\frac{1}{2}\phi^{\prime\prime}(1)\hat{E}(L_{1}^{\prime2}+L_{2}^{\prime2}+2L_{1}^{\prime}L_{2}^{\prime}+2L_{1}^{\prime}L_{3}^{\prime})+\frac{1}{6}\phi^{\prime\prime\prime}(1)\hat{E}\left(L_{1}^{\prime3}+3L_{1}^{\prime2}L_{2}^{\prime}\right)+\frac{1}{24}\phi^{(4)}(1)\hat{E}L_{1}^{\prime4}+O_{p}(n^{-5/2})\label{eq: alpha_iter}
\end{equation}

We will compute the asymptotic expansion of $\tilde{\alpha}$ based
on this equation. To begin, by taking the $O_p(n^{-1})$ part of \eqref{eq: alpha_iter}, we have that
\[
\frac{1}{2}\phi^{\prime\prime}(1)\hat{E}L_{1}^{\prime2}-\frac{q}{2n}=O_{p}(n^{-3/2}).
\]
Plugging in the expression for $L_1^{\prime}$, we get
\[
\frac{1}{2}\phi^{\prime\prime}(1)h^{\prime}(0)^{2}\tilde{\alpha}^{2}\hat{\kappa_2}-\frac{q}{2n}=O_{p}(n^{-3/2})
\]
which implies (note that $\tilde{\alpha}>0$ and by our assumption, $\frac{1}{\hat{E}IF_{1}^{2}}=O_{p}(1)$)
\begin{equation}
\tilde{\alpha}=\sqrt{\frac{q}{n\left(\phi^{\prime\prime}(1)h^{\prime}(0)^{2}\hat{\kappa}_2\right)}}+O_{p}(n^{-1})=\sqrt{\frac{q}{n\left(h^{\prime}(0)\hat{E}IF_{1}^{2}\right)}}+O_{p}(n^{-1}).\label{eq: alpha_app_1}
\end{equation}
Now we go to the next iteration. Taking the $O_p(n^{-3/2})$ part of \eqref{eq: alpha_iter}, we get
\[
\frac{1}{2}\phi^{\prime\prime}(1)\hat{E}\left(L_{1}^{\prime2}+2L_{1}^{\prime}L_{2}^{\prime}\right)+\frac{1}{6}\phi^{\prime\prime\prime}(1)\hat{E}L_{1}^{\prime3}-\frac{q}{2n}=O_{p}(n^{-2})
\]
which implies that
\begin{equation}\label{eq: alpha_equation_2nd}
\frac{1}{2}\phi^{\prime\prime}(1)h^{\prime}(0)^{2}\tilde{\alpha}^{2}\hat{E}IF_{1}^{2}=\frac{q}{2n}-\frac{1}{2}\phi^{\prime\prime}(1)\hat{E}\left(2L_{1}^{\prime}L_{2}^{\prime}\right)-\frac{1}{6}\phi^{\prime\prime\prime}(1)\hat{E}\left(L_{1}^{\prime}\right)^{3}+O_{p}(n^{-2})
\end{equation}
Plugging in the expression for $L_{1}^{\prime}$, $L_{2}^{\prime}$ and \eqref{eq: alpha_app_1} in the RHS, we obtain that
\[
\tilde{\alpha}=\sqrt{\frac{q}{n\left(\phi^{\prime\prime}(1)h^{\prime}(0)^{2}\hat{\kappa}_{2}\right)}\left(1-\sqrt{\frac{q}{n\left(h^{\prime}(0)\hat{\mu}_{2}\right)}}\frac{h^{\prime}(0)^{2}\hat{\mu}_{2,c}-\frac{1}{3}\frac{\phi^{\prime\prime\prime}}{\phi^{\prime\prime3}}\hat{\mu}_{3}}{\frac{1}{2}h^{\prime}(0)\hat{\mu}_{2}}\right)}+O_{p}(n^{-3/2})
\]
By Taylor expansion on the function $\sqrt{1+x}$ around $x=0$, we get 
\begin{equation}
\tilde{\alpha}=\sqrt{\frac{q}{n\left(h^{\prime}(0)\hat{\kappa}_2\right)}}-\frac{q}{n\left(h^{\prime}(0)\hat{\kappa}_{2}\right)}\frac{h^{\prime}(0)^{2}\hat{\mu}_{2,c}-\frac{1}{3}\frac{\phi^{\prime\prime\prime}}{\phi^{\prime\prime3}}\hat{\gamma}}{h^{\prime}(0)\hat{\kappa}_{2}}+O_{p}(n^{-3/2}).\label{eq: alpha_expansion}
\end{equation}
We will see that it suffices to compute
the expansion of $\tilde{\alpha}$ up to this order. Denote the residual above as $r$ (so we have $r=O_p(n^{-3/2})$) and write the above as $\tilde{\alpha}= \tilde{\alpha}_{0}+r$.

\subsubsection{Expansion of the optimal value of the maximization problem}\label{subsec: expansion_optimal}

By Assumption \ref{assu: psi_expansion} part 2, our expansion of $L$ given
in \eqref{eq: L_expand_wrt_alpha} and the estimation of order $L_{k}^{\prime}=n^{-k/2}\bar{O}_{p}^{(k/5)}(1)$,
we have that the optimal value is given by 
\begin{align*}
\psi(\hat{P}L)= & \psi(\hat{P}) + \frac{1}{n}\sum_i IF_i(L_i-1) + \frac{1}{2n^2}\sum_{i,j}IF_{ij}(L_i-1)(L_j-1)\\
& +\frac{1}{6n^3}\sum_{i,j,k}IF_{ijk}(L_i-1)(L_j-1)(L_k-1) \\
 = &\frac{1}{n}\sum_i IF_i(L_{1i}^\prime+L_{2i}^{\prime}+L_{3i}^{\prime}) + \frac{1}{2n^2}\sum_{i,j}IF_{ij}(L_{1i}^{\prime}L_{1j}^{\prime}+2L_{2i}^{\prime}L_{1j}^{\prime})\\
 &+\frac{1}{6n^3}\sum_{i,j,k}IF_{ijk}L_{1i}^{\prime}L_{1j}^{\prime}L_{1k}^{\prime} +O_{p}(n^{-2})
\end{align*}
Since we have the constraint $\hat{E}[\phi(L)]=\frac{q}{2n}$,
the optimal value is equal to 
\begin{align}
 & \frac{1}{n}\sum_{i}IF_{i}(L_{1i}^{\prime}+L_{2i}^{\prime}+L_{3i}^{\prime})+\frac{1}{2n^{2}}\sum_{i,j}IF_{ij}(L_{1i}^{\prime}L_{1j}^{\prime}+2L_{2i}^{\prime}L_{1j}^{\prime})\label{eq: plug_in_expression}\\
 & +\frac{1}{6n^{3}}\sum_{i,j,k}IF_{ijk}L_{1i}^{\prime}L_{1j}^{\prime}L_{1k}^{\prime}-\tilde{\alpha}_{0}^{-1}\left[\hat{E}[\phi(L)]-\frac{q}{2n}\right]+O_{p}(n^{-2})\nonumber 
\end{align}
Plugging in $\hat{E}[\phi(L)]$ as in \eqref{eq: alpha_iter}, we can see that the contribution of $r$ to the $O_p(n^{-3/2})$ term is 0. Indeed, $r$ contributes to $O_p(n^{-3/2})$ only through $\frac{1}{n}\sum_i IF_iL_{1i}^{\prime}-\alpha_0^{-1}\frac{1}{2}\phi^{\prime\prime}(1)\hat{E}L_1^{\prime2}$. Plugging in the expression for $L_{1i}^{\prime}$,we can see that the contribution is given by $\frac{1}{n}\sum_i [IF_ih^{\prime}(0)rIF_i-\alpha_0^{-1}\phi^{\prime\prime}(1)h^{\prime}(0)^2\tilde{\alpha}rIF_i^2] = O_p(n^{-2})$ (note that $h^{\prime}(0)=\phi^{\prime\prime}(1)^{-1}$). Therefore, in the above expression, we can plug in the expressions for $L_{ki}^{\prime}$ and plug in $\tilde{\alpha}$ as $\tilde{\alpha}_0$.   This will give us the desired result (the detailed plug-in step is given in Appendix \ref{sec: proof_technical}). 

\subsubsection{Expansion of the optimal value of the minimization problem}\label{subsec: relation_min_max}

To get the formula for the minimum optimal value, it suffices to convert the minimization problem to maximizing the negation of the objective. To get a more general observation (that also works if we want an expansion with residual $O_p(n^{-k/2})$ where $k>4$), we can study the relation between the minimization and maximization problems. We observe that, for
the KKT condition of the minimization problem, we just need to replace
the constraint $\tilde{\alpha}>0$ with $\tilde{\alpha}<0$. When
we compute the asymptotic expansion for $\tilde{\alpha}$ in Section \ref{subsec: alpha_expansion}, we are indeed solving $\tilde{\alpha}$ based on equation of the following form (WLOG, suppose that $\phi^{\prime\prime}(1)=1$)
\[
\hat{\kappa}_2\tilde{\alpha}^{2}+a_{1}\tilde{\alpha}^{3}+a_{2}\tilde{\alpha}^{4}+\cdots=\frac{q}{n}
\]
The LHS can be regarded as the square of some function of $\tilde{\alpha}$
that has an expansion
\[
f(\tilde{\alpha})=\sqrt{\hat{\kappa}_2}\tilde{\alpha}+b_{1}\tilde{\alpha}^{2}+b_{2}\tilde{\alpha}^{3}+\cdots
\]
Therefore, the equation for $\tilde{\alpha}$ could be written as 
\[
\tilde{\alpha}+b_{1}\tilde{\alpha}^{2}+b_{2}\tilde{\alpha}^{3}+\cdots=\pm\sqrt{\frac{q}{\hat{\kappa}_2n}}.
\]
Depending on whether the constraint is $\tilde{\alpha}>0$ or $\tilde{\alpha}<0$,
we will take the plus or minus sign.  Moreover, notice that for any $t$, $\tilde{\alpha}+b_{1}\tilde{\alpha}^{2}+b_{2}\tilde{\alpha}^{3}+\cdots=tn^{-1/2}$ gives an inverse expansion $\tilde{\alpha} = tn^{-1/2} + c_1t^2n^{-1}+c_2t^3n^{-3/2}+\cdots$. Therefore, depending on whether the constraint is $\tilde{\alpha}>0$ or $\tilde{\alpha}<0$, we have 
\[\tilde{\alpha} = \sqrt{\frac{q}{\hat{\kappa}_2n}} + c_1\frac{q}{\hat{\kappa}_2n} + c_2\left({\frac{q}{\hat{\kappa}_2n}}\right)^{3/2} +\cdots  \] or \[\tilde{\alpha} = -\sqrt{\frac{q}{\hat{\kappa}_2n}} + c_1\frac{q}{\hat{\kappa}_2n} - c_2\left({\frac{q}{\hat{\kappa}_2n}}\right)^{3/2} +\cdots.  \] Notice that after plugging in $L_{ki}^{\prime}$ in Section \ref{subsec: expansion_optimal}, we get a polynomial of $\tilde{\alpha}$. Therefore, the minimum value has the property that compared to the maximum value, we need to replace the coefficient of $n^{-k/2}$ with its negation if $k$ is odd.

\section{Technical Developments for Bartlett Correction}\label{subsec: Bartlett_technical}
\subsection{Stochastic Expansion with Tail Probability Bound}
As discussed in Section 2.7 of \cite{Hall1992}, to rigorously argue that the residual in a stochastic expansion (which has the form $O_p(n^{-k/2})$) only contributes $O(n^{-k/2})$ to the residual of the resulting probability expansion, we need extra conditions to bound the tail probability.

We introduce a stronger version of Assumption \ref{assu: psi_expansion}, which requires a tail probability bound and a higher-order expansion. For any deterministic sequence $y_n$, we write $X_n = O_{p,\delta}(y_n)$
if $X_n$ satisfies $P(X_{n}>My_n n^{\delta})=O(n^{-3/2})$
for some constant $M$. 

\begin{assumption}\label{assu: stronger_assu}
There exist functions $IF_j,j=1,2,3,4$ and $R_n = O_{p,\delta} (n^{-3})$ such that 
\begin{enumerate}
\item
\[
\left|\psi(\hat{P}L)  -\psi(\hat{P})-\sum_{j=1}^5\frac{1}{j!} E_{\Delta}IF_{j}(K_{1},K_{2},\dots,K_{j};\hat{P})\right|\leq R_n
\]
holds for all $L$ such that
$\hat{E}[\phi(L)]\leq\frac{q}{2n}$. Here, $E_{\Delta}$ denotes the expectation under which $K_1,K_2,\dots\stackrel{\text{i.i.d.}}{\sim} \hat{P}L-\hat{P}$ (signed measure). \\
\item For any $1\leq m\leq6$, $\hat{E}IF_{1}^{m}(K_1;\hat{P})=O_{p,\delta}(1)$,
$\hat{E}IF_{2}^{m}(K_{1},K_{2};\hat{P})=O_{p,\delta}(1),\hat{E}IF_{3}^{m}(K_{1},K_{2},K_{3};\hat{P})=O_{p,\delta}(1)$ and $\hat{E}IF_4^m(K_1,K_2,K_3,K_4;\hat{P})=O_{p,\delta}(1)$. Moreover, $\frac{1}{\widehat{Var}IF_{1}^{2}(K_1;\hat{P})}=O_{p,\delta}(1)$. Here, under $\hat{E}$ and $\widehat{Var}$ we have $K_1,K_2,\dots\stackrel{\text{i.i.d.}}{\sim} \hat{P}$. 
\item For $i=1,2,3,4$, $IF_{i}(K_{1},\dots,K_{i};\hat{P})$ is permutation invariant in $K_{1},\dots,K_{i}$, and has zero marginal expectations under $\hat{P}$:
\[
E_{K_{j}\sim \hat{P}}IF_{i}(K_{1},\dots,K_{i};\hat{P})=0,\ j=1,2,\dots,i.
\]
Here, $E_{K_j\sim\hat{P}}$ stands for the expectation under $K_j\sim\hat{P}$ with $K_s,s\neq j$ fixed.
\end{enumerate}
\end{assumption}

To verify Assumption \ref{assu: stronger_assu}, we may use Markov inequalities (Note that $EX^m<\infty$ implies $P(X>n^{\delta})\leq n^{-\delta m}EX^m = O(n^{-\delta m})$). For example, for the smooth function model, Assumption \ref{assu: stronger_assu} holds if the objective function is smooth and the underlying random variable has finite moments up to a sufficiently high order. For $V$-statistics that satisfy the condition in Proposition \ref{prop: V_stat}, since the finiteness of exponential order implies the finiteness of all moments, with the same proof as Proposition \ref{prop: V_stat} we can show that Assumption \ref{assu: stronger_assu} holds for any $\delta>0$.

For functions, correspondingly, we write $\bar{O}_{p,\delta}^{(1/m)}(1)$
as a function $\Lambda(X)$ such that $P(\hat{E}\Lambda(X)^{m}>\delta)=O(n^{-2})$. We observe that $O_{p,\delta}(1)\cdot O_{p,\delta}(1)=O_{p,2\delta}(1)$,
$O_{p,\delta}(1)+O_{p,\delta}(1)=O_{p,\delta}(1)$. Similar observation holds for $\bar{O}_{p,\delta}^{(1/m)}(1)$ terms. Moreover, note that in our computation in Section \ref{sec: DRO_technical}, we only need to perform multiplication for a finite number of times.
Hence, with the same argument as in Theorem \ref{thm: DRO_expansion}, we will get the following stronger result (see Section \ref{subsec: relation_min_max} for the relation between $\psi_{\min}$ and $\psi_{\max}$):
\[
\psi_{max}=\psi(\hat{P})+\sum_{i=1}^{4}n^{-k/2}\left(\frac{q}{\phi^{\prime\prime}(1)}\right)^{k/2}\hat{C}_{k}+O_{p,D\delta}(n^{-5/2}),
\]
\[
\psi_{min}=\psi(\hat{P})+\sum_{i=1}^{4}n^{-k/2}\left(-\frac{q}{\phi^{\prime\prime}(1)}\right)^{k/2}\hat{C}_{k}+O_{p,D\delta}(n^{-5/2}).
\]
Here, $D$ is a finite constant and the coefficients are of order $\hat{C}_{k}=O_{p,D\delta}(1),\hat{C}_{1}^{-1}=O_{p,D\delta}(1)$. The expression for $\hat{C_k},k=1,2,3$ are given as in Theorem \ref{thm: DRO_expansion}.  We will see that it is not necessary to derive the explicit formulas of $\hat{C_4}$.
Let  
\[
g(x)=\hat{C}_{1}x+\hat{C}_{2}x^{2}+\hat{C}_{3}x^{3}+\hat{C}_{4}x^{4},
\]
then $\psi_{max}=\psi(\hat{P})+g(\sqrt{\frac{q}{n\phi^{\prime\prime}(1)}})+O_{p,\delta}(n^{-3})$
and $\psi_{min}=\psi(\hat{P})+g(-\sqrt{\frac{q}{n\phi^{\prime\prime}(1)}})+O_{p,\delta}(n^{-3})$.
For $x$ close to 0, $g(x)=y$ implies an inverse expansion
of $x$ in terms of $y$ given by

\[
f(y):=\frac{y}{\hat{C}_{1}}+\frac{\hat{C}_{2}}{\hat{C}_{1}^{3}}y^{2}+\frac{1}{\hat{C}_{1}^{3}}\left(2\left(\frac{\hat{C}_{2}}{\hat{C}_{1}}\right)^{2}-\frac{\hat{C}_{3}}{\hat{C}_{1}}\right)y^{3}+K_{4}y^{4}
\]
(here $K_{4}$ is a function of $\hat{C}_{k},k=1,\cdots,4$
whose explicit formula is not given; each of the coefficients in the above polynomial is of order
$O_{p,D_1\delta}(1)$ for some $D_1<\infty$). 
We apply $f$ to each term in the expression $\psi_{\min}-\psi(\hat{P})\leq\psi(P_0)-\psi(\hat{P})\leq \psi_{\max}-\psi(\hat{P})$ and multiply $\sqrt{n}$. Then the coverage probability could be written as
\[
P\left(W_n\in[-\sqrt{\frac{q}{\phi^{\prime\prime}(1)}}+O_{p,D_1\delta}(n^{-2}),\sqrt{\frac{q}{\phi^{\prime\prime}(1)}}+O_{p,D_1\delta}(n^{-2})]\right)
\]
where \[
W_{n}:=\sqrt{n}\left(\begin{array}{c}
\frac{\psi(P_{0})-\psi(\hat{P})}{\hat{C}_{1}}+\frac{\hat{C}_{2}}{\hat{C}_{1}}\left(\frac{\psi(P_{0})-\psi(\hat{P})}{\hat{C}_{1}}\right)^{2}\\
+\left(2\left(\frac{\hat{C}_{2}}{\hat{C}_{1}}\right)^{2}-\frac{\hat{C}_{3}}{\hat{C}_{1}}\right)\left(\frac{\psi(P_{0})-\psi(\hat{P})}{\hat{C}_{1}}\right)^{3}\\
+K_{4}\left(\psi(P_{0})-\psi(\hat{P})\right)^{4}
\end{array}\right).
\]
Suppose that $D_1\delta<1/2$. Then by the definition of $O_{p,\delta}$ we have that $P(O_{p,D_1\delta}(n^{-2})>n^{-3/2})=O(n^{-3/2})$. Therefore, the above coverage probability is equal to
\begin{equation}\label{eq: coverage_W_n} 
P\left(W_n\in[-\sqrt{\frac{q}{\phi^{\prime\prime}(1)}}\pm n^{-3/2},\sqrt{\frac{q}{\phi^{\prime\prime}(1)}}\mp n^{-3/2}]\right) + O(n^{-3/2})
\end{equation}
\subsection{Edgeworth Expansion}

Our next task is to investigate the Edgeworth expansion for $W_n$. 
If $\psi$ is a function of vector mean, then $W_{n}$ is also a function
of a vector mean, so the result in \citet{bhattacharya1978} could
be applied. In general, however, $W_n$ could be a von Mises functional. As discussed in the literature review, there are some formal computations and lower-order results on the validity of Edgeworth expansion for von Mises functionals. However, to the best of our knowledge, a general high-order result is not available in the literature. The method in \citet{TAKAHASHI198856}, as the author discussed, could be generalized to higher order expansions. This is potentially promising, but
would require more formulation and detailed proofs, which is out of
the scope of this paper. Thus, here we assume in general the following, where the formulas for $p_1$ and $p_2$ are taken from (2.24) and (2.25) of \cite{Hall1992}, and the form of the expansion for $\kappa_{i,n}$ is due to Theorem 2.1 of \cite{Bhattacharya1983}:

\begin{assumption}\label{assu: Edgeworth}
\[
P(W_n\leq x) = \Phi(x) + n^{-1/2}p_1(x)\phi(x) + n^{-1}p_2(x)\phi(x)+O(n^{-3/2})
\]
where 
\[
p_1(x) = -\{k_{1,2}+\frac{1}{6}k_{3,1}(x^2-1)\}
\]
and
\[
p_2(x) = -x\left\{\frac{1}{2}(k_{2,2}+k_{1,2}^2)+\frac{1}{24}(k_{4,1}+4k_{1,2}k_{3,1})(x^2-3)+\frac{1}{72}k_{3,1}^2(x^4-10x^2+15)\right\}.
\]
Here, $k_{i,j},1\leq i\leq 4,1\leq j\leq 2$ are the coefficients in the cumulant expansion of $W_n$: the $i$-th order cumulant of $W_n$ has expansion $\kappa_{i,n}=n^{-(i-2)/2}(k_{i,1}+n^{-1}k_{i,2}+O(n^{-3/2})) $. 
\end{assumption}

With very extensive algebra (shown in Appendix \ref{sec:Computation_edgeworth}),
we can get the asymptotic expansion for the cumulants of $W_{n}$
with an error of $O(n^{-3/2})$. We also observe that  $\sqrt{n}K_{4}\left(\psi(P_{0})-\psi(\hat{P})\right)^{4}$
will only contribute $O(n^{-3/2})$ to the cumulant. Hence, it is not necessary to have the explicit expression of $K_4$. Plugging in the cumulants to the formula given in Assumption \ref{assu: Edgeworth}, we can get the expansion for the coverage probability. This gives

\begin{thm} \label{thm: Bartlett_full} 
Suppose that $D_1\delta<1/2$ and Assumptions \ref{assu: phi_smoothness}, \ref{assu: stronger_assu} and \ref{assu: Edgeworth} hold. Moreover, suppose that $\psi(\cdot)$ and its influence functions can be expanded in terms of influence functions around $P_0$:
\[
\psi(\hat{P})=\psi(P_{0})+\hat{E}IF_1(X,P_{0})+\frac{1}{2}\hat{E}IF_2(X,Y;P_{0})+\frac{1}{6}\hat{E}IF_3(X,Y,Z;P_{0})+n^{-2}R_0
\]
\begin{align*}
IF_1(X;\hat{P}) & =IF_1(X;P_0)+\hat{E}_{Y}IF_2(X,Y;\hat{P})+\frac{1}{2}\hat{E}_{Y,Z}IF_3(X,Y,Z;\hat{P})\\ & -\hat{E}IF_1(X;P_0)-\hat{E}IF_2(X,Y;P_0)-\frac{1}{2}\hat{E}IF_3(X,Y,Z;P_0)+n^{-3/2}R_1(X)
\end{align*}
\begin{align*} 
IF_2(X,Y;\hat{P})&=IF_2(X,Y;P_0)+\hat{E}_{Z}IF_3(X,Y,Z;P_0)\\& -\hat{E}_{X}IF_2(X,Y;P_0)-\hat{E}_{Y}IF_2(X,Y;P_0)+n^{-1}R_2(X,Y)
\end{align*}
\[
IF_3(X,Y,Z;\hat{P}) = IF_3(X,Y,Z;P_0) + n^{-1/2}R_3(X,Y,Z)
\]
where $R_i, IF_i, i=0,1,2,3$ satisfy $ER_i(X_1,\dots,X_i))^{12} = O(1)$  and $E(IF_i(X_1,\dots,X_i;P_0))^{12} = O(1)$. Here, $\hat{E}$ stands for the expectation under distribution $X,Y,Z\stackrel{\text{i.i.d.}}{\sim}\hat{P}$, and $\hat{E}_{Y,Z}$ stands for the expectation under $Y,Z\stackrel{\text{i.i.d.}}{\sim}\hat{P}$ with $X$ fixed.  $\hat{E}_{X},\hat{E}_{Y}$, and $\hat{E}_Z$ are defined similarly. $E$ stands for the expectation under distribution $X,Y,Z\stackrel{\text{i.i.d.}}{\sim}P_0$. Moreover, suppose that for $i=1,2,3$, $IF_{i}(K_{1},\dots,K_{i};P_0)$ is permutation invariant in $K_{1},\dots,K_{i}$, and has zero marginal expectations under $P_0$, i.e., $E_{K_{j}\sim P_0}IF_{i}(K_{1},\dots,K_{i};P_0)=0,\ j=1,2,\dots,i.$
Then the claim in Theorem \ref{thm: Bartlett} holds.
\end{thm} 

In Theorem \ref{thm: Bartlett_full}, we imposed a condition on the expansion of influence functions in terms of higher-order influence functions and the moments of the residuals. This assumption helps us make arguments like $EO_p(n^{-3/2})=O(n^{-3/2})$. For smooth function models, this follows from the relation among lower-order and higher-order derivatives (e.g., the second displayed condition corresponds to expanding the first-order derivative in terms of the second-order and third-order derivatives). For $V$-statistics, this follows from the relations between $E[h(X_1,\dots,X_T)|X_1,\dots,X_i]$ and $E[h(X_1,\dots,X_T)|X_1,\dots,X_{i+j}]$ for $1\leq i<i+j\leq 3$ (note that the latter can be seen as an influence function of the former). To rigorously verify the condition in Theorem \ref{thm: Bartlett_full}, we just need to check the moment of each term in the expansion. The detailed computation with Mathematica code that leads to the result in Theorem \ref{thm: Bartlett_full} is provided in Appendix \ref{sec:Computation_edgeworth}.

\section*{Acknowledgements}
We gratefully acknowledge support from the National Science Foundation under grants CAREER CMMI-1834710 and IIS-1849280.

\bibliographystyle{plainnat}
\bibliography{arxivbib}

\appendix

\section{Verification of Assumptions in Section \ref{subsec: Verification-of-Assumption}}\label{sec:verification proofs}
\begin{proof}[Proof of Proposition \ref{prop: function_of_mean}]
By Taylor theorem, the residual in Assumption \ref{assu: psi_expansion} can be written as 
\begin{equation}\label{eq: residual_function_mean}
\sum_{i_1,i_2,i_3,i_4}f_{i_1,i_2,i_3,i_4}(t\hat{E}Z+(1-t)\hat{E}LZ)(\hat{E}LZ-\hat{E}Z)_{i_1}\dots (\hat{E}LZ-\hat{E}Z)_{i_4}
\end{equation}
for some $t\in[0,1]$. 
By the law of large numbers, we know that $\hat{E}Z\stackrel{p}{\rightarrow} E_{P_0}Z$. We also have that for each $i=1,2,\dots,q$, $\left|\hat{E}LZ-\hat{E}Z\right|_i \leq \frac{1}{n}\sum_{j=1}^n \left|l_j-1\right|Z^{(j)}_i\leq \frac{1}{\sqrt{n}}\left\Vert L-1\right\Vert_2 \hat{E}Z^2_i\leq C_{\phi}\frac{1}{\sqrt{n}}\hat{E}Z^2_i=O_p(n^{-1/2})$. (Also notice that the bound $C_{\phi}\frac{1}{n}\hat{E}Z_i^2$ does not depend on $L$) From this we have $\hat{E}LZ\stackrel{p}{\rightarrow}E_{P_0}Z$ so by continuous mapping, $f_{i_1,i_2,i_3,i_4}(t\hat{E}Z+(1-t)\hat{E}LZ)\stackrel{p}{\rightarrow}f_{i_1,i_2,i_3,i_4}(E_0 Z)$. Thus $f_{i_1,i_2,i_3,i_4}(t\hat{E}Z+(1-t)\hat{E}LZ)=O_p(1)$. Combing this with the result that $\left|\hat{E}LZ-\hat{E}Z\right|_i=O_p (n^{-1/2})$, we have that the residual in \eqref{eq: residual_function_mean} is $O_p(n^{-2})$. Part 2 of Assumption \ref{assu: psi_expansion} could be verified by checking formula \eqref{eq: IF_function_mean} and the condition that $Z$ has finite moments up to the 15-th order. (The condition $\frac{1}{\widehat{Var}IF_{1}(X_1;\hat{P})}=O_{p}(1)$ follows from $\text{Var}_0 (\nabla f(E_0Z)^{\top}\cdot Z)>0$ as assumed in this proposition)
\end{proof}


 \begin{proof}[Proof of Proposition \ref{prop: V_stat}]
 By the assumption on $h$, we have that
\[
|\hat{E}[h(X_{1:T})|X_{a_{1}}=X^{(1)},\dots,X_{a_{k}}=X^{(k)}]|\leq\Lambda_{a_{1}}(X^{(1)})+\cdots+\Lambda_{a_{k}}(X^{(k)})+\sum_{j=1}^{T}\text{\ensuremath{\hat{E}\Lambda_{j}}}
\]
By this relation and the fact that the summation of finitely many
r.v. with finite exponential moments still has  finite exponential
moment, we can see that $IF_{k}(X^{(1)},X^{(2)},\dots,X^{(k)};\hat{P})$
has finite exponential moment for any $1\leq k\le T$, so part 2 of Assumption \ref{assu: psi_expansion} is verified (as in the previous example, the condition $\frac{1}{\widehat{Var}IF_{1}(X_1;\hat{P})}=O_{p}(1)$ follows from $\text{Var}_0g(X)>0$ as assumed in this proposition). 

Now we verify Assumption \ref{assu: psi_expansion} (1). We may write $\psi(\hat{P}L)$ as
\begin{align*}
\psi(\hat{P}L) & =E_{(X_1,X_2,\dots,X_T)\sim (\hat{P}L)^T} h(X_{1},\dots,X_{T})\\
 & =E_{(X_1,X_2,\dots,X_T)\sim [\hat{P}+(\hat{P}L-\hat{P})]^T} h(X_{1},\dots,X_{T})\\
 & =\sum_{k=1}^{T}\frac{1}{k!}\sum_{i_{1},i_{2},\dots,i_{k}\text{mutually different}}E_{X_{i_{1}},\dots,X_{i_{k}}\stackrel{\text{i.i.d.}}{\sim}\hat{P}L-\hat{P}}E_{X_{j}\stackrel{\text{i.i.d.}}{\sim}\hat{P}\text{\ for \ensuremath{j\notin\{i_{1},\dots,i_{k}\}}}}h(X_{1},\dots,X_{T})\\
 & = \sum_{k=1}^{T} \frac{1}{k!} E_{\Delta} IF_k(X_1,X_2,\dots,X_k;\hat{P})
\end{align*} 
Comparing this with the statement made in part 1 of Assumption \ref{assu: psi_expansion}, it suffices to show that 
\[
 \sum_{k=4}^{T} \frac{1}{k!} E_{\Delta} IF_k(X_1,X_2,\dots,X_k;\hat{P}) = O_p(n^{-2})
\]
This is true since by arguments similar to Lemmas \ref{lem: integrate_moment}-\ref{lem: Integrate_moment_2}, using Holder's inequality, the existence of the exponential moment of $IF_k$ together with the boundness of $\left\Vert L-1\right\Vert_2 $, we can show that $$E_{\Delta} IF_{k}(X_{1},X_{2},\dots,X_{k};\hat{P})=O_p(n^{-k/2})$$ for each $4\leq k\leq T$. 
 \end{proof}
\begin{proof}[Proof of Proposition \ref{prop: optimization}]
\sloppy From the Glivenko-Cantelli assumption, we have that $\sup_{x\in\mathcal{X}}\left|\hat{E}\ell(x,\xi)-E_{P_0}\ell(x,\xi)\right|\stackrel{a.s.}{\rightarrow}0$. For $\hat{P}L$, by Cauchy-Schwartz we have $\sup_{ x\in\mathcal{X},\hat{E}\phi(L)\leq q/2n}\left| E_{\hat{P}L}\ell(x,\xi)-E_{\hat{P}}\ell(x,\xi)\right| \leq \frac{1}{\sqrt{n}}\sup_{\hat{E}\phi(L)\leq q/2n}\left\Vert L-1 \right\Vert_2 \hat{E}\sup_{x\in\mathcal{X}}\ell^2(x,\xi)\leq \frac{1}{\sqrt{n}} C_{\phi}\hat{E}H(\xi)\stackrel{a.s.}{\rightarrow}0$. Then by the assumption of unique well seperated root, we get the consistency $x_{\hat{P}}^*\stackrel{a.s.}{\rightarrow} x_{P_0}^*$ and $x_{\hat{P}L}^*\stackrel{a.s.}{\rightarrow} x_{P_0}^*$ uniformly for all feasible $L$. Since we only need all of our claims hold in an asymptotic sense, we may focus on the case when both $x_{\hat{P}}^*$ and $x_{\hat{P}L}^*$ are in the $\delta$-neighborhood of $x_{P_0}^*$ (the probability that this happens would converge to 1 as $n$ goes to infinity).

Since $\ell(x,\xi)$ is differentiable w.r.t. $x$ and $x_{P_0}^*$ is in the interior of $\mathcal{X}$, we have that $x_{\hat{P}L}$ satisfies the first-order condition (the interchangeability of differentiation and expectation follows from the fact that $\hat{P}$ is discrete-supported)
\[
E_{\hat{P}L}\left[\ell_{x}(x;\xi)\right]=0.
\]
By Taylor's theorem, we have that 
\begin{equation}\label{eq: eg3_Taylor1}
E_{\hat{P}L}\left[\ell_{x}(x_{\hat{P}}^{*};\xi)+\ell_{xx}(\zeta(\xi),\xi)(x_{\hat{P}L}^{*}-x_{\hat{P}}^{*})\right]=0
\end{equation}
for some $\zeta(\epsilon)$ on the line segment between $x_{\hat{P}^*}$ and $x_{\hat{P}L}^*$. This gives us 
\begin{equation}
x_{\hat{P}L}^{*}-x_{\hat{P}}^{*}=-\frac{\left(E_{\hat{P}L}-E_{\hat{P}}\right)\ell_{x}(x_{\hat{P}}^{*};\xi)}{E_{\hat{P}L}\ell_{xx}(\zeta(\xi),\xi)}.\label{eq: opt_argmin}
\end{equation}
Here we used the relation $E_{\hat{P}}\ell_{x}(x_{\hat{P}}^{*};\xi)=0 $ which also follows from the first-order condition. The denominator of \eqref{eq: opt_argmin} we can write it as
\[
E_{\hat{P}L}\ell_{xx}(\zeta(\xi),\xi)=E_{\hat{P}}\ell_{xx}(\zeta(\xi),\xi)+\left(E_{\hat{P}L}-E_{\hat{P}}\right)\ell_{xx}(\zeta(\xi),\xi)
\]
Notice that $\zeta(\xi)$ is between $x_{\hat{P}L}^{*}$ and $x_{\hat{P}}^{*}$,
we have that it will converge (uniformly in $\xi$) to $x_{P}^{*}$
when $n$ is large. Therefore, with a probability approaching 1 as $n\rightarrow\infty$,
we have that 
\[
E_{\hat{P}L}\ell_{xx}(\zeta(\xi),\xi)= E_{\hat{P}}\ell_{xx}(\zeta(\xi),\xi)-\left(E_{\hat{P}L}-E_{\hat{P}}\right)\ell_{xx}(\zeta(\xi),\xi)\geq E_{\hat{P}}h(\xi)-\left(E_{\hat{P}L}-E_{\hat{P}}\right)H^{1/2}(\xi)
\]
Since $Eh(\xi)>0$, we have that $E_{\hat{P}}h(\xi)=\Omega_{p}(1)$.
Since $EH(\xi)<\infty$, we have that 
\begin{equation}\label{eg: eq3_H_bound}
\left(E_{\hat{P}L}-E_{\hat{P}}\right)H^{1/2}(\xi)=\hat{E}(L-1)H^{1/2}\leq\sqrt{\hat{E}(L-1)^{2}\hat{E}H}=O_{p}(n^{-1/2}).
\end{equation}
From the preceding two inequalities, we get
\[
E_{\hat{P}L}\ell_{xx}(\zeta(\xi),\xi)\geq E_{\hat{P}}h(\xi)-\left(E_{\hat{P}L}-E_{\hat{P}}\right)H^{1/2}(\xi)=\Omega_{p}(1).
\]
For the numerator of \eqref{eq: opt_argmin}, we know that 
\[
\left(E_{\hat{P}L}-E_{\hat{P}}\right)\left|\ell_{x}(x_{\hat{P}}^{*};\xi)\right|\leq\left(E_{\hat{P}L}-E_{\hat{P}}\right)H^{1/2}=O_{p}(n^{-1/2}).
\]
Here the second equality follows from \eqref{eg: eq3_H_bound}. As a result, from \eqref{eq: opt_argmin} we have that $x_{\hat{P}L}^{*}-x_{\hat{P}}^{*}=O_{p}(n^{-1/2})$.

By Taylor's theorem, we write the denominator of \eqref{eq: opt_argmin} as 
\[
E_{\hat{P}L}\ell_{xx}(x_{\hat{P}}^{*},\xi)+E_{\hat{P}L}\ell_{xxx}(\tilde{\zeta}(\xi),\xi)(\zeta(\xi)-x_{\hat{P}}^{*})
\]
From the relation $\left|\zeta(\xi)-x_{\hat{P}}^{*}\right|\leq x_{P_0}^{*}-x_{\hat{P}}^{*}=O_{p}(n^{-1/2})$,
we have that 
\[
E_{\hat{P}L}\ell_{xxx}(\tilde{\zeta}(\xi),\xi)(\zeta(\xi)-x_{\hat{P}}^{*})\leq\left(x_{P}^{*}-x_{\hat{P}}^{*}\right)E_{\hat{P}L}H^{1/2}=O_{p}(n^{-1/2}).
\]
Here the second equality follows from \eqref{eg: eq3_H_bound} and the relation $E_{\hat{P}}H^{1/2} = O_p(1)$ (by the weak law of large numbers). From the analysis above, replacing the numerator of \eqref{eq: opt_argmin} with its Taylor approximation, we get the following expansion 
\begin{equation}\label{eq: eg3_exp1}
x_{\hat{P}L}^{*}-x_{\hat{P}}^{*}=-\frac{\left(E_{\hat{P}L}-E_{\hat{P}}\right)\ell_{x}(x_{P}^{*};\xi)}{E_{\hat{P}}\ell_{xx}(x_{\hat{P}}^{*},\xi)}+O_{p}(n^{-1}).
\end{equation}
To get higher order expansions, we may further expand \eqref{eq: eg3_Taylor1} and get the equation
\[
E_{\hat{P}L}\left[\ell_{x}(x_{\hat{P}}^{*};\xi)+\ell_{xx}(x_{\hat{P}}^{*},\xi)(x_{\hat{P}L}^{*}-x_{\hat{P}}^{*})+\frac{1}{2}\ell_{xxx}(\zeta(\xi),\xi)(x_{\hat{P}L}^{*}-x_{\hat{P}}^{*})^{2}\right]=0.
\]
Then we would obtain that
\begin{align*}
x_{\hat{P}L}^{*}-x_{\hat{P}}^{*}&=-\frac{E_{\hat{P}L}\left[\ell_{x}(x_{\hat{P}}^{*};\xi)+\frac{1}{2}\ell_{xxx}(\zeta(\xi),\xi)(x_{\hat{P}L}^{*}-x_{\hat{P}}^{*})^{2}\right]}{E_{\hat{P}L}\ell_{xx}(x_{\hat{P}}^{*},\xi)}\\
&=-\frac{\left(E_{\hat{P}L}-E_{\hat{P}}\right)\ell_{x}(x_{\hat{P}}^{*};\xi)}{E_{\hat{P}L}\ell_{xx}(x_{\hat{P}}^{*},\xi)}-\frac{1}{2}\frac{E_{\hat{P}L}\ell_{xxx}(\zeta_{t},\xi)}{E_{\hat{P}L}\ell_{xx}(x_{P}^{*},\xi)}\left(\frac{\left(E_{\hat{P}L}-E_{\hat{P}}\right)\ell_{x}(x_{\hat{P}}^{*};\xi)}{E_{\hat{P}}\ell_{xx}(x_{\hat{P}}^{*},\xi)}+O_{p}(n^{-1})\right)^{2}\\
\end{align*}
Here the second equality follows by plugging in the expansion \eqref{eq: eg3_exp1} and noting that $E_{\hat{P}}\ell_{x}(x_{\hat{P}}^{*};\xi)=0 $. Based on the preceding displayed expression, with similar arguments following \eqref{eq: opt_argmin}, we have
\begin{align*}
x_{\hat{P}L}^{*}-x_{\hat{P}}^{*}=&-\frac{E_{\hat{P}L-\hat{P}}\ell_{x}(x_{\hat{P}}^{*};\xi)}{E_{\hat{P}}\ell_{xx}(x_{\hat{P}}^{*},\xi)}\left(1-\frac{E_{\hat{P}L-\hat{P}}\ell_{xx}(x_{\hat{P}}^{*};\xi)}{E_{\hat{P}}\ell_{xx}(x_{\hat{P}}^{*},\xi)}\right)\\ & -\frac{1}{2}\frac{E_{\hat{P}}\ell_{xxx}(x_{P}^{*},\xi)}{E_{\hat{P}}\ell_{xx}(x_{P}^{*},\xi)}\left(\frac{E_{\hat{P}L-\hat{P}}\ell_{x}(x_{\hat{P}}^{*};\xi)}{E_{\hat{P}}\ell_{xx}(x_{\hat{P}}^{*},\xi)}\right)^{2}+O_{p}(n^{-3/2}).
\end{align*}


With a similar discussion, we could get an expansion for $x_{\hat{P}L}^{*}-x_{\hat{P}}^{*}$ with residual $O_p(n^{-2})$. Then we will be able to get the expansion 
\begin{align*}
 & \psi(\hat{P}L)-\psi(\hat{P})\\
= & E_{\hat{P}L}\ell(x_{\hat{P}L}^{*};\xi)-E_{\hat{P}}\ell(x_{\hat{P}}^{*};\xi)\\
= & E_{\hat{P}L}\left[\ell(x_{\hat{P}}^{*};\xi)+\ell_{x}(x_{\hat{P}}^{*};\xi)\cdot(x_{\hat{P}L}^{*}-x_{\hat{P}}^{*})+\frac{1}{2}\ell_{xx}(x_{\hat{P}}^{*};\xi)\cdot(x_{\hat{P}L}^{*}-x_{\hat{P}}^{*})^{2}+\frac{1}{6}\ell_{xxx}(x_{\hat{P}}^{*};\xi)\cdot(x_{\hat{P}L}^{*}-x_{\hat{P}}^{*})^{3}\right]\\
 & -E_{\hat{P}}\ell(x_{\hat{P}}^{*};\xi)+ \frac{1}{4!}E_{\hat{P}L}\left[\ell_{xxxx}(\zeta_2(\xi);\xi)\right](x_{\hat{P}L}^{*}-x_{\hat{P}}^{*})^{4} 
\end{align*}
From the definition of $H$ we know that $\left|\ell_{x,x,x,x}(\zeta_2(\xi);\xi)\right| \leq E_{\hat{P}L}H^{1/2}(\xi)=O_p(1)$. Hence $E_{\hat{P}L}\left[\ell_{xxxx}(\zeta_2(\xi);\xi)\right](x_{\hat{P}L}^{*}-x_{\hat{P}}^{*})^{4} =O_{p}(n^{-2})$. Therefore, we get an expansion with residual $O_p(n^{-2})$ after plugging in the expansion for $x_{\hat{P}L}^{*}-x_{\hat{P}}^{*} $. This verifies part 1 of Assumption \ref{assu: psi_expansion}. 

To verify part 2 of Assumption \ref{assu: psi_expansion}, we check the expression for the influence function. For example,
the first order influence function is given as 
\[
IF_{1}(\xi;\hat{P})=\ell(x_{\hat{P}}^{*},\xi)-E_{\hat{P}}\ell(x_{\hat{P}}^{*},\xi).
\]
The second-order influence function
is 
\[
IF_{2}(\xi_{1},\xi_{2};\hat{P})=-\frac{2\ell_{x}(x_{\hat{P}}^{*},\xi_{1})\ell_{x}(x_{\hat{P}}^{*},\xi_{2})}{E_{\hat{P}}\ell_{xx}(x_{\hat{P}}^{*},\xi)}.
\]
From these expressions and the consistency $x_{\hat{P}}^{*}\rightarrow x_{P_0}^{*}$,
we can see that the second part of Assumption \ref{assu: psi_expansion}
follows from the finiteness of the moments of $\ell$
and its derivatives $w.r.t.$ $x$ in the neighborhood of $x_{P_0}^{*}$, which is exactly what we assumed for this proposition. The rigorous argument is analogous to the arguments in the earlier part of the proof which lead to the conclusion  $x_{\hat{P}L}^{*}-x_{\hat{P}}^{*}=O_{p}(n^{-1/2})$. 
\end{proof}

\section{Proofs of Technical Results}\label{sec: proof_technical}
\begin{proof}[Proof of Lemma \ref{lem: KKT}]
Notice that $L$ has discrete support, it is not hard to see that
the domain is compact and hence there exists a global optimizer. From
optimization theory, we know an optimal solution must be either a
solution to the KKT condition or a point that is not regular. 

For a point that satisfies the KKT condition, we have that there exists $\alpha,\beta$ such that
\begin{gather}\label{eq: proof_KKT}
\begin{array}{c}
\frac{\partial}{\partial L_{i}}\psi(\hat{P}L)-\frac{1}{n}\alpha\phi^{\prime}(L_{i})-\frac{1}{n}\beta=0,i=1,2,\dots,n\\
\hat{E}[\phi(L)]-\frac{q}{2n}\leq0,\alpha\geq0\\
\alpha\left(\hat{E}[\phi(L)]-\frac{q}{2n}\right)=0\\
\hat{E}L-1=0
\end{array}
\end{gather}
We argue that the probability that $\alpha=0$ goes to 0 as $n\rightarrow\infty$. Indeed,
if $\alpha=0$, then 
\begin{equation}\label{eq: KKT_proof}
\frac{\partial}{\partial L_{i}}\psi(\hat{P}L)=\frac{1}{n}\beta,i=1,2,\cdots,n
\end{equation}
From the expression for $\psi$ given in \eqref{eq: psi_truncation} and the definition of $E_{\Delta}$ given in Assumption \ref{assu: psi_expansion}, we have that (here $IF_{i}:=IF_{1}(X_{i};\hat{P}),IF_{ij}:=IF_{2}(X_{i},X_{j};\hat{P})$
and similar for $IF_{ijk}$, and we sum over repeated index, e.g., $IF_{ij}(L_j-1)=\sum_jIF_{ij}(L_j-1)$)
\[
\frac{\partial}{\partial L_{i}}\psi(\hat{P}L) = \frac{1}{n}IF_i + \frac{1}{n^2}IF_{ij}(L_j-1) + \frac{1}{2n^3}IF_{ijk}(L_j-1)(L_k-1).
\]
From the preceding two equations we have that
\[
IF_{i}-\beta=\frac{1}{n}IF_{ij}(L_{j}-1)+\frac{1}{2n^{2}}IF_{ijk}(L_{j}-1)(L_{k}-1)
\]
and taking the sum of squares, we get that
\begin{equation}
\frac{1}{n}\sum_{i}\left(IF_{i}-\frac{\sum_iIF_i}{n}\right)^{2}\leq\frac{1}{n}\sum_{i}\left(IF_{i}-\beta\right)^{2}\leq\sum_{i}\left(\frac{1}{n}IF_{ij}(L_{j}-1)+\frac{1}{2n^{2}}IF_{ijk}(L_{j}-1)(L_{k}-1)\right)^{2}.\label{eq: inequality_alpha_equal_0}
\end{equation}
However, the LHS is $\Omega_{p}(1)$ following our assumption that $\frac{1}{\hat{E}IF_1^2}=O_p(1) $
while the RHS is $O_{p}(n^{-1})$ (this can be shown by putting $K(X),K_{i}(X)$
in Lemmas \ref{lem: integrate_moment}, \ref{lem: Integrate_moment_2}
as $\sqrt{n}(L-1)$). Hence, the probability that LHS$\leq$RHS would
converge to 0 as $n\rightarrow\infty$, which means the probability that $\alpha=0$ goes to 0.

If a point is not regular, that means the derivatives of the two constraints of \eqref{eq: hatP_L_leading} are colinear, i.e., there exists $t$ such that $\phi^{\prime}(L_{i})=t,i=1,2,\dots,n$. Since $\phi^{\prime}$ is strictly increasing, we have that each element of vector $L$ is the same number. Since $\hat{E}L=1$, this means $L_{i}=1,i=1,2,\dots,n$. Therefore, for a nonregular point, the objective value is given by $\psi(\hat{P})$.  To show that $L=1$ can not
be the optimal solution, we do a local analysis. For any $i\neq j$ we can
perturb $L$ by adding and subtracting a same small amount for $L_{i}$
and $L_{j}$, and the objective would change by $\frac{\partial}{\partial L_{i}}\psi(\hat{P}L)-\frac{\partial}{\partial L_{j}}\psi(\hat{P}L)$.
If $L=1$ is optimal, then this amount must be 0 for all choice of
$i$ and $j$. Hence all of $\frac{\partial}{\partial L_{i}}\psi(\hat{P}L),i=1,2,\dots,n$
are the same number which will again lead to \eqref{eq: KKT_proof}.

Now we have shown that \eqref{eq: proof_KKT} holds with $\alpha>0$ (w.p. $1-o(1)$). Reformulating \eqref{eq: proof_KKT} by letting $\tilde{\alpha}=\alpha^{-1}$, we get the \eqref{eq: KKT_alpha-1}.
\end{proof}
We remark that it is also possible to enhance the probability in this
lemma to $1-O(n^{-2})$ by bounding the probability of \eqref{eq: inequality_alpha_equal_0},
e.g., by Markov's inequality and moment conditions.


\begin{proof}[Proof of Lemma \ref{lem: integrate_moment} and Lemma \ref{lem: Integrate_moment_2}]
We focus on Lemma \ref{lem: integrate_moment}. This follows from an application of H\"older's inequality. Indeed, 
\begin{align*}
\hat{E}G(X)^{5}= & \frac{1}{n}\sum_{i=1}^{n}G(X_{i})^{5}\\
= & \frac{1}{n}\sum_{i=1}^{n}\left(\frac{1}{n}\sum_{j=1}^{n}IF_{2}(X_{i},X_{j})K(X_{j})\right)^{5}\\
\text{(H\"older)}\leq & \frac{1}{n}\sum_{i=1}^{n}\left(\left(\frac{1}{n}\sum_{j=1}^{n}IF_{2}^{5}(X_{i},X_{j})\right)^{1/5}\left(\frac{1}{n}\sum_{j=1}^{n}K^{5/4}(X_{j})\right)^{4/5}\right)^{5}\\
= & \left(\frac{1}{n}\sum_{i=1}^{n}\left(\frac{1}{n}\sum_{j=1}^{n}IF_{2}^{5}(X_{i},X_{j})\right)^{1/5}\right)^{5}\left(\frac{1}{n}\sum_{j=1}^{n}K^{5/4}(X_{j})\right)^{4}\\
(\text{by convexity of \ensuremath{x^{5}}})\leq & \left(\frac{1}{n^{2}}\sum_{i,j=1}^{n}IF_{2}^{5}(X_{i},X_{j})\right)\left(\frac{1}{n}\sum_{j=1}^{n}K^{5/4}(X_{j})\right)^{4}\\
= & \hat{E}IF_{2}(X,Y)^{5}\left(\hat{E}K^{5/4}(X)\right)^{4}\\
\text{(by Assumption \ref{assu: psi_expansion} (2))}= & O_{p}(1)
\end{align*}

With a similar proof, we can show Lemma \ref{lem: Integrate_moment_2}.
\end{proof}


\begin{proof}[Proof of Lemma \ref{lem: order_alpha_beta}]
By the first equation in \eqref{eq: KKT_alpha-1}, we have
\begin{equation}\label{eq: alpha_order} 
\tilde{\alpha}\frac{\partial}{\partial L_{i}}n\psi(\hat{P}L)-\tilde{\alpha}\beta=\phi^{\prime}(L_{i})
\end{equation}
Since $\hat{E}\phi(L)\leq \frac{q}{2n}$, we have that $L_{i}\in[l_{\phi},u_{\phi}]$ where $l_{\phi},u_{\phi}$
are the two positive values such that $\phi(l_{\phi})=\phi(u_{\phi})=\frac{q}{2}$
(or $l_{\phi}=0$ if $\phi(0)$, $u_{\phi}<\infty$ since $\phi(1)=0$ and
the convexity of $\phi$ implies the coerciveness of $\phi(L_{i})$
as $L_{i}\rightarrow\infty$). By Taylor expansion on $\phi$ we have
\[
\phi^{\prime}(L_{i})=\phi^{\prime\prime}(\xi_{i})(L_{i}-1)
\]
where $\xi_{i}$ is between $1$ and $L_{i}$. Since $L_{i}\in[l_{\phi},u_{\phi}]$,
we have that $\phi^{\prime\prime}(\xi_{i})\in[c_{1},C_{1}]$
where $c_{1}>0$ and $C_{1}$ are constants determined by $\phi$
(and $q$). Hence taking squares in the preceding equality we get
\[
\phi^{\prime}(L_{i})^{2}\in[c_{2}(L_{i}-1)^{2},C_{2}(L_{i}-1)^{2}]
\]
Summing over $i$ and using Lemma \ref{lem: L_bound}, we will get that
\[
c_{3}\leq\sum_{i=1}^{n}\phi^{\prime}(L_{i})^{2}\leq C_{3}.
\]
Therefore from \eqref{eq: alpha_order} we get
\begin{equation}
c_{3}\leq\tilde{\alpha}^{2}\sum_{i=1}^{n}\left(\frac{\partial}{\partial L_{i}}n\psi(\hat{P}L)-\beta\right)^{2}\leq C_{3}.\label{eq: L2_D_phi}
\end{equation}
From \eqref{eq: fixed_point}, by Taylor expansion of $h$ around 0, we have that 
\[
L_{i}=1+h^{\prime}(0)(\tilde{\alpha}\frac{\partial}{\partial L_{i}}n\psi(\hat{P}L)-\tilde{\alpha}\beta)+\frac{1}{2}h^{\prime\prime}(\xi_{i})(\tilde{\alpha}\frac{\partial}{\partial L_{i}}n\psi(\hat{P}L)-\tilde{\alpha}\beta)^{2}
\]
Since $\tilde{\alpha}\frac{\partial}{\partial L_{i}}n\psi(\hat{P}L)-\tilde{\alpha}\beta$
has bounded range (implied by \eqref{eq: fixed_point} and the boundness of $L_i$), from Assumption \ref{assu: phi_smoothness} we have that there exists $M$ such that $|h^{\prime\prime}(\xi_{i})|\leq M$. Therefore, summing w.r.t. $i$ in the above display and using relation \eqref{eq: L2_D_phi},
we will get that
\[
\sum_{i=1}^{n}L_{i}=n+h^{\prime}(0)\sum_{i=1}^{n}(\tilde{\alpha}\frac{\partial}{\partial L_{i}}n\psi(\hat{P}L)-\tilde{\alpha}\beta)+R
\]
where $|R|<M_1$ for some deterministic $M_1$. Since $\hat{E}L=1$, we
have the LHS of the equation above is equal to $n$ and hence from
the above equation we conclude that
\[
\sum_{i=1}^{n}(\tilde{\alpha}\frac{\partial}{\partial L_{i}}n\psi(\hat{P}L)-\tilde{\alpha}\beta)=R/h^{\prime}(0)=R\phi^{\prime\prime}(1)
\]
Thus $\tilde{\alpha}\beta$ can be expressed as
\begin{equation}
\tilde{\alpha}\beta=\frac{1}{n}\sum_{i=1}^{n}\tilde{\alpha}\frac{\partial}{\partial L_{i}}n\psi(\hat{P}L)-\frac{1}{n}R\phi^{\prime\prime}(1)=\frac{1}{n}\sum_{i=1}^{n}\tilde{\alpha}\frac{\partial}{\partial L_{i}}n\psi(\hat{P}L)-O_{p}(n^{-1}).\label{eq: ab}
\end{equation}
Plugging this into \eqref{eq: L2_D_phi}, we obtain that
\[
c_{3}\leq\sum_{i=1}^{n}\left(\tilde{\alpha}\left(\frac{\partial}{\partial L_{i}}n\psi(\hat{P}L)-\frac{1}{n}\sum_{i=1}^{n}\frac{\partial}{\partial L_{i}}n\psi(\hat{P}L)\right)-O_{p}(n^{-1})\right)^{2}\leq C_{3}.
\]
Since the cross term is zero (because $\sum_{i=1}^{n}\left[\frac{\partial}{\partial L_{i}}n\psi(\hat{P}L)-\frac{1}{n}\sum_{i=1}^{n}\frac{\partial}{\partial L_{i}}n\psi(\hat{P}L)\right]=0$),
we have that 
\begin{equation}
c_{3}+O_{p}(n^{-2})\leq\tilde{\alpha}^{2}\sum_{i=1}^{n}\left(\frac{\partial}{\partial L_{i}}n\psi(\hat{P}L)-\frac{1}{n}\sum_{i=1}^{n}\frac{\partial}{\partial L_{i}}n\psi(\hat{P}L)\right)^{2}\leq C_{3}+O_{p}(n^{-2})\label{eq: alpha_asymptotic_proof}
\end{equation}
Since $\hat{E}(L-1)^{2}=n^{-1}O_{p}(1)$ which implies $L-1=n^{-1/2}\bar{O}_{p}^{(1/2)}(1)$,
by Assumption \ref{assu: psi_expansion} and Lemmas \ref{lem: integrate_moment},\ref{lem: Integrate_moment_2}, we have that 
\[
\frac{\partial}{\partial L_{i}}n\psi(\hat{P}L)-\frac{1}{n}\sum_{i=1}^{n}\frac{\partial}{\partial L_{i}}n\psi(\hat{P}L)=IF_{i}+\frac{1}{n}IF_{ij}(l_{j}-1)+\frac{1}{n^{2}}IF_{ijk}(l_{j}-1)(l_{k}-1)=IF_{i}+\bar{O}_{p}^{(1/5)}(n^{-1/2})
\]
(see the beginning of Section \ref{subsec: DRO_expansion_computation} for the definition of the residual term). Taking the sum of squares in the preceding equation, we get
\[
\sum_{i=1}^{n}\left(\frac{\partial}{\partial L_{i}}n\psi(\hat{P}L)-\frac{1}{n}\sum_{i=1}^{n}\frac{\partial}{\partial L_{i}}n\psi(\hat{P}L)\right)^{2}=n\left(\hat{E}IF_{1}^{2}+O_{p}(n^{-1/2})\right)
\]
Plugging the above equation into \eqref{eq: alpha_asymptotic_proof},
and after a little algebra, we will obtain
\[
\frac{c_{4}}{\hat{E}IF_{1}^{2}}+O_{p}(n^{-1/2})\leq n\tilde{\alpha}^{2}\leq\frac{C_{4}}{\hat{E}IF_{1}^{2}}+O_{p}(n^{-1/2}).
\]
By the condition on $IF_{1}$ in Assumption \ref{assu: psi_expansion}, from the above we get
\[
\tilde{\alpha}=O_{p}(n^{-1/2}),\tilde{\alpha}^{-1}=O_{p}(n^{1/2}).
\]
Then in (\ref{eq: ab}), multiplying $\tilde{\alpha}^{-1}$, we have
that 
\[
\beta=O_{p}(1).
\]
This finishes the proof.
\end{proof}

\begin{proof}[Proof of \eqref{eq: phi_L_expansion}]
First, we consider $\hat{E}(L-1)^{2}=\hat{E}(L_{1}^{\prime2}+L_{2}^{\prime2}+2L_{1}^{\prime}L_{2}^{\prime}+2L_{1}^{\prime}L_{3}^{\prime})+O_{p}(n^{-5/2})$.
A technical issue is that, if we simply plug in $L=1+L_1^\prime+L_2^{\prime}+L_3^\prime + n^{-2}\left(\bar{O}_{p}^{(4/5)}(1)\right)$, we would get that
\[
(L-1)^{2}=\left(L_{1}^{\prime}+L_{2}^{\prime}+L_{3}^{\prime}+n^{-2}\left(\bar{O}_{p}^{(4/5)}(1)\right)\right)^{2}
\]
This would involve a multiplication of two $\bar{O}_{p}^{(4/5)}(1)$
terms which would give $\bar{O}_{p}^{(8/5)}(1)$. Since $8/5>1$, it is not guaranteed
that its expectation under $\hat{P}$ is $O_p(1)$. To address this
issue, notice that by considering similar lower order expansions,
we know that the following relations hold:
\[
\begin{array}{c}
L=1+L_{1}^{\prime}+L_{2}^{\prime}+L_{3}^{\prime}+n^{-2}\bar{O}_{p}^{(4/5)}(1)\\
L=1+L_{1}^{\prime}+L_{2}^{\prime}+n^{-3/2}\bar{O}_{p}^{(3/5)}(1)\\
L=1+L_{1}^{\prime}+n^{-1}\bar{O}_{p}^{(2/5)}(1)\\
L=1+n^{-1/2}\bar{O}_{p}^{(1/5)}(1)
\end{array}
\]
Hence, we can write $(L-1)^{2}$ as
\begin{align*}
(L-1)^{2} & =(L_{1}^{\prime}+L_{2}^{\prime}+L_{3}^{\prime}+n^{-2}\bar{O}_{p}^{(4/5)}(1))(L-1)\\
 & =L_{1}^{\prime}(L-1)+L_{2}^{\prime}(L-1)+L_{3}^{\prime}(L-1)+n^{-2}\bar{O}_{p}^{(4/5)}(1)(L-1)\\
 & =L_{1}^{\prime}\left(L_{1}^{\prime}+L_{2}^{\prime}+L_{3}^{\prime}+n^{-2}\bar{O}_{p}^{(4/5)}(1)\right)+L_{2}^{\prime}\left(L_{1}^{\prime}+L_{2}^{\prime}+n^{-3/2}\bar{O}_{p}^{(3/5)}(1)\right)\\
 & +L_{3}^{\prime}\left(L_{1}^{\prime}+n^{-1}\bar{O}_{p}^{(2/5)}(1)\right)+n^{-2}\bar{O}_{p}^{(4/5)}(1)n^{-1/2}\bar{O}_{p}^{(1/5)}(1)\\
 & =L_{1}^{\prime}\left(L_{1}^{\prime}+L_{2}^{\prime}+L_{3}^{\prime}\right)+L_{2}^{\prime}\left(L_{1}^{\prime}+L_{2}^{\prime}\right)+L_{3}^{\prime}L_{1}^{\prime}+n^{-5/2}\bar{O}_{p}^{(1)}(1)
\end{align*}
In the last equation, we used that $L_{k}^{\prime}=n^{-k/2}\bar{O}_{p}^{(k/5)}(1)$.
Therefore, taking expectation, we get that $\hat{E}(L-1)^{2}=\hat{E}(L_{1}^{\prime2}+L_{2}^{\prime2}+2L_{1}^{\prime}L_{2}^{\prime}+2L_{1}^{\prime}L_{3}^{\prime})+O_{p}(n^{-5/2})$.

Next, we consider $(L-1)^{3}$. This can be handled in a similar way.
With similar computation as above, we may get that 
\[
(L-1)^{2}=L_{1}^{\prime}\left(L_{1}^{\prime}+L_{2}^{\prime}\right)+L_{2}^{\prime}L_{1}^{\prime}+n^{-2}\bar{O}_{p}^{(4/5)}(1).
\]
Then by regarding $(L-1)^{3}$ as the multiplication of $(L-1)^{2}$
and $(L-1)$, we can handle it with a similar procedure as above. The
residual is also given by $n^{-5/2}\bar{O}_{p}^{(1)}(1)$. $(L-1)^{4}$
can be handled in a similar way.
\end{proof}


\begin{proof}[The plug-in step in the end of the proof of Theorem \ref{thm: DRO_expansion}]
We consider the expression 
as given in \eqref{eq: plug_in_expression}.

Plugging in the expansion for $\hat{E}\phi(L)$, we get that \eqref{eq: plug_in_expression}
equals
\begin{align*}
 & \frac{1}{n}\sum_{i}IF_{i}(L_{1i}^{\prime}+L_{2i}^{\prime}+L_{3i}^{\prime})+\frac{1}{2n^{2}}\sum_{i,j}IF_{ij}(L_{1i}^{\prime}L_{1j}^{\prime}+2L_{2i}^{\prime}L_{1j}^{\prime})\\
 & +\frac{1}{6n^{3}}\sum_{i,j,k}IF_{ijk}L_{1i}^{\prime}L_{1j}^{\prime}L_{1k}^{\prime}-\tilde{\alpha}_{0}^{-1}\left[\begin{array}{c}
\frac{1}{2}\phi^{\prime\prime}(1)\hat{E}\left(L_{1}^{\prime2}+L_{2}^{\prime2}+2L_{1}^{\prime}L_{2}^{\prime}+2L_{1}^{\prime}L_{3}^{\prime}\right)-\frac{q}{2n}\\
+\frac{1}{6}\phi^{\prime\prime\prime}(1)\hat{E}\left(L_{1}^{\prime3}+3L_{1}^{\prime2}L_{2}^{\prime}\right)+\frac{1}{24}\phi^{(4)}(1)\hat{E}L_{1}^{4}
\end{array}\right]+O_{p}(n^{-2})
\end{align*}
Since $\tilde{\alpha}_0$ is found to satisfy \eqref{eq: alpha_equation_2nd}, we know that the term in the bracket above is $O_{p}(n^{-2})$. Therefore, we can replace $\tilde{\alpha}_{0}^{-1}$
above with its leading term, i.e., $\sqrt{\frac{nh^{\prime}(0)\hat{\kappa}_{2}}{q}}=\sqrt{\frac{n\hat{\kappa}_{2}}{\phi^{\prime\prime}(1)q}}$.
Therefore, it suffices to compute
\begin{align}
 & \frac{1}{n}\sum_{i}IF_{i}(L_{1i}^{\prime}+L_{2i}^{\prime}+L_{3i}^{\prime})+\frac{1}{2n^{2}}\sum_{i,j}IF_{ij}(L_{1i}^{\prime}L_{1j}^{\prime}+2L_{2i}^{\prime}L_{1j}^{\prime})\label{eq: plug_in_2}\\
 & +\frac{1}{6n^{3}}\sum_{i,j,k}IF_{ijk}L_{1i}^{\prime}L_{1j}^{\prime}L_{1k}^{\prime}-\sqrt{\frac{n\hat{\kappa}_{2}}{\phi^{\prime\prime}(1)q}}\left[\begin{array}{c}
\frac{1}{2}\phi^{\prime\prime}(1)\hat{E}\left(L_{1}^{\prime2}+L_{2}^{\prime2}+2L_{1}^{\prime}L_{2}^{\prime}+2L_{1}^{\prime}L_{3}^{\prime}\right)-\frac{q}{2n}\\
+\frac{1}{6}\phi^{\prime\prime\prime}(1)\hat{E}\left(L_{1}^{\prime3}+3L_{1}^{\prime2}L_{2}^{\prime}\right)+\frac{1}{24}\phi^{(4)}(1)\hat{E}L_{1}^{4}
\end{array}\right]+O_{p}(n^{-2})\nonumber 
\end{align}

 From the expressions for $L_{ki}^{\prime}$ and $\tilde{\alpha}_{0}$, we have that 
\begin{align*}
L_{1i}^{\prime} & =\frac{1}{\phi^{\prime\prime}(1)}IF_{i}\tilde{\alpha}_{0}\\
 & =\frac{1}{\phi^{\prime\prime}(1)}IF_{i}\left(\sqrt{\frac{q\phi^{\prime\prime}(1)}{n\hat{\kappa}_{2}}}-\frac{q\phi^{\prime\prime}(1)}{n\hat{\kappa}_{2}}\frac{\frac{1}{\phi^{\prime\prime}(1)^{2}}\hat{\mu}_{2,c}-\frac{1}{3}\frac{\phi^{\prime\prime\prime}(1)}{\phi^{\prime\prime}(1)^{3}}\hat{\gamma}}{\frac{1}{\phi^{\prime\prime}(1)}\hat{\kappa}_{2}}\right)\\
 & =IF_{i}\left(\sqrt{\frac{q}{\phi^{\prime\prime}(1)n\hat{\kappa}_{2}}}-\frac{q}{\phi^{\prime\prime}(1)n\hat{\kappa}_{2}}\frac{\hat{\mu}_{2,c}-\frac{1}{3}\frac{\phi^{\prime\prime\prime}(1)}{\phi^{\prime\prime}(1)}\hat{\gamma}}{\hat{\kappa}_{2}}\right)\\
 & =IF_{i}\left(\sqrt{\tilde{q}}-\tilde{q}\frac{\hat{\mu}_{2,c}-\frac{1}{3}\frac{\phi^{\prime\prime\prime}(1)}{\phi^{\prime\prime}(1)}\hat{\gamma}}{\hat{\kappa}_{2}}\right)
\end{align*}
Here, $\tilde{q}:=\frac{q}{\phi^{\prime\prime}(1)n\hat{\kappa}_{2}}$.
\begin{align*}
L_{2i}^{\prime}= & \tilde{\alpha}_{0}^{2}\left(\frac{1}{\phi^{\prime\prime}(1)^{2}}\frac{1}{n}IF_{ij}IF_{j}-\frac{1}{2}\frac{\phi^{\prime\prime\prime}(1)}{\phi^{\prime\prime}(1)^{3}}\left(IF_{i}^{2}-\hat{\kappa}_{2}\right)\right)\\
= & \left(\frac{q\phi^{\prime\prime}(1)}{n\hat{\kappa}_{2}}-2\sqrt{\frac{q\phi^{\prime\prime}(1)}{n\hat{\kappa}_{2}}}\frac{q\phi^{\prime\prime}(1)}{n\hat{\kappa}_{2}}\frac{\frac{1}{\phi^{\prime\prime}(1)^{2}}\hat{\mu}_{2,c}-\frac{1}{3}\frac{\phi^{\prime\prime\prime}(1)}{\phi^{\prime\prime}(1)^{3}}\hat{\gamma}}{\frac{1}{\phi^{\prime\prime}(1)}\hat{\kappa}_{2}}\right)\\
 & \times\left(\frac{1}{\phi^{\prime\prime}(1)^{2}}\frac{1}{n}IF_{ij}IF_{j}-\frac{1}{2}\frac{\phi^{\prime\prime\prime}(1)}{\phi^{\prime\prime}(1)^{3}}\left(IF_{i}^{2}-\hat{\kappa}_{2}\right)\right)+n^{-2}\bar{O}_{p}^{(4/5)}(1)\\
= & \left(\tilde{q}-2\tilde{q}^{3/2}\frac{\hat{\mu}_{2,c}-\frac{1}{3}\frac{\phi^{\prime\prime\prime}(1)}{\phi^{\prime\prime}(1)}\hat{\gamma}}{\hat{\kappa}_{2}}\right)\left(\frac{1}{n}IF_{ij}IF_{j}-\frac{1}{2}\frac{\phi^{\prime\prime\prime}(1)}{\phi^{\prime\prime}(1)}\left(IF_{i}^{2}-\hat{\kappa}_{2}\right)\right)+n^{-2}\bar{O}_{p}^{(4/5)}(1).
\end{align*}
Note that for the derivatives of $h$ in the expression of $L_{3i}^{\prime}$, from $h=(\phi^{\prime})^{-1}-1$ we have that $h^{\prime\prime}(0)=-\frac{\phi^{\prime\prime\prime}(1)}{(\phi^{\prime\prime}(1))^3},h^{\prime\prime\prime}(0)=-\frac{\phi^{(4)}(1)}{(\phi^{\prime\prime}(1))^4}+3\frac{(\phi^{\prime\prime\prime}(1))^2}{(\phi^{\prime\prime}(1))^5}$. Therefore,
\begin{align*}
L_{3i}^{\prime} & =\tilde{\alpha}^{3}\left(\begin{array}{c}
-\frac{\phi^{\prime\prime\prime}(1)}{\phi^{\prime\prime}(1)^{4}}\left(IF_{ij}IF_{i}IF_{j}-\hat{\mu}_{2,c}\right)-\frac{1}{6}\left(\frac{\phi^{(4)}(1)}{\phi^{\prime\prime}(1)^{4}}-\frac{\phi^{\prime\prime\prime}(1)^{2}}{\phi^{\prime\prime}(1)^{5}}\right)\left(IF_{i}^{3}-\hat{\gamma}\right)\\
-\frac{1}{2}\frac{\phi^{\prime\prime\prime}(1)^{2}}{\phi^{\prime\prime}(1)^{5}}IF_{i}\hat{\kappa}_{2}+\frac{1}{n^{2}\phi^{\prime\prime}(1)^{3}}IF_{ij}IF_{jk}IF_{k}+\frac{1}{6n^{2}\phi^{\prime\prime}(1)^{3}}IF_{ijk}IF_{j}IF_{k}\\
-\frac{1}{2}\frac{\phi^{\prime\prime\prime}(1)}{\phi^{\prime\prime}(1)^{4}}IF_{ij}\left(IF_{j}\right)^{2}
\end{array}\right)\\
 & =\tilde{q}^{3/2}\left(\begin{array}{c}
-\frac{\phi^{\prime\prime\prime}(1)}{\phi^{\prime\prime}(1)}\left(IF_{ij}IF_{i}IF_{j}-\hat{\mu}_{2,c}\right)-\frac{1}{6}\left(\frac{\phi^{(4)}(1)}{\phi^{\prime\prime}(1)}-\frac{\phi^{\prime\prime\prime}(1)^{2}}{\phi^{\prime\prime}(1)^{2}}\right)\left(IF_{i}^{3}-\hat{\gamma}\right)\\
-\frac{1}{2}\frac{\phi^{\prime\prime\prime}(1)^{2}}{\phi^{\prime\prime}(1)^{2}}IF_{i}\hat{\kappa}_{2}+\frac{1}{n^{2}}IF_{ij}IF_{jk}IF_{k}+\frac{1}{6n^{2}}IF_{ijk}IF_{j}IF_{k}\\
-\frac{1}{2}\frac{\phi^{\prime\prime\prime}(1)}{\phi^{\prime\prime}(1)}IF_{ij}\left(IF_{j}\right)^{2}
\end{array}\right).
\end{align*}
Again, recall that we do not sum over $i$ when write $IF_{ij}IF_{i}IF_{j}$.
Thus $IF_{ij}IF_{i}IF_{j}\neq\hat{\mu}_{2,c}$.

Next, we can plug in the expressions of $L_{ki}^{\prime}$ in each
part of \eqref{eq: plug_in_2}. We have that
\begin{align*}
 & \frac{1}{n}\sum_{i}IF_{i}(L_{1i}^{\prime}+L_{2i}^{\prime}+L_{3i}^{\prime})\\
= & \hat{\kappa}_{2}\left(\sqrt{\tilde{q}}-\tilde{q}\frac{\hat{\mu}_{2,c}-\frac{1}{3}\frac{\phi^{\prime\prime\prime}(1)}{\phi^{\prime\prime}(1)}\hat{\gamma}}{\hat{\kappa}_{2}}\right)+\left(\tilde{q}-2\tilde{q}^{3/2}\frac{\hat{\mu}_{2,c}-\frac{1}{3}\frac{\phi^{\prime\prime\prime}(1)}{\phi^{\prime\prime}(1)}\hat{\gamma}}{\hat{\kappa}_{2}}\right)\left(\hat{\mu}_{2,c}-\frac{1}{2}\frac{\phi^{\prime\prime\prime}(1)}{\phi^{\prime\prime}(1)}\hat{\gamma}\right)\\
 & +\tilde{q}^{3/2}\left(\begin{array}{c}
-\frac{\phi^{\prime\prime\prime}(1)}{\phi^{\prime\prime}(1)}\hat{\mu}_{2,b}-\frac{1}{6}\left(\frac{\phi^{(4)}(1)}{\phi^{\prime\prime}(1)}-\frac{\phi^{\prime\prime\prime}(1)^{2}}{\phi^{\prime\prime}(1)^{2}}\right)\hat{\mu}_{4}\\
-\frac{1}{2}\frac{\phi^{\prime\prime\prime}(1)^{2}}{\phi^{\prime\prime}(1)^{2}}\hat{\kappa}_{2}^{2}+\frac{1}{n^{2}}\hat{\mu}_{2,a}+\frac{1}{6n^{2}}\hat{\mu}_{3,c}-\frac{1}{2}\frac{\phi^{\prime\prime\prime}(1)}{\phi^{\prime\prime}(1)}\hat{\mu}_{2,b}
\end{array}\right)+O_{p}(n^{-2})\\
= & \tilde{q}^{1/2}\hat{\kappa}_{2}+\tilde{q}^{3/2}\left(\begin{array}{c}
-\frac{\phi^{\prime\prime\prime}(1)}{\phi^{\prime\prime}(1)}\hat{\mu}_{2,b}-\frac{1}{6}\left(\frac{\phi^{(4)}(1)}{\phi^{\prime\prime}(1)}-\frac{\phi^{\prime\prime\prime}(1)^{2}}{\phi^{\prime\prime}(1)^{2}}\right)\hat{\mu}_{4}\\
-\frac{1}{2}\frac{\phi^{\prime\prime\prime}(1)^{2}}{\phi^{\prime\prime}(1)^{2}}\hat{\kappa}_{2}^{2}+\hat{\mu}_{2,a}+\frac{1}{6n^{2}}\hat{\mu}_{3,c}-\frac{1}{2}\frac{\phi^{\prime\prime\prime}(1)}{\phi^{\prime\prime}(1)}\hat{\mu}_{2,b}\\
-2\frac{\hat{\mu}_{2,c}-\frac{1}{3}\frac{\phi^{\prime\prime\prime}(1)}{\phi^{\prime\prime}(1)}\hat{\gamma}}{\hat{\kappa}_{2}}\left(\hat{\mu}_{2,c}-\frac{1}{2}\frac{\phi^{\prime\prime\prime}(1)}{\phi^{\prime\prime}(1)}\hat{\gamma}\right)
\end{array}\right)+O_{p}(n^{-2}),
\end{align*}
\begin{align*}
 & \frac{1}{2n^{2}}\sum_{i,j}IF_{ij}(L_{1i}^{\prime}L_{1j}^{\prime}+2L_{2i}^{\prime}L_{1j}^{\prime})\\
= & \left(\tilde{q}-2\tilde{q}^{3/2}\frac{\hat{\mu}_{2,c}-\frac{1}{3}\frac{\phi^{\prime\prime\prime}(1)}{\phi^{\prime\prime}(1)}\hat{\gamma}}{\hat{\kappa}_{2}}\right)\frac{\hat{\mu}_{2,c}}{2}+\tilde{q}^{3/2}\left(\hat{\mu}_{2,a}-\frac{1}{2}\frac{\phi^{\prime\prime\prime}(1)}{\phi^{\prime\prime}(1)}\hat{\mu}_{2,b}\right)+O_{p}(n^{-2}),\\
= & \tilde{q}\frac{\hat{\mu}_{2,c}}{2}+\tilde{q}^{3/2}\left(\hat{\mu}_{2,a}-\frac{1}{2}\frac{\phi^{\prime\prime\prime}(1)}{\phi^{\prime\prime}(1)}\hat{\mu}_{2,b}-\frac{\hat{\mu}_{2,c}-\frac{1}{3}\frac{\phi^{\prime\prime\prime}(1)}{\phi^{\prime\prime}(1)}\hat{\gamma}}{\hat{\kappa}_{2}}\hat{\mu}_{2,c}\right),
\end{align*}
\[
\frac{1}{6n^{3}}\sum_{i,j,k}IF_{ijk}L_{1i}^{\prime}L_{1j}^{\prime}L_{1k}^{\prime}=\frac{1}{6}\tilde{q}^{3/2}\hat{\mu}_{3,c}+O_{p}(n^{-2}),
\]
\begin{align*}
 & \sqrt{\frac{n\hat{\kappa}_{2}}{\phi^{\prime\prime}(1)q}}\left[\frac{1}{2}\phi^{\prime\prime}(1)\hat{E}\left(L_{1}^{\prime2}+L_{2}^{\prime2}+2L_{1}^{\prime}L_{2}^{\prime}+2L_{1}^{\prime}L_{3}^{\prime}\right)-\frac{q}{2n}\right]\\
= & \tilde{q}^{-1/2}\left[\frac{1}{2}\hat{E}\left(L_{1}^{\prime2}+L_{2}^{\prime2}+2L_{1}^{\prime}L_{2}^{\prime}+2L_{1}^{\prime}L_{3}^{\prime}\right)-\frac{\tilde{q}}{2}\hat{\kappa}_{2}\right]\\
= & \begin{array}{c}
\frac{\tilde{q}^{1/2}}{2}\hat{\kappa}_{2}-\tilde{q}\left(\hat{\mu}_{2,c}-\frac{1}{3}\frac{\phi^{\prime\prime\prime}(1)}{\phi^{\prime\prime}(1)}\hat{\gamma}\right)+\frac{1}{2}\tilde{q}^{3/2}\left(\frac{\hat{\mu}_{2,c}-\frac{1}{3}\frac{\phi^{\prime\prime\prime}(1)}{\phi^{\prime\prime}(1)}\hat{\gamma}}{\hat{\kappa}_{2}}\right)^{2}\hat{\kappa}_{2}+\frac{1}{2}\tilde{q}^{3/2}\left(\hat{\mu}_{2,a}-\frac{\phi^{\prime\prime\prime}(1)}{\phi^{\prime\prime}(1)}\hat{\mu}_{2,b}+\frac{1}{4}\frac{\phi^{\prime\prime\prime}(1)^{2}}{\phi^{\prime\prime}(1)^{2}}\left(\hat{\mu}_{4}-\kappa_{2}^{2}\right)\right)\end{array}\\
 & +\left(\tilde{q}-3\tilde{q}^{3/2}\frac{\hat{\mu}_{2,c}-\frac{1}{3}\frac{\phi^{\prime\prime\prime}(1)}{\phi^{\prime\prime}(1)}\hat{\gamma}}{\hat{\kappa}_{2}}\right)\left(\hat{\mu}_{2,c}-\frac{1}{2}\frac{\phi^{\prime\prime\prime}(1)}{\phi^{\prime\prime}(1)}\hat{\gamma}\right)\\
 & +\tilde{q}^{3/2}\left(\begin{array}{c}
-\frac{\phi^{\prime\prime\prime}(1)}{\phi^{\prime\prime}(1)}\hat{\mu}_{2,b}-\frac{1}{6}\left(\frac{\phi^{(4)}(1)}{\phi^{\prime\prime}(1)}-\frac{\phi^{\prime\prime\prime}(1)^{2}}{\phi^{\prime\prime}(1)^{2}}\right)\hat{\mu}_{4}\\
-\frac{1}{2}\frac{\phi^{\prime\prime\prime}(1)^{2}}{\phi^{\prime\prime}(1)^{2}}\hat{\kappa}_{2}^{2}+\hat{\mu}_{2,a}+\frac{1}{6}\hat{\mu}_{3,c}-\frac{1}{2}\frac{\phi^{\prime\prime\prime}(1)}{\phi^{\prime\prime}(1)}\hat{\mu}_{2,b}
\end{array}\right)-\frac{\tilde{q}^{1/2}}{2}\hat{\kappa}_{2}+O_{p}(n^{-2})\\
= & -\tilde{q}\frac{1}{6}\frac{\phi^{\prime\prime\prime}(1)}{\phi^{\prime\prime}(1)}\hat{\gamma}+\tilde{q}^{3/2}\left(\begin{array}{c}
-2\frac{\phi^{\prime\prime\prime}(1)}{\phi^{\prime\prime}(1)}\hat{\mu}_{2,b}-\frac{1}{6}\left(\frac{\phi^{(4)}(1)}{\phi^{\prime\prime}(1)}-\frac{\phi^{\prime\prime\prime}(1)^{2}}{\phi^{\prime\prime}(1)^{2}}\right)\hat{\mu}_{4}-\frac{1}{2}\frac{\phi^{\prime\prime\prime}(1)^{2}}{\phi^{\prime\prime}(1)^{2}}\hat{\kappa}_{2}^{2}+\hat{\mu}_{2,a}\\
+\frac{1}{6}\hat{\mu}_{3,c}+\frac{1}{2}\hat{\mu}_{2,a}+\frac{1}{8}\frac{\phi^{\prime\prime\prime}(1)^{2}}{\phi^{\prime\prime}(1)^{2}}\left(\hat{\mu}_{4}-\kappa_{2}^{2}\right)\\
+\frac{1}{2}\left(\frac{\hat{\mu}_{2,c}-\frac{1}{3}\frac{\phi^{\prime\prime\prime}(1)}{\phi^{\prime\prime}(1)}\hat{\gamma}}{\hat{\kappa}_{2}}\right)\left(\hat{\mu}_{2,c}-\frac{1}{3}\frac{\phi^{\prime\prime\prime}(1)}{\phi^{\prime\prime}(1)}\hat{\gamma}-6\left(\hat{\mu}_{2,c}-\frac{1}{2}\frac{\phi^{\prime\prime\prime}(1)}{\phi^{\prime\prime}(1)}\hat{\gamma}\right)\right)
\end{array}\right)+O_{p}(n^{-2}),
\end{align*}
and
\begin{align*}
 & \sqrt{\frac{n\hat{\kappa}_{2}}{\phi^{\prime\prime}(1)q}}\left[\frac{1}{6}\phi^{\prime\prime\prime}(1)\hat{E}\left(L_{1}^{\prime3}+3L_{1}^{\prime2}L_{2}^{\prime}\right)+\frac{1}{24}\phi^{(4)}(1)\hat{E}L_{1}^{4}\right]\\
= & \tilde{q}^{-1/2}\left[\frac{1}{6}\frac{\phi^{\prime\prime\prime}(1)}{\phi^{\prime\prime}(1)}\hat{E}\left(L_{1}^{\prime3}+3L_{1}^{\prime2}L_{2}^{\prime}\right)+\frac{1}{24}\frac{\phi^{(4)}(1)}{\phi^{\prime\prime}(1)}\hat{E}L_{1}^{4}\right]\\
= & \frac{1}{6}\frac{\phi^{\prime\prime\prime}(1)}{\phi^{\prime\prime}(1)}\left(\tilde{q}-3\tilde{q}^{3/2}\frac{\hat{\mu}_{2,c}-\frac{1}{3}\frac{\phi^{\prime\prime\prime}(1)}{\phi^{\prime\prime}(1)}\hat{\gamma}}{\hat{\kappa}_{2}}\right)\hat{\gamma}+\frac{1}{6}\frac{\phi^{\prime\prime\prime}(1)}{\phi^{\prime\prime}(1)}3\tilde{q}^{3/2}\left(\hat{\mu}_{2,b}-\frac{1}{2}\frac{\phi^{\prime\prime\prime}(1)}{\phi^{\prime\prime}(1)}\left(\hat{\mu}_{4}-\hat{\kappa}_{2}^{2}\right)\right)\\
 & +\frac{1}{24}\frac{\phi^{(4)}(1)}{\phi^{\prime\prime}(1)}\tilde{q}^{3/2}\hat{\mu}_{4}+O_{p}(n^{-2})\\
= & \tilde{q}\frac{1}{6}\frac{\phi^{\prime\prime\prime}(1)}{\phi^{\prime\prime}(1)}\hat{\gamma}+\tilde{q}^{3/2}\left(-\frac{\hat{\mu}_{2,c}-\frac{1}{3}\frac{\phi^{\prime\prime\prime}(1)}{\phi^{\prime\prime}(1)}\hat{\gamma}}{\hat{\kappa}_{2}}\frac{\hat{\gamma}}{2}+\frac{1}{2}\frac{\phi^{\prime\prime\prime}(1)}{\phi^{\prime\prime}(1)}\hat{\mu}_{2,b}-\frac{1}{4}\frac{\phi^{\prime\prime\prime}(1)^{2}}{\phi^{\prime\prime}(1)^{2}}\left(\hat{\mu}_{4}-\hat{\kappa}_{2}^{2}\right)+\frac{1}{24}\frac{\phi^{(4)}(1)}{\phi^{\prime\prime}(1)}\hat{\mu}_{4}\right)+O_{p}(n^{-2}).
\end{align*}
Plugging in the preceding five equations, it is not hard to see that
(\ref{eq: plug_in_2}) is equal to the desired maximum optimal value
as claimed in Theorem \ref{thm: DRO_expansion}.
\end{proof}


\begin{proof}[Proof of Corollary \ref{cor: function_of_mean}]
WLOG, we suppose that $\mu_{0}=0$ (otherwise we just need to consider
$\tilde{\theta}(\mu):=\theta(\mu-\mu_{0})$ and $\tilde{X}=X-\mu_{0}$).
Then by Proposition 1, the influence functions are given by $IF_{k}(X_{1},\dots,X_{k};P_{0})=\theta_{i_{1}i_{2}\dots i_{k}}X_{1}^{i_{1}}\dots X_{k}^{i_{k}}$.
Therefore, 
\begin{align*}
 & \kappa_{2}=EIF_{1}^{2}(X;P_{0})=E\theta_{i}\theta_{j}X^{i}X^{j}=\alpha_{ij}\theta^{i}\theta^{j}=\delta_{ij}\theta^{i}\theta^{j}=\sum_{i}\left(\theta^{i}\right)^{2}=Q^{-1},
\end{align*}
Similarly,
\[
\gamma=EIF_{1}^{3}(X;P_{0})=\alpha_{ijk}\theta^{i}\theta^{j}\theta^{k},
\]
\[
\mu_{4}=\alpha_{ijkl}\theta^{i}\theta^{j}\theta^{k}\theta^{l},
\]
\[
\mu_{2,a}=E[X_{i}X_{j}\theta^{ij}X_{k}Z_{l}\theta^{kl}Y_{p}\theta^{p}Z_{u}\theta^{u}]=\alpha_{ik}\alpha_{jp}\alpha_{lu}\theta^{ij}\theta^{kl}\theta^{p}\theta^{u},
\]
\[
\mu_{2,b}=E[X_{i}Y_{j}\theta^{ij}X_{k}X_{l}\theta^{k}\theta^{l}Y_{p}\theta^{p}]=\alpha_{ikl}\delta_{jp}\theta^{ij}\theta^{k}\theta^{l}\theta^{p},
\]
\[
\mu_{2,c}=E[X_{i}Y_{j}\theta^{ij}X_{k}\theta^{k}Y_{l}\theta^{l}]=\delta_{ik}\delta_{jl}\theta^{ij}\theta^{k}\theta^{l},
\]
\[
\mu_{2,d}=E[X_{i}X_{j}\theta^{ij}]=\delta_{ij}\theta^{ij},
\]
\[
\mu_{2,2}=E[X_{i}Y_{j}\theta^{ij}X_{k}Y_{l}\theta^{kl}]=\delta_{ik}\delta_{jl}\theta^{ij}\theta^{kl},
\]
\[
\mu_{1,2,d}=E[X_{i}\theta^{i}X_{j}X_{k}\theta^{jk}]=\alpha_{ijk}\theta^{i}\theta^{jk}.
\]
From these and the definition of $t_{i}$, we can get that
\[
t_{1}=t_{2}=\gamma^{2}\kappa_{2}^{-3},t_{3}=\mu_{4}\kappa_{2}^{-2},
\]
\begin{align*}
t_{4} & =\alpha^{jkl}\theta^{j}\theta^{mn}\kappa_{2}^{-1}\left(\delta^{mk}-\theta^{m}\theta^{k}\kappa_{2}^{-1}\right)\left(\delta^{nl}-\theta^{n}\theta^{l}\kappa_{2}^{-1}\right)\\
 & =\mu_{1,2,d}\kappa_{2}^{-1}-2\mu_{2,b}\kappa_{2}^{-2}+\gamma\mu_{2,c}\kappa_{2}^{-3},
\end{align*}
and
\begin{align*}
t_{5} & =\kappa_{2}^{-1}\theta^{jk}\theta^{lm}(\delta^{jk}-\theta^{j}\theta^{k}\kappa_{2}^{-1})(\delta^{lm}-\theta^{l}\theta^{m}\kappa_{2}^{-1})-2\kappa_{2}^{-1}\theta^{jk}\theta^{lm}(\delta^{jl}-\theta^{j}\theta^{l}\kappa_{2}^{-1})(\delta^{km}-\theta^{k}\theta^{m}\kappa_{2}^{-1})\\
 & =\kappa_{2}^{-1}\left(\mu_{2,d}^{2}-2\mu_{2,c}\mu_{2,d}\kappa_{2}^{-1}+\mu_{2,c}^{2}\kappa_{2}^{-2}\right)-2\kappa_{2}^{-1}\left(\mu_{2,2}-2\mu_{2,a}\kappa_{2}^{-1}+\mu_{2,c}^{2}\kappa_{2}^{-2}\right).
\end{align*}
This gives
\begin{align*}
 & \frac{5}{3}t_{1}-2t_{2}+\frac{1}{2}t_{3}-t_{4}+\frac{1}{4}t_{5}\\
= & \frac{1}{36\kappa_{2}^{3}}\left(\begin{array}{c}
-12\gamma^{2}+18\kappa_{2}\mu_{4}+36\kappa_{2}(\mu_{2,a}+2\mu_{2,b})-36\gamma\mu_{2,c}\\
-9\mu_{2,c}^{2}-18\kappa_{2}\mu_{2,c}\mu_{2,d}+9\kappa_{2}^{2}(-2\mu_{2,2}+\mu_{2,d}^{2}-4\mu_{1,2,d})
\end{array}\right).
\end{align*}
It is not hard to verify that $-A(x)/x$ where $A(x)$ is given as in Theorem \ref{thm: Bartlett} with $\phi(x)=-\log x -x + 1$ (whose derivatives at 1 are given by $\phi^{\prime\prime}(1)=1,\phi^{\prime\prime\prime}(1)=-2,\phi^{(4)}(1)=6$) gives the same formula. Hence the corollary is proved. 

\end{proof}



\section{Computation for Edgeworth Expansion\label{sec:Computation_edgeworth}}





\paragraph*{Notations} \sloppy For ease of computation, we let $IF_{1}:=IF_{1}(X_{1};P_{0}),IF_{2}:=IF_{2}(X_{1},X_{2};P_{0})$
and $IF_{3}:=IF_{3}(X_{1},X_{2},X_{3};P_{0})$. Sometimes we only hide $P_0$ and write $IF_1(X):=IF_1(X;P_0), IF_2(X,Y):=IF_1(X,Y;P_0), IF_3(X,Y,Z):=IF_3(X,Y,Z;P_0)$.  Let $\hat{E}_{i}:=\hat{E}_{X_{i}}$ denote the operation of taking expectation for $X_i$ under distribution $\hat{P}$ (conditional on other variables), e.g., $\hat{E}_2IF_2 = \frac{1}{n}\sum_iIF_2(X_1,X_i;P_0)$. Similarly define $E_i:=E_{X_i}$. Let $\Delta_{i}:=\hat{E}_{i}-E_{i}$. $\hat{E}$ means taking expectation with the distribution $X_1,X_2,X_3\stackrel{\text{i.i.d.}}{\sim}\hat{P}$ (so we have $\hat{E}=\hat{E}_1\hat{E}_2\hat{E}_3$).

\subsection{Computation of the Cumulants of \texorpdfstring{$W_{n}$}{TEXT}}


Let $A_{n}:=\sqrt{n}\frac{\psi(P_{0})-\psi(\hat{P})}{\hat{C}_{1}}$,
then $W_{n}$ could be written as 
\[
W_{n}=A_{n}+n^{-1/2}\frac{\hat{C}_{2}}{\hat{C}_{1}}A_{n}^{2}+n^{-1}\left(2\left(\frac{\hat{C}_{2}}{\hat{C}_{1}}\right)^{2}-\frac{\hat{C}_{3}}{\hat{C}_{1}}\right)A_{n}^{3}+O_{p}(n^{-3/2})
\]
To compute the first four cumulants of $W_{n}$ with residual $O(n^{-3/2})$,
it suffices to compute the expectation of $A_{n},A_{n}^{2},A_{n}^{3},A_{n}^{4}$
with residual $O(n^{-3/2})$, and the expectation of $A_{n}^{i}\frac{\hat{C}_{2}}{\hat{C}_{1}},i=2,3,4,5$
with residual $O(n^{-1})$, and the expectation of $\left(2\left(\frac{\hat{C}_{2}}{\hat{C}_{1}}\right)^{2}-\frac{\hat{C}_{3}}{\hat{C}_{1}}\right)A_{n}^{j},j=3,4,5,6$
with residual $O(n^{-1/2})$. They are computed in each
of the following six subsections, respectively. We will prepare some useful intermediate results in the rest of this section.

From the expansions for $\psi(\hat{P})$, we have that
$$\psi(\hat{P})-\psi(P_{0})=\hat{E}IF_1+\frac{1}{2}\hat{E}IF_2+\frac{1}{6}\hat{E}IF_3+n^{-2}R_0$$

Next, we derive the expansion for $\frac{1}{\hat{C}_{1}}=\frac{1}{\widehat{Var}[IF(X;\hat{P})]^{-1/2}}$. Using the notation introduced in this section, the expansion for $IF_1$ introduced in Theorem \ref{thm: Bartlett_full} can be written as 

\begin{equation} \label{eq: IF_1_expansion}
IF_1(X_1;\hat{P})= IF_1-\hat{E}IF_1+\hat{E}_{2}IF_2-\hat{E}_1\hat{E}_2IF_2+\frac{1}{2}\hat{E}_{2}\hat{E}_3IF_3-\frac{1}{2}\hat{E}IF_3+n^{-3/2}R_1
\end{equation}
Therefore, its variance under $\hat{P}$ can be expanded as follows 

\begin{align*}
 & \widehat{Var}[IF(X;\hat{P})]\\
= & \hat{E}_1\left[IF(X_1;\hat{P})\right]^{2}\\
= & \hat{E}_1\left[\left[IF_1-\hat{E}IF_1\right]^{2}+\left(\hat{E}_{2}IF_{2}-\hat{E}_{1}\hat{E}_{2}IF_{2}\right)^{2}\right]+\hat{E}_1\left[\begin{array}{c}
2\left(IF_1-\hat{E}IF_1\right)\left(\hat{E}_{2}IF_{2}-\hat{E}_{1}\hat{E}_{2}IF_{2}\right)\\
+\left(IF_1-\hat{E}IF_1\right)\left(\hat{E}_{2}\hat{E}_3IF_{3}-\hat{E}IF_3\right)
\end{array}\right]+O_p(n^{-3/2})\\
\end{align*}
Noting that $\hat{E}_1\left[(IF_1-\hat{E}IF_1)\hat{E}_1\hat{E}_2IF_2\right]=\hat{E}_1\left[(IF_1-\hat{E}IF_1)\right]\hat{E}_1\hat{E}_2IF_2=0$ and similarly $\hat{E}_1\left[(IF_1-\hat{E}IF_1)\hat{E}IF_3\right]=0$, $\hat{E}_1\left[\hat{E}_1\hat{E}_{2}IF_{2}\left(\hat{E}_{2}IF_{2}-\hat{E}_{1}\hat{E}_{2}IF_{2}\right)\right]=0$, the RHS above can be written as 
\begin{align*}
&\hat{E}_1\left[\left[IF_1-\hat{E}IF_1\right]^{2}+\left(\hat{E}_{2}IF_{2}\right)^{2}\right]+\hat{E}_1\left[\begin{array}{c}
2\left(IF_1-\hat{E}IF_1\right)\left(\hat{E}_{2}IF_{2}\right)\\
+\left(IF_1-\hat{E}IF_1\right)\left(\hat{E}_{2}\hat{E}_3IF_{3}\right)
\end{array}\right]+O_p(n^{-3/2})
\end{align*}
Note that $\hat{E}IF_1=O_p(n^{-1/2}),\hat{E}IF_2=O_p(n^-1),\hat{E}IF_3=O_p(n^{-3/2})$ implies that $(\hat{E}_1\hat{E}_2IF_2)^2+\hat{E}_1IF_1\hat{E}_1\hat{E}_2IF_2+\hat{E}_1IF_1\hat{E}_1\hat{E}_2\hat{E}_3IF_3=O_p(n^{-2})$, we have that the above is equal to 
\[
EIF_1^{2}+\hat{E}_1IF_1^{2}+\hat{E}_{1}\left(\hat{E}_2IF_{2}\right)^{2}+2\left(\hat{E}_{1}\left(IF_1\hat{E}_2IF_{2}\right)\right) +\left(\hat{E}_{1}\left(IF_1\hat{E}_2\hat{E}_3IF_{3}\right)\right)-\left(\hat{E}_1IF_1\right)^{2}
\]

Since marginal expectations are 0 (see the end of Section \ref{subsec: technical lemmas}), we have $E_2IF_2=0,E_2IF_2=0,E_2IF_3=0,E_3IF_3=0$. Therefore, the above can be written as (recall that $\Delta_{i}:=\hat{E}_{i}-E_{i}$)
\begin{align*} 
 & EIF_1^{2}+\Delta_1IF_1^{2}+\hat{E}_{1}\left[\left(\Delta_2IF_{2}\right)^{2}\right]+2\left(\hat{E}_{1}\left[IF_1\Delta_2IF_{2}\right]\right) +\left(\hat{E}_{1}\left[IF_1\Delta_2\Delta_3IF_{3}\right]\right)-\left(\Delta_1IF_1\right)^{2}\\
=: & EIF_1^{2}+I_{1}+I_{2}+O_{p}(n^{-3/2})
\end{align*}
Here 
\[
I_{1}:=\Delta_{1}IF_{1}^{2}+2\left(\hat{E}_{1}\left[\Delta_{2}\left(IF_{1}IF_{2}\right)\right]\right),
\]
and 
\[
I_{2}:=\hat{E}_{1}\left(\Delta_{2}IF_{2}\right)^{2}+\hat{E}_{1}\Delta_{2}\Delta_{3}IF_{1}IF_{3}-\left(\Delta_{1}IF_{1}\right)^{2}.
\]
Moreover, it follows from the moment condition introduced in Theorem \ref{thm: Bartlett_full} that the residual term can be written as $n^{-3/2}S$ where $ES^4=O(1)$. 

The purpose of introducing the notation with $\Delta$ is to remind us that the order of a term is $O_p(n^{-k/2})$ where $k$ is the number of occurrence of $\Delta$. In particular, we can check that $I_i=O_p(n^{-i/2})$. Taking $(\cdot)^{-1/2}$, we get (let $\kappa_{2}:=EIF(X;P_{0})^{2}$
and note that $\kappa_{2}>0$)
\begin{equation}
\widehat{Var}[IF(X;\hat{P})]^{-1/2}=\kappa_{2}^{-1/2}\left(1+\frac{I_{1}}{\kappa_{2}}+\frac{I_{2}}{\kappa_{2}}\right)^{-1/2}+O_{p}(n^{-3/2})=\kappa_{2}^{-1/2}\left(1-\frac{I_{1}}{2\kappa_{2}}-\frac{I_{2}}{2\kappa_{2}}+\frac{3}{8}\frac{I_{1}}{\kappa_{2}^{2}}\right)+O_{p}(n^{-3/2}).\label{eq: expand_var}
\end{equation}

From a Taylor expansion argument together with the moment condition in Theorem \ref{thm: Bartlett_full}, we can show that the residual above can also be written as $n^{-3/2}S_1$ where $ES_1^4<\infty$. Then, from \eqref{eq: expand_var} and the expansion for
$\psi$, using the notation with $\Delta$, $A_{n}$ could be written
as 
\begin{align}
A_{n} & =\sqrt{n}\left(\psi(\hat{P})-\psi(P_{0})\right)\kappa_{2}^{-1/2}\left(1-\frac{I_{1}}{2\kappa_{2}}-\frac{I_{2}}{2\kappa_{2}}+\frac{3}{8}\frac{I_{1}}{\kappa_{2}^{2}}\right)+O_{p}(n^{-3/2})\nonumber\\
 & =\sqrt{n}\kappa_{2}^{-1/2}\left(\Delta_{1}IF_{1}+\frac{1}{2}\Delta_{1}\Delta_{2}IF_{2}+\frac{1}{6}\Delta_{1}\Delta_{2}\Delta_{3}IF_{3}\right)\left(1-\frac{I_{1}}{2\kappa_{2}}-\frac{I_{2}}{2\kappa_{2}}+\frac{3}{8}\frac{I_{1}^{2}}{\kappa_{2}^{2}}\right)+O_{p}(n^{-3/2})\label{eq: A_n_expression}
\end{align}
Here, the residual can be written as $n^{-3/2}S_2$ where $ES_2^4<\infty$. Hence, for computing the moments of $A_n$ with residual $O(n^{-3/2})$, the residual does not have contribution.

\subsubsection{First moment of \texorpdfstring{$A_{n}$}{TEXT}}

From \eqref{eq: A_n_expression} we have that
\begin{align*}
\frac{EA_{n}}{\sqrt{n}\kappa_{2}^{-1/2}} & =E\left[\Delta_{1}IF_{1}\left(-\frac{I_{1}}{2\kappa_{2}}-\frac{I_{2}}{2\kappa_{2}}+\frac{3}{8}\frac{I_{1}^{2}}{\kappa_{2}^{2}}\right)\right]+E\left[\frac{1}{2}\Delta_{1}\Delta_{2}IF_{2}\right]-\frac{1}{4\kappa_{2}}E[\Delta_{1}\Delta_{2}\left(IF_{2}\right)I_{1}]\\
 & +E\left[\frac{1}{6}\Delta_{1}\Delta_{2}\Delta_{3}IF_{3}\right]+O(n^{-2})
\end{align*}
We observe that, after taking expectation, the
$\Theta(n^{-3/2})$ part in the expansion will only contribute $O(n^{-2})$ (for example, we can check that $E\left(\frac{1}{6}\Delta_1\Delta_2\Delta_3IF_3\right)=O(n^{-2})$ instead of just $O(n^{-3/2})$). Therefore, we do not need to take these terms into consideration when computing $EA_n$. Hence
\begin{align*}
\frac{EA_{n}}{\sqrt{n}\kappa_{2}^{-1/2}} & =E\left[\Delta_{1}IF_{1}\left(-\frac{I_{1}}{2\kappa_{2}}\right)\right]+E\left[\frac{1}{2}\Delta_{1}\Delta_{2}IF_{2}\right]+O(n^{-2})\\
 & =-\frac{1}{2\kappa_{2}}E_{1}[\Delta_1IF_{1}\Delta_1IF_1^2]-\frac{1}{\kappa_{2}}E[\Delta_{1}IF_{1}\hat{E}_{1}\Delta_{2}\left(IF_{1}IF_{2}\right)]+E\left[\frac{1}{2}\Delta_{1}\Delta_{2}IF_{2}\right]+O(n^{-2})\\
 & =-\frac{1}{2\kappa_{2}}\frac{1}{n}E_{1}IF_{1}^{3}-\frac{1}{\kappa_{2}}\frac{1}{n}EIF_{1}(X_{1})IF_{1}(X_{2})IF_{2}(X_{1},X_{2})+\frac{1}{2n}E\left[IF_{2}(X,X)\right]+O(n^{-2}).
\end{align*}
To show the last equality, we show that $E[\Delta_{1}IF_{1}\hat{E}_{1}\Delta_{2}\left(IF_{1}IF_{2}\right)] = \frac{1}{n}EIF_{1}(X_{1})IF_{1}(X_{2})IF_{2}(X_{1},X_{2}) +O(n^{-2})$ as an example. Indeed, we have that 
\begin{align}
 & E\left[\Delta_{1}IF_{1}\hat{E}_{1}\Delta_{2}\left(IF_{1}IF_{2}\right)\right]\label{eq: computation_eg}\\
 & \text{(note that \ensuremath{EIF_{1}=0,E_{2}IF_{1}IF_{2}=0})}\nonumber \\
= & E\left[\frac{\sum_{i=1}^{n}IF_{1}(X_{i})}{n}\frac{\sum_{1\leq j,k\leq n}IF_{1}(X_{j})IF_{2}(X_{j},X_{k})}{n^{2}}\right]\nonumber \\
= & \frac{1}{n^{3}}\sum_{1\leq i,j,k\leq n}E\left[IF_{1}(X_{i})IF_{1}(X_{j})IF_{2}(X_{j},X_{k})\right]\nonumber 
\end{align}
Since the marginal expectations of influence functions are 0, we have
that $EIF_{1}(X_{i})IF_{1}(X_{j})IF_{2}(X_{j},X_{k})\neq0$
only if $k=i$. Therefore, the RHS above is given by $\frac{1}{n}EIF_{1}(X_{1})IF_{1}(X_{2})IF_{2}(X_{1},X_{2})+O(n^{-2})$
as claimed. Here the residual is due to the terms with $i=j=k$, which
have a different expectation but only contribute $O(n^{-2})$ to the
summation.

Recall the notations: $\gamma=EIF_{1}^{3}$, $\mu_{2,c}:=EIF_{1}(X_{1})IF_{1}(X_{2})IF_{2}(X_{1},X_{2})$,
$\mu_{2,d}:=EIF_{2}(X,X)$, we get that
\[
EA_{n}=n^{-1/2}\left[-\frac{\kappa_{2}^{-3/2}}{2}\gamma-\kappa_{2}^{-3/2}\mu_{2,c}+\frac{1}{2}\kappa_{2}^{-1/2}\mu_{2,d}\right] + O(n^{-3/2})
\]

\subsubsection{Second moment of \texorpdfstring{$A_{n}$}{TEXT}}

Again, we start from \eqref{eq: A_n_expression}. Taking squares, we get that
\begin{align*}
A_{n}^{2}= & n\kappa_{2}^{-1}(\left(\Delta_{1}IF_{1}\right)^{2}+\Delta_{1}IF_{1}\Delta_{1}\Delta_{2}IF_{2}+\frac{1}{4}\left(\Delta_{1}\Delta_{2}IF_{2}\right)^{2}+\frac{1}{3}\Delta_{1}IF_{1}\Delta_{1}\Delta_{2}\Delta_{3}IF_{3})\left(1-\frac{I_{1}}{\kappa_{2}}-\frac{I_{2}}{\kappa_{2}}+\frac{I_{1}^{2}}{\kappa_{2}^{2}}\right)\\& +O_{p}(n^{-3/2})
\end{align*}
Hence plugging in $I_{1},I_{2}$, we get
\begin{align*}
\frac{1}{n\kappa_{2}^{-1}}EA_{n}^{2} & =E\left[\left(\Delta_{1}IF_{1}\right)^{2}+\Delta_{1}IF_{1}\Delta_{1}\Delta_{2}IF_{2}+\frac{1}{4}\left(\Delta_{1}\Delta_{2}IF_{2}\right)^{2}+\frac{1}{3}\Delta_{1}IF_{1}\Delta_{1}\Delta_{2}\Delta_{3}IF_{3}\right]\\
 & -\frac{1}{\kappa_{2}}E\left[\left(\Delta_{1}IF_{1}^{2}+2\left(\hat{E}_{1}\left(\Delta_{2}\left(IF_{1}IF_{2}\right)\right)\right)\right)\left[\left(\Delta_{1}IF_{1}\right)^{2}+\Delta_{1}IF_{1}\Delta_{1}\Delta_{2}IF_{2}\right]\right]\\
 & +\frac{1}{\kappa_{2}^{2}}E\left[\left(\begin{array}{c}
\left(\Delta_{1}IF_{1}^{2}+2\left(\hat{E}_{1}\left(\Delta_{2}\left(IF_{1}IF_{2}\right)\right)\right)\right)^{2}\\
-\left(\hat{E}_{1}\left(\Delta_{2}IF_{2}\right)^{2}+\hat{E}_{1}\Delta_{2}\Delta_{3}IF_{1}IF_{3}-\left(\Delta_{1}IF_{1}\right)^{2}\right)\kappa_{2}
\end{array}\right)\left[\left(\Delta_{1}IF_{1}\right)^{2}\right]\right]+O(n^{-5/2})
\end{align*}
We could compute the expectation of each term in the RHS as follows (each equation holds with residual $O(n^{-3})$, and the derivation for each of them is similar to \eqref{eq: computation_eg} ):
\[
E[(\Delta_{1}IF_{1})^{2}]=\frac{1}{n}\kappa_{2}
\]
\[
E[\Delta_{1}IF_{1}\Delta_{1}\Delta_{2}IF_{2}]=\frac{1}{n^{2}}E[IF_{1}(X)IF_{2}(X,X)]
\]
\begin{align*}
E[\left(\Delta_{1}\Delta_{2}IF_{2}\right)^{2}] & =\frac{2}{n^{2}}E\left[(IF_{2}(X,Y))^{2}\right]+\frac{1}{n^{2}}\left(EIF_{2}(X,X)\right)^{2}
\end{align*}
\[
E[\Delta_{1}IF_{1}\Delta_{1}\Delta_{2}\Delta_{3}IF_{3}]=\frac{3}{n^{2}}E[IF_{1}(X)IF_{3}(X,Y,Y)]
\]
\[
E[\Delta_{1}IF_{1}^{2}\left(\Delta_{1}IF_{1}\right)^{2}]=\frac{1}{n^{2}}E(IF^{4}(X))-\frac{1}{n^{2}}\left(E(IF^{2}(X))\right)^{2}
\]
\begin{align*}
E[\Delta_{1}IF_{1}^{2}\Delta_{1}IF_{1}\Delta_{1}\Delta_{2}IF_{2}] & =\frac{1}{n^{2}}EIF^{3}(X)EIF_{2}(Y,Y)\\
 & +2\frac{1}{n^{2}}EIF(Y)IF^{2}(X)IF_{2}(X,Y)
\end{align*}
\[
E[\hat{E}_{1}\left(\Delta_{2}\left(IF_{1}IF_{2}\right)\right)\left(\Delta_{1}IF_{1}\right)^{2}]=\frac{1}{n^{2}}EIF_{1}(X)IF_{2}(X,X)EIF_{1}^{2}(X)+\frac{3}{n^{2}}EIF_{1}^{2}(X)IF_{2}(Y)IF_{2}(X,Y)
\]
\begin{align*}
E[\left(\hat{E}_{1}\left(\Delta_{2}\left(IF_{1}IF_{2}\right)\right)\right)\Delta_{1}IF_{1}\Delta_{1}\Delta_{2}IF_{2}] & =\frac{1}{n^{2}}EIF_{1}(X)IF_{1}(Y)IF_{2}(X,Y)E(IF_{2}(Z,Z))\\
 & +\frac{2}{n^{2}}EIF_{1}(X)IF_{1}(Z)IF_{2}(X,Y)IF_{2}(Y,Z)
\end{align*}
\[
E[\Delta_{1}IF_{1}^{2}\Delta_{1}IF_{1}^{2}\left(\Delta_{1}IF_{1}\right)^{2}]=\frac{1}{n^{2}}\kappa_{2}E(IF_{1}^{2}(X)-\kappa_{2})^{2}+\frac{2}{n^{2}}\left(EIF_{1}^{3}(X)\right)^{2}
\]
\begin{align*}
E\left[\Delta_{1}IF_{1}^{2}\left(\hat{E}_{1}\left(\Delta_{2}\left(IF_{1}IF_{2}\right)\right)\right)\left(\Delta_{1}IF_{1}\right)^{2}\right] & =\frac{1}{n^{2}}\kappa_{2}E(IF_{1}^{2}(X)-\kappa_{2})(IF_{1}(Y)IF_{2}(Y,X))\\
 & +\frac{2}{n^{2}}EIF_{1}^{3}EIF_{1}(X)IF_{1}(Y)IF_{2}(X,Y)
\end{align*}
\begin{align*}
E\left[\left(\hat{E}_{1}\left(\Delta_{2}\left(IF_{1}IF_{2}\right)\right)\right)^{2}\left(\Delta_{1}IF_{1}\right)^{2}\right] & =\frac{1}{n^{2}}\kappa_{2}EIF_{1}(X)IF_{1}(Z)IF_{2}(X,Y)IF_{2}(Z,Y)\\
 & +\frac{2}{n^{2}}E\left(IF_{1}(X)IF_{1}(Y)IF_{2}(X,Y)\right)^{2}
\end{align*}
\[
E\left[\hat{E}_{1}\left(\Delta_{2}IF_{2}\right)^{2}\left(\Delta_{1}IF_{1}\right)^{2}\right]=\frac{1}{n^{2}}\kappa_{2}E(IF_{2}(X,Y))^{2}+\frac{2}{n^{2}}\mu_{2,a}
\]
\[
E\left[\hat{E}_{1}\Delta_{2}\Delta_{3}IF_{1}IF_{3}\left(\Delta_{1}IF_{1}\right)^{2}\right]=\frac{1}{n^{2}}\kappa_{2}E[IF_{1}(X)IF_{3}(X,Y,Y)]+\frac{2}{n^{2}}E[IF_{1}(X)IF_{1}(Y)IF_{1}(Z)IF_{3}(X,Y,Z)]
\]
\[
E[\left(\Delta_{1}IF_{1}\right)^{4}]=\frac{3}{n}\kappa_{2}^{2}
\]
Let

\begin{equation}
\mu_{(1,3),d}:=EIF_{1}(X)\left(IF_{3}(X,Y,Y)-IF_{3}(X,Z,Y)\right)\label{eq: mu_defination}
\end{equation}
\[
\mu_{2,c,2}=E\left(IF_{1}(X)IF_{1}(Y)IF_{2}(X,Y)\right)^{2}
\]
Then putting the expectation of each term (computed above) together,
we would get the second moment 
\begin{align*}
EA_{n}^{2}= & 1+\frac{1}{n}(-\frac{1}{2\kappa_{2}}\mu_{2,2}+\frac{1}{4\kappa_{2}}\mu_{2,d}^{2}-\frac{\gamma}{\kappa_{2}^{2}}\mu_{2,d}-\frac{4}{\kappa_{2}^{2}}\mu_{2,b}-\frac{1}{\kappa_{2}}\mu_{(1,2),d}-\frac{2}{\kappa_{2}^{2}}\mu_{2,c}\mu_{2,d}\\
 & +\frac{2}{\kappa_{2}^{3}}\gamma^{2}+\frac{8\gamma}{\kappa_{2}^{3}}\mu_{2,c}+\frac{8}{\kappa_{2}^{3}}\mu_{2,c}^{2}-2\mu_{2,a}\frac{1}{\kappa_{2}^{2}}\\
 & -\frac{2}{\kappa_{2}^{2}}\mu_{3,c}+3)+O(n^{-2}).
\end{align*}

\subsubsection{Third moment of \texorpdfstring{$A_{n}$}{TEXT}}

From \eqref{eq: A_n_expression} we get that
\[
A_{n}^{3}=n^{3/2}\kappa_{2}^{-3/2}(\left(\Delta_{1}IF_{1}\right)^{3}+\frac{3}{2}\left(\Delta_{1}IF_{1}\right)^{2}\Delta_{1}\Delta_{2}IF_{2})\left(1-\frac{3I_{1}}{2\kappa_{2}}\right)+O_{p}(n^{-1}).
\]
Hence
\begin{equation}
\frac{A_{n}^{3}}{n^{3/2}\kappa_{2}^{-3/2}}=(\left(\Delta_{1}IF_{1}\right)^{3}+\frac{3}{2}\left(\Delta_{1}IF_{1}\right)^{2}\Delta_{1}\Delta_{2}IF_{2})\left(1-\frac{3I_{1}}{2\kappa_{2}}\right)+O_{p}(n^{-5/2})\label{eq: 3rd_moment_expansion}
\end{equation}
Similar to the computation for the first moment, the expectation of the residual in the RHS is of order $O(n^{-3})$, because there is no $O(n^{-5/2})$ term after taking expectation.
We observe that the following holds with residual $O(n^{-3})$
\[
E\left(\Delta_{1}IF_{1}\right)^{3}=\frac{1}{n^{2}}EIF_{1}^{3}
\]
\[
E\left[\left(\Delta_{1}IF_{1}\right)^{3}\left(\Delta_{1}IF_{1}^{2}\right)\right]=\frac{3}{n^{2}}\kappa_{2}\gamma
\]
\[
E\left[\left(\Delta_{1}IF_{1}\right)^{3}\left(\hat{E}_{1}\left(\Delta_{2}\left(IF_{1}IF_{2}\right)\right)\right)\right]=\frac{3}{n^{2}}\kappa_{2}EIF_{1}(X)IF_{1}(Y)IF_{2}(X,Y)
\]
\[
E\left[\left(\Delta_{1}IF_{1}\right)^{2}\Delta_{1}\Delta_{2}IF_{2}\right]=\frac{1}{n^{2}}\kappa_{2}E[IF_{2}(X,X)]+\frac{2}{n^{2}}EIF_{1}(X)IF_{1}(Y)IF_{2}(X,Y)
\]
Therefore, taking expectation in \eqref{eq: 3rd_moment_expansion} and plugging
in the above, we get
\[
EA_{n}^{3}=\kappa_{2}^{-3/2}n^{-1/2}\left(-\frac{7}{2}\gamma-6\mu_{2,c}+\frac{3}{2}\kappa_{2}\mu_{2,d}\right)+O(n^{-3/2}).
\]

\subsubsection{Fourth moment of \texorpdfstring{$A_{n}$}{TEXT}}

From \eqref{eq: A_n_expression} we get
\begin{align*}
\frac{A_{n}^{4}}{n^{2}\kappa_{2}^{-2}}= & \left(\left(\Delta_{1}IF_{1}\right)^{4}+2\left(\Delta_{1}IF_{1}\right)^{3}\Delta_{1}\Delta_{2}IF_{2}+\frac{6}{4}\left(\Delta_{1}IF_{1}\right)^{2}\left(\Delta_{1}\Delta_{2}IF_{2}\right)^{2}+\frac{4}{6}\left(\Delta_{1}IF_{1}\right)^{3}\Delta_{1}\Delta_{2}\Delta_{3}IF_{3}\right)\\
 & \times\left(1-\frac{2I_{1}}{\kappa_{2}}-\frac{2I_{2}}{\kappa_{2}}+\frac{3I_{1}^{2}}{\kappa_{2}^{2}}\right).
\end{align*}
After plugging in $I_{1}$ and $I_{2}$, we can compute the expectation
of each term (all of them hold with residual $O(n^{4})$)
\[
E[(\Delta_{1}IF_{1})^{4}]=\frac{3\kappa_{2}^{2}}{n^{2}}+\frac{1}{n^{3}}(EIF_{1}^{4}-3\kappa_{2}^{2})
\]
\[
E[(\Delta_{1}IF_{1})^{6}]=\frac{15}{n^{3}}\kappa_{2}^{3}
\]
\[
E[\left(\Delta_{1}IF_{1}\right)^{3}\Delta_{1}\Delta_{2}IF_{2}]=\frac{3\kappa_{2}}{n^{2}}\mu_{1,2,d}+\frac{\gamma}{n^{3}}\mu_{2,d}+\frac{6}{n^{3}}\mu_{2,b}
\]
\begin{align*}
E[\left(\Delta_{1}IF_{1}\right)^{2}\left(\Delta_{1}\Delta_{2}IF_{2}\right)^{2}] & =\frac{2\kappa_{2}}{n^{3}}E\left[(IF_{2}(X,Y))^{2}\right]+\frac{\kappa_{2}}{n^{3}}\left([(EIF_{2}(X,X))]\right)^{2}\\
 & +\frac{4}{n^{3}}E[IF_{2}(X,X)]EIF_{1}(X)IF_{1}(Y)IF_{2}(X,Y)\\
 & +\frac{8}{n^{3}}E[IF_{1}(X)IF_{1}(Y)IF_{2}(X,Z)IF_{2}(Y,Z)]
\end{align*}
\[
E[\left(\Delta_{1}IF_{1}\right)^{3}\Delta_{1}\Delta_{2}\Delta_{3}IF_{3}]=\frac{9\kappa_{2}}{n^{3}}E[IF_{1}(X)IF_{3}(X,Y,Y)]+\frac{6}{n^{3}}EIF_{1}(X)IF_{1}(Y)IF_{1}(Z)IF_{3}(X,Y,Z)
\]
\[
E[\Delta_{1}IF_{1}^{2}\left(\Delta_{1}IF_{1}\right)^{4}]=6\kappa_{2}\left[\frac{1}{n^{2}}E(IF^{4}(X))-\frac{1}{n^{2}}\left(E(IF^{2}(X))\right)^{2}\right]+4\gamma^{2}
\]
\begin{align*}
E[\Delta_{1}IF_{1}^{2}\left(\Delta_{1}IF_{1}\right)^{3}\Delta_{1}\Delta_{2}IF_{2}] & =\frac{3\kappa_{2}}{n^{2}}EIF^{3}(X)E(IF_{2}(Y,Y))\\
 & +2\frac{3\kappa_{2}}{n^{2}}EIF(Y)IF^{2}(X)IF_{2}(X,Y)+6\gamma EIF_{1}(X)IF_{1}(Y)IF_{2}(X,Y)
\end{align*}
\begin{align*}
E[\hat{E}_{1}\left(\Delta_{2}\left(IF_{1}IF_{2}\right)\right)\left(\Delta_{1}IF_{1}\right)^{4}] & =\frac{3\kappa_{2}}{n^{3}}EIF_{1}(X)IF_{2}(X,X)EIF_{1}^{2}(X)+\frac{18\kappa_{2}}{n^{3}}EIF_{1}^{2}(X)IF_{1}(Y)IF_{2}(X,Y)\\
 & +\frac{4}{n^{3}}\gamma EIF_{1}(X)IF_{1}(Y)IF_{2}(X,Y)
\end{align*}
\begin{align*}
E[\left(\hat{E}_{1}\left(\Delta_{2}\left(IF_{1}IF_{2}\right)\right)\right)\left(\Delta_{1}IF_{1}\right)^{3}\Delta_{1}\Delta_{2}IF_{2}] & =\frac{3\kappa_{2}}{n^{3}}EIF_{1}(X)IF_{1}(Y)IF_{2}(X,Y)E(IF_{2}(Z,Z))\\
 & +\frac{6\kappa_{2}}{n^{3}}EIF_{1}(X)IF_{1}(Z)IF_{2}(X,Y)IF_{2}(Y,Z)\\
 & +\frac{6}{n^{3}}\left(EIF_{1}(X)IF_{1}(Y)IF_{2}(X,Y)\right)^{2}
\end{align*}
\[
E[\Delta_{1}IF_{1}^{2}\Delta_{1}IF_{1}^{2}\left(\Delta_{1}IF_{1}\right)^{4}]=\frac{3}{n^{3}}\kappa_{2}^{2}E(IF_{1}^{2}(X)-\kappa_{2})^{2}+\frac{12\kappa_{2}}{n^{3}}\left(EIF_{1}^{3}(X)\right)^{2}
\]
\begin{align*}
E\left[\Delta_{1}IF_{1}^{2}\left(\hat{E}_{1}\left(\Delta_{2}\left(IF_{1}IF_{2}\right)\right)\right)\left(\Delta_{1}IF_{1}\right)^{4}\right]= & \frac{3}{n^{3}}\kappa_{2}^{2}E(IF_{1}^{2}(X)-\kappa_{2})(IF_{1}(Y)IF_{2}(Y,X))\\
 & +\frac{12\kappa_{2}}{n^{3}}EIF_{1}^{3}EIF_{1}(X)IF_{1}(Y)IF_{2}(X,Y)
\end{align*}
\begin{align*}
E\left[\left(\hat{E}_{1}\left(\Delta_{2}\left(IF_{1}IF_{2}\right)\right)\right)^{2}\left(\Delta_{1}IF_{1}\right)^{4}\right] & =\frac{3}{n^{3}}\kappa_{2}^{2}EIF_{1}(X)IF_{1}(Z)IF_{2}(X,Y)IF_{2}(Z,Y)\\
 & +\frac{12\kappa_{2}}{n^{3}}E\left(IF_{1}(X)IF_{1}(Y)IF_{2}(X,Y)\right)^{2}
\end{align*}
\[
E\left[\hat{E}_{1}\left(\Delta_{2}IF_{2}\right)^{2}\left(\Delta_{1}IF_{1}\right)^{4}\right]=\frac{3}{n^{3}}\kappa_{2}^{2}E(IF_{2}(X,Y))^{2}+\frac{12\kappa_{2}}{n^{3}}\mu_{2,a}
\]
\[
E\left[\hat{E}_{1}\Delta_{2}\Delta_{3}IF_{1}IF_{3}\left(\Delta_{1}IF_{1}\right)^{4}\right]=\frac{3}{n^{2}}\kappa_{2}^{2}E[IF_{1}(X)IF_{3}(X,Y,Y)]+\frac{12\kappa_{2}}{n^{2}}E[IF_{1}(X)IF_{1}(Y)IF_{1}(Z)IF_{3}(X,Y,Z)]
\]
After plug them in, we will get 
\begin{align*}
EA_{n}^{4}= & 3\kappa_{2}^{2}+\frac{1}{n\kappa_{2}^{2}}(\kappa_{4}+6\kappa_{2}\mu_{(1,2),d}+2\gamma\mu_{2,d}+12\mu_{2,b}+3\kappa_{2}\mu_{2,2}+\frac{3}{2}\kappa_{2}\mu_{2,d}^{2}+6\mu_{2,d}\mu_{2,c}\\
 & +12\mu_{2,a}+6\kappa_{2}\mu_{(1,3),d}+4\mu_{3,c}-12(\kappa_{4}+2\kappa_{2}^{2})-\frac{8}{\kappa_{2}}\gamma^{2}-12\gamma\mu_{2,d}\\
 & -24\mu_{2,b}-\frac{24\gamma}{\kappa_{2}}\mu_{2,c}-12\kappa_{2}(\mu_{1,2,d})-72\mu_{2,b}-\frac{16}{\kappa_{2}}\gamma\mu_{2,c}-24\mu_{2,c}\mu_{2,d}-48\mu_{2,a}\\
 & -\frac{48}{\kappa_{2}}\mu_{2,c}^{2}+9(\kappa_{4}+2\kappa_{2}^{2})+\frac{36}{\kappa_{2}}\gamma^{2}+36\mu_{2,b}+\frac{144}{\kappa_{2}}\gamma\mu_{2,c}+36\mu_{2,a}\\
 & +\frac{144}{\kappa_{2}}\mu_{2,c}^{2}-6\kappa_{2}\mu_{2,2}-24\mu_{2,a}-6\kappa_{2}\mu_{(1,3),d}-24\mu_{3,c}+30\kappa_{2}^{2} + O(n^{-2})
\end{align*}

\subsubsection{The expectation of \texorpdfstring{$A_{n}^{i}\frac{\hat{C}_{2}}{\hat{C}_{1}}$}{TEXT}}

When $i=2$ or $i=4$, with error $O(n^{-1})$, the expectation is given by $\frac{C_2}{C_1}EA^i$. Here, $C_{1},C_{2}$ are $\hat{C}_{1},\hat{C}_{2}$ with $\hat{E},\hat{P}$ replaced by $E,P_{0}$. 

Next, we compute the expectation of $A_{n}^{3}\frac{\hat{C}_{2}}{\hat{C}_{1}}$.
We write it as $A_{n}^{3}\frac{\hat{C}_{2}}{\hat{C}_{1}}=A_{n}^{3}\frac{C_{2}}{C_{1}}+A_{n}^{3}\left(\frac{\hat{C}_{2}}{\hat{C}_{1}}-\frac{C_{2}}{C_{1}}\right)$. For
the expectation of the first term, notice that $C_{1},C_{2}$ are
not random, it is just $\frac{C_{2}}{C_{1}}EA_{n}^{3}$ and we have
computed $EA_{n}^{3}$ before. For the second term, recall that 
\[
\hat{C}_{1}=\sqrt{\hat{\kappa}_{2}},\hat{C}_{2}=-\frac{1}{6}\frac{\phi^{\prime\prime\prime}(1)\hat{\gamma}}{\left(\phi^{\prime\prime}(1)\right)\hat{\kappa}_{2}}+\frac{1}{2}\frac{1}{\hat{\kappa}_{2}}\hat{\mu}_{2,c}.
\]
The expansion for $\hat{\kappa}_2^{-1/2}$ is given in \eqref{eq: expand_var}. Taking the third power, and only keeping the $O(n^{-1/2})$ part, we get 
\[
\hat{\kappa}_{2}^{-3/2}=\kappa_{2}^{-3/2}\left(1-\frac{3I_{1}}{2\kappa_{2}}\right)+O_{p}(n^{-1})=\kappa_{2}^{-3/2}\left(1-\frac{3}{2\kappa_{2}}\Delta_{1}IF_{1}^{2}-\frac{3}{\kappa_{2}}\hat{E}_{1}\Delta_{2}IF_{1}IF_{2}\right)+O_{p}(n^{-1})
\]
Similarly, from the expansion for $IF_1$ given in \eqref{eq: IF_1_expansion}, we have the expansion for $\hat{\gamma}$ given by (note that here we have fewer terms than \eqref{eq: IF_1_expansion} because we only want an expansion with residual $O_p(n^{-1})$)
\begin{align*}
\hat{\gamma} & =\hat{E}\left(IF_{1}+\hat{E}_{2}IF_{2}-\hat{E}IF_{1}\right)^{3}+O_{p}(n^{-1})\\
 & =\hat{E}_{1}\left(IF_{1}+\Delta_{2}IF_{2}-\Delta_{1}IF_{1}\right)^{3}+O_{p}(n^{-1})\\
 & =\hat{E}_{1}IF_{1}^{3}+3\hat{E}_{1}\left[IF_{1}^{2}\Delta_{2}IF_{2}\right]-3\hat{E}_{1}\left[IF_{1}^{2}\Delta_{1}IF_{1}\right]+O_{p}(n^{-1})\\
 & =\gamma+\Delta_{1}IF_{1}^{3}+3\hat{E}_{1}\left[IF_{1}^{2}\Delta_{2}IF_{2}\right]-3\hat{E}_{1}\left[IF_{1}^{2}\Delta_{1}IF_{1}\right]+O_{p}(n^{-1})
\end{align*}
Therefore, $E\left[A_{n}^{3}\left(-\frac{1}{6}\frac{\phi^{\prime\prime\prime}(1)\hat{\gamma}}{\left(\phi^{\prime\prime}(1)\right)\hat{\kappa}_{2}^{3/2}}+\frac{1}{6}\frac{\phi^{\prime\prime\prime}(1)\gamma}{\left(\phi^{\prime\prime}(1)\right)\kappa_{2}^{3/2}}\right)\right]$
is given by (note that $A_{n}=\sqrt{n}\Delta_{1}IF_{1}+O_{p}(n^{-1/2})$)
\begin{align*}
 & -\frac{1}{6}\frac{\phi^{\prime\prime\prime}(1)n^{3/2}}{\left(\phi^{\prime\prime}(1)\right)\kappa_{2}^{3/2}}E\left[\left(\Delta_{1}IF_{1}\right)^{3}\left(\frac{\hat{\gamma}}{\hat{\kappa}_{2}^{3/2}}-\frac{\gamma}{\kappa_{2}^{3/2}}\right)\right]+O(n^{-1})\\
= & -\frac{1}{6}\frac{\phi^{\prime\prime\prime}(1)n^{3/2}}{\left(\phi^{\prime\prime}(1)\right)\kappa_{2}^{3/2}}E\left[\left(\Delta_{1}IF_{1}\right)^{3}\left(\begin{array}{c}
\kappa_{2}^{-3/2}\left(\Delta_{1}IF_{1}^{3}+3\hat{E}_{1}\left[IF_{1}^{2}\Delta_{2}IF_{2}\right]-3\hat{E}_{1}\left[IF_{1}^{2}\Delta_{1}IF_{1}\right]\right)\\
-\gamma\left(\frac{3}{2\kappa_{2}}\Delta_{1}IF_{1}^{2}+\frac{3}{\kappa_{2}}\hat{E}_{1}\Delta_{2}IF_{1}IF_{2}\right)
\end{array}\right)\right]+O(n^{-1})
\end{align*}
It can be computed that
\[
E\left[\left(\Delta_{1}IF_{1}\right)^{3}\Delta_{1}IF_{1}^{3}\right]=\frac{3\kappa_{2}\mu_{4}}{n^{2}}+O(n^{-3})
\]
\[
E\left[\left(\Delta_{1}IF_{1}\right)^{3}\hat{E}_{1}\left[IF_{1}^{2}\Delta_{2}IF_{2}\right]\right]=\frac{3\kappa_{2}}{n^{2}}EIF_{1}(X)IF_{1}^{2}(Y)IF_{2}(X,Y)+O(n^{-3})=\frac{3\kappa_{2}}{n^{2}}\mu_{2,b}+O(n^{-3})
\]
\[
E\left[\left(\Delta_{1}IF_{1}\right)^{3}\hat{E}_{1}\left[IF_{1}^{2}\Delta_{1}IF_{1}\right]\right]=\frac{3\kappa_{2}^{2}}{n^{2}}+O(n^{-3})
\]
\[
E\left[\left(\Delta_{1}IF_{1}\right)^{3}\Delta_{1}IF_{1}^{2}\right]=\frac{3\kappa_{2}\gamma}{n^{2}}+O(n^{-3})
\]
\[
E\left[\left(\Delta_{1}IF_{1}\right)^{3}\hat{E}_{1}\Delta_{2}IF_{1}IF_{2}\right]=\frac{3\kappa_{2}}{n^{2}}E[IF_{1}(X)IF_{1}(Y)IF_{2}(X,Y)]=\frac{3\kappa_{2}}{n^{2}}\mu_{2,a}+O(n^{-3})
\]
Plugging this into the above, we get that 
\begin{equation}
E\left[A_{n}^{3}\left(-\frac{1}{6}\frac{\phi^{\prime\prime\prime}(1)\hat{\gamma}}{\left(\phi^{\prime\prime}(1)\right)\hat{\kappa}_{2}^{3/2}}+\frac{1}{6}\frac{\phi^{\prime\prime\prime}(1)\gamma}{\left(\phi^{\prime\prime}(1)\right)\kappa_{2}^{3/2}}\right)\right]=-\frac{\phi^{\prime\prime\prime}(1)}{2\phi^{\prime\prime}(1)n^{1/2}\kappa_{2}^{-3}}\left(\mu_{4}-3\kappa_{2}^{2}-\frac{3}{2\kappa_{2}}\left(\gamma^{2}+2\gamma\mu_{2,c}\right)+3\mu_{2,b}\right).\label{eq: gamma_result}
\end{equation}

For $\hat{\mu}_{2,c}$, it can be expanded as
\begin{align*}
\hat{\mu}_{2,c}= & \hat{E}\left[IF_{1}(X;\hat{P})IF_{1}(Y;\hat{P})IF_{2}(X,Y;\hat{P})\right]\\
= & \hat{E}\left[IF_{1}(X)IF_{1}(Y)IF_{2}(X,Y)\right]+\hat{E}\left[IF_{1}(X)IF_{1}(Y)\hat{E}_{Z}IF_{3}(X,Y,Z)\right]\\
 & +2\hat{E}\left[\left(\hat{E}_{Y}IF_{2}(X,Y)-\hat{E}IF_{1}(X)\right)IF_{1}(Y)IF_{2}(X,Y)\right]+O_p(n^{-1})\\
= & \mu_{2,c}+\left(\hat{E}-E\right)\left[IF_{1}(X)IF_{1}(Y)IF_{2}(X,Y)\right]+\hat{E}\left[IF_{1}(X)IF_{1}(Y)\hat{E}_{Z}IF_{3}(X,Y,Z)\right]\\
 & +2\hat{E}\left[\left(\hat{E}_{Y}IF_{2}(X,Y)-\hat{E}IF_{1}(X)\right)IF_{1}(Y)IF_{2}(X,Y)\right]
\end{align*}
Here, in the second equality, we used the expansion for $IF_2(X,Y;\hat{P})$ given in Theorem \ref{thm: Bartlett_full}, but we didn't keep terms like $\hat{E}_XIF_3(X,Y,Z)$ in that expansion because $$\hat{E}\left[IF_1(X)IF_1(Y)\hat{E}_{X}IF_{3}(X,Y,Z)\right]=\hat{E}[IF_1(X)]\hat{E}\left[IF_1(Y)\hat{E}_{Z}IF_{3}(X,Y,Z)\right]=O_p(n^{-1}).$$

Therefore, from the expansion for $\hat{\mu}_{2,c}$ and $\hat{\kappa}_2^{-3/2}$, we get
\begin{align}
 &EA_{n}^{3}\left(\frac{\hat{\mu}_{2,c}}{\hat{\kappa}_{2}^{3/2}}-\frac{\mu_{2,c}}{\kappa_{2}^{3/2}}\right)\nonumber \\= & n^{3/2}\kappa_{2}^{-3/2}E\left[\left(\Delta_{1}IF_{1}\right)^{3}\left(\begin{array}{c}
\kappa_{2}^{-3/2}\left(\begin{array}{c}
\left(\hat{E}-E\right)\left[IF_{1}(X)IF_{1}(Y)IF_{2}(X,Y)\right]\\
+\hat{E}\left[IF_{1}(X)IF_{1}(Y)\hat{E}_{Z}IF_{3}(X,Y,Z)\right]\\
+2\hat{E}\left[\left(\hat{E}_{Y}IF_{2}(X,Y)-\hat{E}IF_{1}(X)\right)IF_{1}(Y)IF_{2}(X,Y)\right]
\end{array}\right)\\
-\mu_{2,c}\left(\frac{3}{2\kappa_{2}}\Delta_{1}IF_{1}^{2}+\frac{3}{\kappa_{2}}\hat{E}_{1}\Delta_{2}IF_{1}IF_{2}\right)
\end{array}\right)\right]\label{eq: compute_mu_2c}
\end{align}
Note that we have 
\begin{align*}
 & E\left[\left(\Delta_{1}IF_{1}\right)^{3}\left(\hat{E}-E\right)\left[IF_{1}(X)IF_{1}(Y)IF_{2}(X,Y)\right]\right]\\
= & \frac{1}{n^{5}}\sum_{1\leq i,j,k,l,m\leq n}E\left[IF_{1}(X_{i})IF_{1}(X_{j})IF_{1}(X_{k})\left(IF_{1}(X_{l})IF_{1}(X_{m})IF_{2}(X_{l},X_{m})-\mu_{2,c}\right)\right]
\end{align*}
For each term in the summation, for the expectation to be nonzero, we must have $\{i,j,k\}\cap\{l,m\}\neq\emptyset$
and each of $i,j,k$ appears at least twice among $\{i,j,k,l,m\}$.
Suppose that $i=l$ is the pair that makes $\{i,j,k\}\cap\{l,m\}\neq\emptyset$.
Then we need the other pair to be $j=k$ (note that the number of
choices of $\{i,j,k,l,m\}$ with more than 2 pairs is of order $O(n^{2})$,
so these terms only contribute $O(n^{-3})$ to the above summation).
Clearly, there are $2\times3=6$ ways to choose the pair that makes
$\{i,j,k\}\cap\{l,m\}\neq\emptyset$. Therefore, the above is given
by
\[
\frac{6}{n^{2}}\kappa_{2}E\left[IF_{1}(X)\left(IF_{1}(X)IF_{1}(Y)IF_{2}(X,U)-\mu_{2,c}\right)\right]+O(n^{-3})=\frac{6}{n^{3}}\kappa_{2}\mu_{2,b}+O(n^{-3}).
\]
Similarly, 
\[
E\left[\left(\Delta_{1}IF_{1}\right)^{3}\hat{E}\left[IF_{1}(X)IF_{1}(Y)\hat{E}_{Z}IF_{3}(X,Y,Z)\right]\right]=\frac{3}{n^{2}}\kappa_{2}\mu_{3,c}+O(n^{-3})
\]
\[
E\left[\left(\Delta_{1}IF_{1}\right)^{3}\hat{E}\left[\left(\hat{E}_{Y}IF_{2}(X,Y)-\hat{E}IF_{1}(X)\right)IF_{1}(Y)IF_{2}(X,Y)\right]\right]=\frac{3\kappa_{2}\mu_{2,a}}{n^{2}}+O(n^{-3}).
\]
Plugging these (and some terms we have computed previously) into (\ref{eq: compute_mu_2c}),
we get that
\begin{equation}
\frac{1}{2}EA_{n}^{3}\left(\frac{\hat{\mu}_{2,c}}{\hat{\kappa}_{2}^{3/2}}-\frac{\mu_{2,c}}{\kappa_{2}^{3/2}}\right)=\frac{3}{2n^{1/2}}\kappa_{2}^{-3}\left(\mu_{3,c}+2\mu_{2,a}+2\mu_{2,b}-\frac{3}{2\kappa_{2}}\mu_{2,c}\left(\gamma+2\mu_{2,c}\right)\right).\label{eq: mu_2c_result}
\end{equation}
Summing up (\ref{eq: gamma_result}) and (\ref{eq: mu_2c_result}),
we get 
\begin{align*}
& EA_{n}^{3}\left(\frac{\hat{C}_{2}}{\hat{C}_{1}}-\frac{C_{2}}{C_{1}}\right) \\
= & \frac{3}{2n}\kappa_{2}^{-2}(-\frac{\phi^{\prime\prime\prime}(1)}{\phi^{\prime\prime}(1)}\frac{\mu_{4}-3\kappa_{2}^{2}}{3}+\frac{1}{2\kappa_{2}}\frac{\phi^{\prime\prime\prime}(1)}{\phi^{\prime\prime}(1)}(\gamma^{2}+2\gamma\mu_{2,c})-\frac{3}{2\kappa_{2}}\mu_{2,c}(\gamma+2\mu_{2,c})\\&+\left(2-\frac{\phi^{\prime\prime\prime}(1)}{\phi^{\prime\prime}(1)}\right)\mu_{2,b}+\mu_{3,c}+2\mu_{2,a}).
\end{align*}

Summing this with $\frac{C_2}{C_1}EA_n^3$, we get $EA_n^3\frac{\hat{C_2}}{\hat{C_1}}$. $EA_n^5\frac{\hat{C_2}}{
\hat{C_1}}$ can be computed in a similar way. Indeed, since the 5-th cumulant of $A_n$ is of order $O(n^{-3/2})$ (this follows from the general observation made in (2.20) of \cite{Hall1992} or Theorem 2.1 of \cite{Bhattacharya1983}), we have that $EA_n^5 = 10EA_n^3 -15EA_n+O(n^{-1})$. In addition, we can show that
\begin{align*}
& EA_{n}^{5}\left(\frac{\hat{C}_{2}}{\hat{C}_{1}}-\frac{C_{2}}{C_{1}}\right) \\
= & \frac{15}{2n}\kappa_{2}^{-2}(-\frac{\phi^{\prime\prime\prime}(1)}{\phi^{\prime\prime}(1)}\frac{\mu_{4}-3\kappa_{2}^{2}}{3}+\frac{1}{2\kappa_{2}}\frac{\phi^{\prime\prime\prime}(1)}{\phi^{\prime\prime}(1)}(\gamma^{2}+2\gamma\mu_{2,c})-\frac{3}{2\kappa_{2}}\mu_{2,c}(\gamma+2\mu_{2,c})\\&+\left(2-\frac{\phi^{\prime\prime\prime}(1)}{\phi^{\prime\prime}(1)}\right)\mu_{2,b}+\mu_{3,c}+2\mu_{2,a}).
\end{align*}
Actually, the above is exactly 5 times $EA_n^3\left(\frac{\hat{C}_{2}}{\hat{C}_{1}}-\frac{C_{2}}{C_{1}}\right)$. Summing the above with $\frac{C_2}{C_1}EA_n^5$, we get $EA_n^5\frac{\hat{C_2}}{\hat{C_1}}$.

\subsubsection{The expectation of \texorpdfstring{$\left(2\left(\frac{\hat{C}_{2}}{\hat{C}_{1}}\right)^{2}-\frac{\hat{C}_{3}}{\hat{C}_{1}}\right)A_{n}^{j}$}{TEXT}}

With an error of $O(n^{-1/2})$, it is given by $\left(2\left(\frac{C_{2}}{C_{1}}\right)^{2}-\frac{C_{3}}{C_{1}}\right)EA_{n}^{j}$. Actually, we only need the leading term of $EA_{n}^{j}$, which is $EA_{n}^{j}=O(n^{-1/2})$ when $j$ is odd, $EA_{n}^{4}=3+O(n^{-1})$, $EA_{n}^{6}=15+O(n^{-1})$ (this can be checked from the observation that $A_n$ is asymptotically normal). 

\subsection{Obtaining Edgeworth Expansion}

Finally, based on the moments computed in the previous sections, we
can get the cumulants of $W_{n}$ and thus the Edgeworth
expansion for $W_n$. These plug-in steps are implemented in Mathematica whose code can be found in \url{https://www.wolframcloud.com/obj/sh3972/Published/DRO_Bartlett.nb}. 

\section{Numerical Comparison on the Smooth Function Model}\label{sec:numerics smooth}

Consider $\psi(P)=f(E_PX_1,E_PX_2)$, where $f(x,y)=x+y^2$ and $(X_1,X_2)$ follows standard 2-dimensional normal distribution under $P$. Suppose we have $n$ samples drawn independently from $P$. We choose $\phi$ as the reversed KL divergence as in \cite{diciccio1991}. We compare among standard EL, our TB, and theoretical Bartlett using the formula in \cite{diciccio1991} (TB2). With a bit of algebra, we have that $t_1=t_2=t_4=0,t_3=3,t_5=-4$ where $t_i,i=1,\dots,5$ are defined in Corollary \ref{cor: function_of_mean}. In contrast, the result in Section 2.4 of \cite{diciccio1991} suggests that $t_5 = 12$. As a result, our Bartlett correction formula would choose $q=\chi^2_{1;1-\alpha}\left(1+\frac{1}{2n}\right)$, while the formula in \cite{diciccio1991} would choose $q=\chi^2_{1;1-\alpha}\left(1+\frac{9}{2n}\right)$. Table \ref{tab: numerical_sfm} describes the result.  

\begin{table}[H]
\centering
\caption{Coverage probabilities (with 95\% confidence interval) for the smooth function model}
\label{tab: numerical_sfm}
\small
\begin{tabular}{c|l l l}
\toprule

 Nominal Level       & 80\%   & 90\%      & 95\% \\
 \midrule
 EL, $n=30$     &$79.44\%\pm 0.25\%$  &$89.45\%\pm0.19\%$ & $94.45\%\pm 0.14\%$  \\
TB, $n=30$     & $79.80\%\pm0.25\%$ &$89.75\%\pm 0.19\%$ & $94.65\%\pm 0.14\%$  \\
TB2, $n=30$ & $82.54\%\pm 0.24\%$ &$91.72\%\pm 0.17\%$ & $95.88\%\pm 0.12\%$\\
 EL,$n=50$     &$79.92\%\pm 0.25\%$ &$89.78\%\pm 0.19\%$ &$94.89\%\pm 0.14\%$ \\
TB, $n=50$   &$80.17\%\pm 0.25\%$ & $89.94\%\pm 0.19\%$ &$95.01\%\pm 0.13\%$ \\
TB2, $n=50$   &$81.76\%\pm 0.24\%$ &$91.19\%\pm 0.18\%$  &$95.83\%\pm 0.12\%$    \\

\bottomrule
\end{tabular}
\end{table}

We can see that the estimated coverage probabilities of TB2 are significantly higher than the nominal level, while the estimated coverage probabilities of TB are close to the nominal level. For example, when $n=30$ and the nominal level is 80\%, the estimated coverage probability of TB2 is more than 2\% above the nominal level, while the estimated coverage probability of TB is only 0.2\% below the nominal level.

\end{document}